\documentclass[11pt,leqno,a4paper]{amsart}
\usepackage{amssymb,amscd,amsmath}
\usepackage{pspicture}
\usepackage{graphicx}

\usepackage{mathptmx}

\usepackage[matrix,arrow,curve,frame]{xy}    

\xymatrixcolsep{1.9pc}                          
\xymatrixrowsep{1.9pc}
\newdir{ >}{{}*!/-5pt/\dir{>}}                  

 \calclayout

\makeatletter

\renewcommand{\phi}{\varphi}
\renewcommand{\leq}{\leqslant}
\renewcommand{\geq}{\geqslant}
\renewcommand{\epsilon}{\varepsilon}

\renewcommand{\kappa}{\varkappa}

 \DeclareMathOperator{\ind}{ind}
\DeclareMathOperator{\cyl}{cyl} 
\DeclareMathOperator{\spl}{spl} \DeclareMathOperator{\sw}{\sf sw}

\DeclareMathOperator{\sd}{sd} 
 \DeclareMathOperator{\Map}{Map}

\DeclareMathOperator{\surj}{surj} 
\DeclareMathOperator{\Hom}{Hom} 
 
 \DeclareMathOperator{\id}{id}
\DeclareMathOperator{\SimpAlg}{SimAlg}
 \DeclareMathOperator{\Mor}{Mor}
\DeclareMathOperator{\Alg}{Alg} \DeclareMathOperator{\colim}{colim}
\DeclareMathOperator{\Ho}{Ho}

 \DeclareMathOperator{\kr}{Ker}

 \DeclareMathOperator{\Ar}{Ar}
\DeclareMathOperator{\lm}{lim} 
 \DeclareMathOperator{\Mod}{Mod}
 \DeclareMathOperator{\Ob}{Ob}
\DeclareMathOperator{\CAlg}{CAlg}

\newcommand{\sn}{\par\smallskip\noindent}

\newcommand {\lp}{\colim}

\newcommand{\lo}{\lm}

\newcommand{\lra}[1]{\bl{#1}\longrightarrow\relax}
\newcommand{\bl}[1]{\buildrel #1\over}
\newcommand{\cc}{\mathcal}
\newcommand{\bb}{\mathbb}
\newcommand{\ps}{\oplus}
\newcommand{\ff}{\mathfrak}

\newcommand{\op}{{\textrm{\rm op}}}

\newcommand{\wt}{\widetilde}

\newcommand{\sdi}{\sd^\bullet}

\newcommand{\ahaw}{{\Alg_{k}}}
\newcommand{\sahaw}{\SimpAlg_k}
\newcommand{\cahaw}{{\CAlg}_k}

\newcommand{\Spt}{Sp(\Re)}

\newcommand{\inda}{{\Alg_k^{\ind}}}

\def\Sz{\mathbb{S}}

\newtheorem{thm}{Theorem}[section]
\newtheorem{prop}[thm]{Proposition}
\newtheorem{cor}[thm]{Corollary}
\newtheorem{lem}[thm]{Lemma}
\newtheorem*{excisa}{Excision Theorem A}
\newtheorem*{excisb}{Excision Theorem B}
\newtheorem*{comparisa}{Comparison Theorem A}
\newtheorem*{comparisb}{Comparison Theorem B}

\newtheorem*{rem}{Remark}
\newtheorem*{thmm}{Theorem}
\newtheorem*{exs}{Examples}
\newtheorem*{example}{Example}
\newtheorem*{defs}{Definition}
\newtheorem*{hauptlemma}{Hauptlemma}
\newtheorem*{hauptsub}{Hauptsublemma}

\makeatother

\begin{document}

\footskip30pt


\title{Algebraic Kasparov K-theory. I}
\author{Grigory Garkusha}

\thanks{Supported by EPSRC grant EP/H021566/1.}
\address{Department of Mathematics, Swansea University, Singleton Park, Swansea SA2 8PP, UK}
\email{G.Garkusha@swansea.ac.uk}
\urladdr{http://math.swansea.ac.uk/staff/gg/}



\begin{abstract}
This paper is to construct unstable, Morita stable and stable
bivariant algebraic Kasparov $K$-theory spectra of $k$-algebras.
These are shown to be homotopy invariant, excisive in each variable
$K$-theories. We prove that the spectra represent universal
unstable, Morita stable and stable bivariant homology theories
respectively.
\end{abstract}
\maketitle

\thispagestyle{empty} \pagestyle{plain}

\newdir{ >}{{}*!/-6pt/@{>}} 


\section{Introduction}

$K$-theory was originally discovered by Grothendieck in the late 50-s.
Thanks to works by Atiyah, Hirzebruch, Adams $K$-theory
was firmly entrenched in topology in the 60-s. Along with
topological $K$-theory mathematicians developed algebraic
$K$-theory. After Atiyah-Singer's index theorem for elliptic
operators $K$-theory penetrated into analysis and gave rise to
operator $K$-theory.

The development of operator $K$-theory in the 70-s took place in a
close contact with the theory of extensions of $C^*$-algebras and
prompted the creation of a new technical apparatus, the Kasparov
$K$-theory~\cite{Kas}. The Kasparov bifunctor $KK_*(A,B)$ combines
Grothendieck's $K$-theory $K_*(B)$ and its dual (contravariant)
theory $K^*(A)$. The existence of the product $KK_*(A,D)\otimes
KK_*(D,B)\to KK_*(A,B)$ makes the bifunctor into a very strong and
flexible tool.

One way of constructing an algebraic counterpart of the bifunctor
$KK_*(A,B)$ with a similar biproduct and similar universal
properties is to define a triangulated category whose objects are
algebras. In~\cite{Gar} the author constructed various bivariant
$K$-theories of algebras, but he did not study their universal
properties. Motivated by ideas and work of J. Cuntz on bivariant
$K$-theory of locally convex algebras~\cite{Cu2,Cu,Cu1}, {\it
universal\/} algebraic bivariant $K$-theories were constructed by
Corti\~nas--Thom in~\cite{CT}.

Developing ideas of~\cite{Gar} further, the author introduces and
studies in~\cite{Gar1} universal bivariant homology theories of
algebras associated with various classes $\ff F$ of fibrations on an
``admissible category of $k$-algebras" $\Re$. The methods used by
the author to construct $D(\Re,\ff F)$ are very different from those
used by Corti\~nas--Thom~\cite{CT} to construct $kk$. However they
coincide in the appropriate stable case. In a certain sense
\cite{Gar1} uses the same approach as in constructing $E$-theory of
$C^*$-algebras~\cite{Hig}. We start with a datum of an admissible
category of algebras $\Re$ and a class $\ff F$ of fibrations on it
and then construct a {\it universal\/} algebraic bivariant
$K$-theory $j:\Re\to D(\Re,\ff F)$ out of the datum $(\Re,\ff F)$ by
inverting certain arrows which we call weak equivalences. The
category $D(\Re,\ff F)$ is naturally triangulated. The most
important cases in practice are the class of $k$-linear split
surjections $\ff F=\ff F_{\spl}$ or the class $\ff F=\ff F_{\surj}$
of all surjective homomorphisms.

If $\ff F=\ff F_{\spl}$ (respectively $\ff F=\ff F_{\surj}$) then
$j:\Re\to D(\Re,\ff F)$ is called the unstable algebraic Kasparov
$K$-theory (respectively unstable algebraic $E$-theory) of $\Re$. It
should be emphasized that~\cite{Gar1} does not consider any
matrix-invariance in general. This is caused by the fact that many
interesting admissible categories of algebras deserving to be
considered separately like that of all commutative ones are not
closed under matrices.

If we want to have matrix invariance, then \cite{Gar1} introduces
matrices into the game and gets universal algebraic, excisive,
homotopy invariant {\it and\/} ``Morita invariant" (respectively
``$M_\infty$-invariant") $K$-theories $j:\Re\to D_{mor}(\Re,\ff F)$
(respectively $j:\Re\to D_{st}(\Re,\ff F)$). The triangulated
category $D_{mor}(\Re,\ff F)$ (respectively $D_{st}(\Re,\ff F)$) is
constructed out of $D(\Re,\ff F)$ just by ``inverting matrices"
$M_nA$, $n>0$, $A\in\Re$ (respectively by inverting the natural
arrows $A\to M_\infty A$ with $M_\infty A=\cup_nM_nA$). We call
$D_{mor}(\Re,\ff F_{\spl})$ and $D_{mor}(\Re,\ff F_{\surj})$
(respectively $D_{st}(\Re,\ff F_{\spl})$ and $D_{st}(\Re,\ff
F_{\surj})$) the Morita stable algebraic $KK$- and $E$-theories
(respectively the stable algebraic $KK$- and $E$-theories). A
version of the Corti\~nas--Thom theorem~\cite{CT} says that there is
a natural isomorphism of $\bb Z$-graded abelian groups
(see~\cite{Gar1})
   $$D_{st}(\Re,\ff F)_*(k,A)\cong KH_*(A),$$
where $KH_*(A)$ is the $\bb Z$-graded abelian group consisting of
the homotopy $K$-theory groups in the sense of Weibel~\cite{W1}.

One of the aims of this paper is to represent unstable, Morita
stable and stable algebraic Kasparov $K$-theories. Here we deal only
with the class of $k$-linear split surjections $\ff F=\ff F_{\spl}$.
We introduce the ``unstable, Morita stable and stable algebraic
Kasparov $K$-theory spectra" $\bb K^{\star}(A,B)$ of $k$-algebras
$A,B\in\Re$ where $\star\in\{unst,mor,st\}$ and $\Re$ is an
appropriate admissible category of algebras. It should be emphasized
that the spectra do not use any realizations of categories and are
defined by means of algebra homomorphisms only. This makes our
constructions rather combinatorial.

\begin{thmm}[Excision Theorem A for spectra]
Let $\star\in\{unst,mor,st\}$. The assignment
$B\mapsto\mathbb{K}^{\star}(A,B)$ determines a functor
   $$\mathbb{K}^{\star}(A,?):\Re\to(Spectra)$$
which is homotopy invariant and excisive in the sense that for every
$\ff F$-extension $F\to B\to C$ the sequence
   $$\mathbb{K}^{\star}(A,F)\to\mathbb{K}^{\star}(A,B)\to\mathbb{K}^{\star}(A,C)$$
is a homotopy fibration of spectra. In particular, there is a long
exact sequence of abelian groups
   $$\cdots\to\mathbb{K}_{i+1}^{\star}(A,C)\to\mathbb{K}_i^{\star}(A,F)\to\mathbb{K}_i^{\star}(A,B)
     \to\mathbb{K}_i^{\star}(A,C)\to\cdots$$
for any $i\in\bb Z$.
\end{thmm}

We also have the following

\begin{thmm}[Excision Theorem B for spectra]
Let $\star\in\{unst,mor,st\}$. The assignment
$B\mapsto\mathbb{K}^{\star}(B,D)$ determines a functor
   $$\mathbb{K}^{\star}(?,D):\Re^{\op}\to(Spectra),$$
which is excisive in the sense that for every $\ff F$-extension
$F\to B\to C$ the sequence
   $$\mathbb{K}^{\star}(C,D)\to\mathbb{K}^{\star}(B,D)\to\mathbb{K}^{\star}(F,D)$$
is a homotopy fibration of spectra. In particular, there is a long
exact sequence of abelian groups
   $$\cdots\to\mathbb{K}_{i+1}^{\star}(F,D)\to\mathbb{K}_i^{\star}(C,D)\to\mathbb{K}_i^{\star}(B,D)
     \to\mathbb{K}_i^{\star}(F,D)\to\cdots$$
for any $i\in\bb Z$.
\end{thmm}

The following result gives the desired representability.

\begin{thmm}[Comparison]
Let $\star\in\{unst,mor,st\}$. Then for any algebras $A,B\in\Re$
there is an isomorphism of $\bb Z$-graded abelian groups
   $$\bb K^{\star}_*(A,B)\cong D_{\star}(\Re,\ff F)_*(A,B)=\bigoplus_{n\in\bb Z}D_{\star}(\Re,\ff F)(A,\Omega^nB),$$
functorial both in $A$ and in $B$.
\end{thmm}

The Corti\~nas--Thom theorem~\cite[8.2.1]{CT} and the Comparison
Theorem above imply that the spectrum $\bb K^{st}(k,-)$ represents
$KH$. Namely, we have the following

\begin{thmm}
For any $A\in\Re$ there is a natural isomorphism of $\bb Z$-graded
abelian groups
   $$\bb K^{st}_*(k,A)\cong KH_*(A).$$
\end{thmm}

The preceding theorem is an analog of the same result of $KK$-theory
saying that there is a natural isomorphism $KK_*(\bb C,A)\cong
K_*(A)$ for any $C^*$-algebra $A$.

It is important to mention that another aim of the present paper
together with~\cite{G8} is to develop the theory of ``$K$-motives of
algebras", for which Excision Theorems A-B as well as the Comparison
Theorem are of great utility. This theory shares lots of common
properties with $K$-motives and bivariant $K$-theory of algebraic
varieties, introduced and studied by the author and Panin
in~\cite{GP,GP2} in order to solve some problems in~\cite{GP1} for
motivic spectral sequence. In fact, $K$-motives of algebras are a
kind of a connecting language between Kasparov $K$-theory of
$C^*$-algebras and $K$-motives of algebraic varieties (hence the
choice of the title of this paper and that of~\cite{G8}).

Throughout the paper $k$ is a fixed commutative ring with unit and
$\ahaw$  is the category of non-unital $k$-algebras and non-unital
$k$-homomorphisms.

\subsubsection*{Organization of the paper} In Section~\ref{prel}
we fix some notation and terminology. We study simplicial algebras
and simplicial sets of algebra homomorphisms associated with
simplicial algebras there. In Section~\ref{exten} we discuss
extensions of algebras and classifying maps. It is thanks to
simplicial algebras and some elementary facts of their extensions
that Excision Theorem~A is possible to prove. Then comes
Section~\ref{excisionsec} in which Excision Theorem~A is proved. We
also formulate Excision Theorem~B in this section but its proof
requires an additional material. The spectrum $\bb
K^{unst}(A,B)$ is introduced and studied in
Section~\ref{thespectrum}. In Section~\ref{homotopy} we present
necessary facts about model categories and Bousfield localization.
This material is needed to prove Excision Theorem~B. In
Section~\ref{compa} we study relations between simplicial and
polynomial homotopies. As an application Comparison Theorem~A is
proved in the section. Comparison Theorem~B is proved in
Section~\ref{compb}. It says that the Hom-sets of $D(\Re,\ff F)$ are
represented by stable homotopy groups of spectra $\bb
K^{unst}(A,B)$-s. The spectra $\bb K^{st},\bb K^{mor}$ are
introduced and studied in Section~\ref{moritast}. We also prove
there Comparison Theorems for $D^{st}(\Re,\ff F)$, $D^{mor}(\Re,\ff
F)$ and construct an isomorphism between stable groups $\bb K^{st}_*(k,A)$
of an algebra $A$ and its homotopy $K$-theory groups.

\subsubsection*{Acknowledgement} I would like to thank an anonymous referee for
helpful suggestions concerning the material of the paper.

\section{Preliminaries}\label{prel}

\subsection{Algebraic homotopy}
Following Gersten~\cite{G} a category of $k$-algebras without unit
$\Re$ is {\it admissible\/} if it is a full subcategory of $\ahaw$
and

\begin{enumerate}

\item $R$ in $\Re$, $I$ a (two-sided) ideal of $R$ then $I$ and
$R/I$ are in $\Re$;

\item if $R$ is in $\Re$, then so is $R[x]$, the polynomial algebra in
one variable;

\item given a cartesian square
   $$\xymatrix{D\ar[r]^\rho\ar[d]_\sigma &A\ar[d]^f\\
               B\ar[r]^g &C}$$
in $\ahaw$ with $A,B,C$ in $\Re$, then $D$ is in $\Re$.

\end{enumerate}

One may abbreviate 1, 2, and 3 by saying that $\Re$ is closed under
operations of taking ideals, homomorphic images, polynomial
extensions in a finite number of variables, and fibre products. For
instance, the {\it category of commutative $k$-algebras\/} $\CAlg_k$ is
admissible.

Observe that every $k$-module $M$ can be regarded as a non-unital
$k$-algebra with trivial multiplication: $m_1\cdot m_2=0$ for all
$m_1,m_2\in M$. Then $\Mod k$ is an admissible category of
commutative $k$-algebras.

If $R$ is an algebra then the polynomial algebra $R[x]$ admits two
homomorphisms onto $R$
   $$\xymatrix{R[x]\ar@<2.5pt>[r]^{\partial_x^0}\ar@<-2.5pt>[r]_{\partial_x^1}&R}$$
where
   $\partial_x^i|_R=1_R,\ \ \ \partial_x^i(x)=i,\ \ \ i=0,1.$
Of course, $\partial_x^1(x)=1$ has to be understood in the sense
that $\Sigma r_nx^n\mapsto\Sigma r_n$.

\begin{defs}{\rm
Two homomorphisms $f_0,f_1:S\to R$ are {\it elementary homo\-topic},
written $f_0\sim f_1$, if there exists a homomorphism
   $$f:S\to R[x]$$
such that $\partial^0_xf=f_0$ and $\partial^1_xf=f_1$. A map $f:S\to
R$ is called an {\it elementary homotopy equivalence\/} if there is
a map $g:R\to S$ such that $fg$ and $gf$ are elementary homotopic to
$\id_R$ and $\id_S$ respectively.

}\end{defs}

For example, let $A$ be a $\bb N$-graded algebra, then the inclusion
$A_0\to A$ is an elementary homotopy equivalence. The homotopy
inverse is given by the projection $A\to A_0$. Indeed, the map $A\to
A[x]$ sending a homogeneous element $a_n\in A_n$ to $a_nt^n$ is a
homotopy between the composite $A\to A_0\to A$ and the identity
$\id_A$.

The relation ``elementary homotopic'' is reflexive and
symmetric~\cite[p.~62]{G}. One may take the transitive closure of
this relation to get an equivalence relation (denoted by the symbol
``$\simeq$''). The set of equivalence classes of morphisms $R\to S$
is written $[R,S]$. This equivalence relation will also be called
{\it polynomial or algebraic homotopy}.

\begin{lem}[Gersten \cite{G1}]
Given morphisms in $\ahaw$
   $$\xymatrix{R\ar[r]^f &S\ar@/^/[r]^g \ar@/_/[r]_{g'} &T\ar[r]^h &U}$$
such that $g\simeq g'$, then $gf\simeq g'f$ and $hg\simeq hg'$.
\end{lem}

Thus homotopy behaves well with respect to composition and we have
category $Hotalg$, the {\it homotopy category of $k$-algebras},
whose objects are $k$-algebras and such that $Hotalg(R,S)=[R,S]$.
The homotopy category of an admissible category of algebras $\Re$
will be denoted by $\cc H(\Re)$. Call a homomorphism $s:A\to B$ an
{\it $I$-weak equivalence\/} if its image in $\cc H(\Re)$ is an
isomorphism. Observe that $I$-weak equivalences are those
homomorphisms which have homotopy inverses.

The diagram in $\ahaw$
   $$A\bl{f}\to B\bl{g}\to C$$
is a short exact sequence if $f$ is injective ($\equiv \kr f=0$),
$g$ is surjective, and the image of $f$ is equal to the kernel of
$g$. Thus $f$ is a monomorphism and $f=\ker g$.

\begin{defs}{\rm
An algebra $R$ is {\it contractible\/} if $0\sim 1$; that is, if
there is a homomorphism $f:R\to R[x]$ such that $\partial^0_xf=0$
and $\partial^1_xf=1_R$.

}\end{defs}

For example, every square zero algebra $A\in\ahaw$ is contractible
by means of the homotopy $A\to A[x]$, $a\in A\mapsto ax\in A[x]$. In
other words, every $k$-module, regarded as a $k$-algebra with trivial
multiplication, is contractible.

Following Karoubi and Villamayor~\cite{KV} we define $ER$, the {\it
path algebra\/} on $R$, as the kernel of $\partial_x^0:R[x]\to R$,
so $ER\to R[x]\bl{\partial_x^0}\to R$ is a short exact sequence in
$\ahaw$. Also $\partial_x^1:R[x]\to R$ induces a surjection
   $$\partial_x^1:ER\to R$$
and we define the {\it loop algebra\/} $\Omega R$ of $R$ to be its
kernel, so we have a short exact sequence in $\ahaw$
   $$\Omega R\to ER\bl{\partial_x^1}\to R.$$
We call it the {\it loop extension} of $R$. Clearly, $\Omega R$ is
the intersection of the kernels of $\partial_x^0$ and
$\partial_x^1$. By~\cite[3.3]{G} $ER$ is contractible for any
algebra $R$.

\subsection{Simplicial algebras}

Let $Ord$ denote the category of finite nonempty ordered sets and
order-preserving maps, and for each $n\geq 0$ we introduce the
object $[n]=\{0<1<\cdots<n\}$ of $Ord$. We let
$\Delta^n=\Hom_{Ord}(-,[n])$, so that $|\Delta^n|$ is the standard
$n$-simplex. In what follows the category of non-unital simplicial
$k$-algebras will be denoted by $\sahaw$.

Given a simplicial set $X$ and a simplicial algebra $A_\bullet$, we
denote by $A_\bullet(X)$ the simplicial algebra
$\Map(X,A_\bullet):[n]\mapsto\Hom_{\bb
S}(X\times\Delta^n,A_\bullet)$. We note that all simplicial algebras
are fibrant simplicial sets. If $A_\bullet$ is contractible as a
simplicial set, then the fact that $\bb S$ is a simplicial category
and thus satisfies axiom M7 for simplicial model categories
(see~\cite[section~9.1.5]{Hir}) implies that $A_\bullet(X)$ is
contractible.

In what follows a unital simplicial $k$-algebra $A_\bullet$ is an
object of $\sahaw$ such that all structure maps are unital algebra
homomorphisms.

\begin{prop}\label{contract}
Suppose $A_\bullet$ is a unital simplicial $k$-algebra. Then the
following statements are equivalent:
\begin{enumerate}
\item $A_\bullet$ is contractible as a simplicial set;
\item $A_\bullet$ is connected;
\item there is an element $t\in A_1$ such that $\partial_0(t)=0$ and
$\partial_1(t)=1$.
\end{enumerate}
Furthermore, if one of the equivalent assumptions is satisfied then
every simplicial ideal $I_\bullet\subset A_\bullet$ is contractible.
\end{prop}

\begin{proof}
$(1)\Rightarrow(2),(2)\Rightarrow(3)$ are obvious.

$(3)\Rightarrow(1)$. One can construct a homotopy $f:\Delta^1\times
A_\bullet\to A_\bullet$ from 0 to 1 by defining, for $n\geq 0$,
the map $f_n:\Delta^1_n\times A_n\to A_n$ with the formula
$f_n(\alpha,a)=(\alpha^*(t))\cdot a$. The same contraction applies
to $I_\bullet$.
\end{proof}

The main example of a simplicial algebra we shall work with is
defined as
   $$A^{\Delta}:[n]\mapsto A^{{\Delta}^n}:= A[t_0,\dots, t_n]/\langle 1-\sum_i t_i\rangle\ \ \ (\cong
     A[t_1,\ldots,t_n]),$$
where $A\in\ahaw$.
Given a map $\alpha:[m]\to[n]$ in $Ord$, the map
$\alpha^*:A^{\Delta^n}\to A^{\Delta^m}$ is defined by
$\alpha^*(t_j)=\sum_{\alpha(i)=j}t_i$. Observe that $A^\Delta\cong
A\otimes k^\Delta$.

Note that the face maps $\partial_{0;1}:A^{\Delta^1}\to
A^{\Delta^0}$ are isomorphic to $\partial^{0;1}_t:A[t]\to A$ in the
sense that the diagram
   $$\xymatrix{A[t]\ar[r]^{\partial^{\epsilon}_t}\ar[d]_{t\mapsto t_0}&A\ar[d]\\
               A^{\Delta^1}\ar[r]^{\partial_{\epsilon}}&A^{\Delta^0}}$$
is commutative and the vertical maps are isomorpisms. Let $A^+:=A\ps
k$ as a group and
   $$(a,n)(b,m)=(ab+ma+nb,nm).$$
Then $A^+$ is a unital $k$-algebra containing $A$ as an ideal. The
simplicial algebra $(A^+)^{\Delta}$ has the element $t=t_0$ in
dimension 1, which satisfies $\partial_0(t)=0$ and
$\partial_1(t)=1$. Thus, $t$ is an edge which connects 1 to 0,
making $(A^+)^{\Delta}$ a unital connected simplicial algebra. By
Proposition~\ref{contract} $(A^+)^{\Delta}$ is contractible as a
simplicial set. It follows that the same is true for $A^\Delta$.

There is a mapping space functor $\Hom^{\bullet}_{\ahaw}:
(\ahaw)^{\op}\times\ahaw\to\Sz$, given by
   $$(A,B) \mapsto ([n]\mapsto\Hom_{\ahaw}(A,B^{\Delta^n})).$$
For every $A\in\ahaw$, the functor $\Hom^{\bullet}_{\ahaw}(?,A):
(\ahaw)^{\op}\to\Sz$ has a left adjoint $A^?:\Sz\to(\ahaw)^{\op}$.
If $X\in\Sz$,
    $$A^{X} = \lim_{\Delta^n \to X} A^{\Delta^n}.$$
Observe that
   $$A^X=\Hom_\Sz(X,A^\Delta).$$
We have
    $$\Hom_{\ahaw}(A,B^X)=\Hom_\Sz(X,\Hom^\bullet_\ahaw(A,B)).$$

As above, for any simplicial algebra $A_\bullet$ the functor
$\Hom_{\ahaw}(?,A_\bullet): (\ahaw)^{\op}\to\Sz$ has a left adjoint
$A_\bullet\langle?\rangle=\Hom_\Sz(?,A_\bullet):\Sz\to(\ahaw)^{\op}$.
We have
    $$\Hom_{\ahaw}(B,A_\bullet\langle X\rangle)=\Hom_\Sz(X,\Hom_{\ahaw}(B,A_\bullet)).$$
Note that if $A_\bullet=A^\Delta$ then $A^\Delta\langle
X\rangle=A^X$.

Let $\Sz_\bullet$ be the category of pointed simplicial sets. For
$(K,\star)\in\Sz_\bullet$, put
$$
A_\bullet\langle K,\star\rangle:=\Hom_{\Sz_\bullet}((K,\star),A_\bullet)
=\ker(\Hom_\Sz(K,A_\bullet)\to \Hom_\Sz (\star,A_\bullet))
=\ker(A_\bullet\langle K\rangle\to A_\bullet).
$$

\begin{prop}[Corti\~{n}as--Thom \cite{CT}]\label{power=tensor}
Let $K$ be a finite simplicial set, $\star$ a vertex of $K$, and $A$
a $k$-algebra. Then $k^K$ and $k^{(K,\star)}$ are free $k$-modules,
and there are natural isomorphisms
\[
A\otimes_k k^K\bl\cong\to A^K\qquad A\otimes_k
k^{(K,\star)}\bl\cong\to A^{(K,\star)}.
\]
\end{prop}

\begin{proof}
This is a consequence of~\cite[3.1.3]{CT}:
   $A\otimes_k k^K\cong A\otimes_k (k\otimes_{\bb Z}\bb Z^K)\cong A\otimes_{\bb Z}\bb Z^K\cong A^K.$
\end{proof}

\subsection{Subdivision}

To give an explicit fibrant replacement of the
simplicial set $\Hom_\ahaw(A,B_\bullet)$ with $B_\bullet$ a
simplicial algebra, we should first define ind-algebras. In this
paragraph we shall adhere to~\cite{CT}.

If $\cc C$ is a category, we write $\ind-\cc C$ for the category of
ind-objects of $\cc C$. It has as objects the directed diagrams in
$\cc C$. An object in $\ind-\cc C$ is described by a filtering
partially ordered set $(I,\leq)$ and a functor $X:I\to\cc C$. The
set of homomorphisms from $(X,I)$ to $(Y,J)$ is
   $$\lo_{i\in I}\lp_{j \in J}\Hom_{\cc C}(X_i,Y_j).$$
We shall identify objects of $\cc C$ with constant $\ind$-objects,
so that we shall view $\cc C$ as a subcategory of $\ind-\cc C$. The
category of ind-algebras over $k$ will be denoted by $\inda$.

If $A=(A,I),B=(B,J)\in\inda$ we put
   $$[A,B]=\lo_i\lp_j\Hom_{\cc H(\ahaw)}(A_i,B_j).$$
Note that there is a natural map $\Hom_{\inda}(A,B)\to [A,B]$. Two
homomorphisms $f,g:A\to B$ in $\inda$ are called {\rm homotopic\/}
if they have the same image in $[A,B]$.

Write $\sd:\Sz \to \Sz$ for the simplicial subdivision functor (see
\cite[Ch. III.\S4]{GJ}). It comes with a natural transformation $h:
\sd\to\id_{\Sz}$, which is usually called the last vertex map. We
have an inverse system
   $$\xymatrix{{\sd}^\bullet
     K:\sd^0K&&{\sd}^1K\ar[ll]_(.4){h_K}&&{\sd}^2K\ar[ll]_{h_{\sd
     K}}&&{\sd}^3K\ar[ll]_{h_{\sd^2 K}}&&{\dots}. \ar[ll]_ {h_{\sd^3 K}}}$$
We may regard $\sdi K$ as a pro-simplicial set, that is, as an
$\ind$-object in $\Sz^{\op}$. The $\ind$-extension of the functor
$A_\bullet\langle ?\rangle:\Sz^{\op}\to\ahaw$ with $A_\bullet$ a
simplicial algebra maps $\sdi K$ to
   $$A_\bullet\langle\sdi K\rangle=\{A_\bullet\langle\sd^n K\rangle\mid n\in\bb Z_{\geq 0}\}.$$
If we fix $K$, we obtain a functor $(?)\langle\sdi
K\rangle:\sahaw\to \inda$, which extends to $(?)\langle\sdi
K\rangle:\sahaw^{\ind}\to\inda$ in the usual manner explained above.
In the special case when $A_\bullet=A^\Delta, A\in\ahaw$, the
ind-algebra $A^\Delta\langle\sdi K\rangle$ is denoted by $A^{\sdi
K}$.

Let $A\in\ahaw,B_\bullet\in\sahaw^{\ind}$. The space of the
preceding paragraph extends to $\ind$-algebras by
   $$\Hom_{\inda}(A,B_\bullet):=([n]\mapsto\Hom_{\inda}(A,B_n)).$$
Let $K$ be a finite simplicial set and $B_\bullet\in\sahaw^{\ind}$.
Denote by $\bb B_\bullet(K)$ the simplicial ind-algebra
$([n,\ell]\mapsto
B_\bullet\langle\sd^n(K\times\Delta^\ell)\rangle)=\{(Ex^n
B_\bullet)^K\mid n\geq 0\}$. If $K=*$ we write $\bb B_\bullet$ for
$\bb B_\bullet(*)$.

Similar to \cite[3.2.2]{CT} one can prove that there is a natural
isomorphism
   $$\Hom_\Sz(K,\Hom_{\inda}(A,\bb B_\bullet))\cong\Hom_{\inda}(A,B_\bullet\langle\sdi K\rangle),$$
where $A\in\ahaw,B_\bullet\in\sahaw^{\ind}$ and $K$ is a finite
simplicial set. Notice that the formula holds even at the ind-level,
before taking colimit.

\begin{thm}[Corti\~{n}as--Thom]\label{exinfi}
Let $A\in\ahaw$, $B_\bullet\in\sahaw^{\ind}$. Then
   $$\Hom_{\inda}(A,\bb B_\bullet)=Ex^\infty\Hom_{\inda}(A,B_\bullet).$$
In particular, $\Hom_{\inda}(A,\bb B_\bullet)$ is fibrant.
\end{thm}

\begin{proof}
The proof is like that of~\cite[3.2.3]{CT}.
\end{proof}

\begin{prop}\label{exinficor}
Let $K$ be a finite simplicial set, $A\in\ahaw$ and
$(B_\bullet,J)\in\sahaw^{\ind}$. Then
   $$\Hom_{\inda}(A,\bb B_\bullet(K))=(Ex^\infty\Hom_{\inda}(A,B_\bullet))^K.$$
In particular, the left hand side is fibrant.
\end{prop}

\begin{proof}
The proof is like that of~\cite[3.2.3]{CT}.
 \begin{align*} \Hom_\bb S(\Delta^\ell,&\Hom_{\inda}(A,\bb B_\bullet(K)))
=\colim_{(j,n) \in J\times\bb Z_{\geq 0}}\Hom_{\ahaw}(A,B_{\bullet,j}\langle\sd^n(K\times\Delta^\ell)\rangle) \\
=&  \colim_{n\in\bb Z_{\geq 0}}\colim_{j \in J}\Hom_{\bb S}(\sd^n(K\times\Delta^\ell),\Hom_{\ahaw}(A,B_{\bullet,j})) \\
=&  \colim_{n\in \bb Z_{\geq 0}}\Hom_{\bb S}(\sd^n(K\times\Delta^\ell),\colim_{j\in J}\Hom_{\ahaw}(A,B_{\bullet,j})) \\
=&  \colim_{n\in\bb Z_{\geq 0}}\Hom_{\bb S}(K\times\Delta^\ell,Ex^n\colim_{j\in J}\Hom_{\ahaw}(A,B_{\bullet,j})) \\
=& \Hom_{\bb
S}(K\times\Delta^\ell,Ex^{\infty}\Hom_{\inda}(A,B_{\bullet})).
\end{align*}
\end{proof}

\begin{cor}\label{corloop1}
Let $A\in\ahaw$ and let $K,L$ be finite simplicial sets, then
   $$\Hom_{\inda}(A,\bb B_\bullet(K))^L=\Hom_{\inda}(A,\bb B_\bullet(K\times L)).$$
\end{cor}

Denote by $\bb B_\bullet(I)$ and $\bb B_\bullet(\Omega)$ the
simplicial ind-algebras $\bb B_\bullet(\Delta^1)$ and $\ker(\bb
B_\bullet(I)\xrightarrow{(d^0,d^1)}\bb B_\bullet)$ respectively. We
define inductively $\bb B_\bullet(I^n):=(\bb B_\bullet(I^{n-1}))^I$,
$\bb B_\bullet(\Omega^n):=(\bb B_\bullet(\Omega^{n-1}))(\Omega)$.
Clearly, $\bb B_\bullet(I^n)=\bb B_\bullet(\Delta^1\times\bl
n\cdots\times\Delta^1)$ and $\bb B_\bullet(\Omega^n)$ is a
simplicial ideal of $\bb B_\bullet(I^n)$ that consists in each
degree $\ell$ of simplicial maps $F:\Delta^1\times\bl
n\cdots\times\Delta^1\times\Delta^\ell\to\bb B_\bullet$ such that
$F|_{\partial(\Delta^1\times\bl
n\cdots\times\Delta^1)\times\Delta^\ell}=0$.

\begin{cor}\label{corloop2}
Let $A\in\ahaw$, then
   $$\Hom_{\inda}(A,\bb B_\bullet(\Omega^n))=\Omega^n(\Hom_{\inda}(A,\bb B_\bullet)),$$
where $\Hom_{\inda}(A,\bb B_\bullet)$ is based at zero.
\end{cor}

\begin{proof}
This is a consequence of Theorem~\ref{exinfi},
Proposition~\ref{exinficor} and Corollary~\ref{corloop1}.
\end{proof}

\section{Extensions and Classifying Maps}\label{exten}

Throughout, we assume fixed an underlying category $\cc U$, which is
a full subcategory of $\Mod k$. In what follows we denote by $\ff F$
the class of $k$-split surjective algebra homomorphisms. We shall
also refer to $\ff F$ as {\it fibrations}.

\begin{defs}{\rm
An admissible category of algebras $\Re$ is said to be {\it
$T$-closed\/} if we have a faithful forgetful functor $F:\Re\to\cc
U$ (i.e. $F$ is the restriction to $\Re$ of the forgetful functor
from $\ahaw$ to $k$-modules) and a left adjoint functor $\wt{T}:\cc
U\to\Re$. Notice that the counit map $\eta_A: T(A):=\wt TF(A)\to A$,
$A\in\Re$, is a fibration.

We denote by $\Re^{\ind}$ the category of ind-objects for an
admissible category of algebras $\Re$. If $\Re$ is $T$-closed then
$TA$, $A\in\Re^{\ind}$, is defined in a natural way.

}\end{defs}

Throughout this section $\Re$ is supposed to be $T$-closed.

\begin{lem}\label{ta}
For every $A\in\Re$ the algebra $TA$ is contractible, i.e. there is
a polynomial contraction $\tau:TA\to TA[x]$ such that
$\partial_x^0\tau=0,\partial_x^1\tau=1$. Moreover, the contraction
is functorial in $A$.
\end{lem}

\begin{proof}
Consider a map $u:FTA\to FTA[x]$ sending an element $b\in FTA$ to
$bx\in FTA[x]$. If $X\in\Ob\cc U$ then we denote the unit map $X\to
F\wt TX$ by $i_X$. The desired contraction $\tau$ is uniquely
determined by the map $u\circ i_{FA}:FA\to FTA[x]$. By using
elementary properties of adjoint functors, one can show that
$\partial_x^0\tau=0,\partial_x^1\tau=1$.
\end{proof}

\begin{exs}{\rm
(1) Let $\Re=\ahaw$. Given an algebra $A$, consider the algebraic
tensor algebra
   $$TA=A\oplus A\otimes A\oplus A^{\otimes^3}\oplus\cdots$$
with the usual product given by concatenation of tensors. In Cuntz's
treatment of bivariant $K$-theory~\cite{Cu2,Cu,Cu1}, tensor algebras
play a prominent role.

There is a canonical $k$-linear map $A\to TA$ mapping $A$ into the
first direct summand. Every $k$-linear map $s:A\to B$ into an
algebra $B$ induces a homomorphism $\gamma_s:TA\to B$ defined by
   $$\gamma_s(x_1\otimes\cdots\otimes x_n)=s(x_1)s(x_2)\cdots s(x_n).$$
$\Re$ is plainly $T$-closed.

(2) If $\Re=\cahaw$ then
   $$T(A)=Sym(A)=\oplus_{n\geq 1}S^nA,\quad S^nA=A^{\otimes n}/\langle a_1\otimes\cdots\otimes a_n
     -a_{\sigma(1)}\otimes\cdots\otimes a_{\sigma(n)}\rangle,\quad\sigma\in\Sigma_n,$$
the symmetric algebra of $A$, and $\Re$ is $T$-closed.

}\end{exs}

We shall say that a sequence in $\Re$
   $$0\to C\to B\to A\to 0$$
is an {\it $F$-split extension\/} or just an {\it ($\ff
F$-)extension\/} if it is split exact in the category of
$k$-modules. We have the natural extension of algebras
   $$0\lra{}JA\lra{\iota_A}TA\lra{\eta_A}A\lra{}0.$$
Here $JA$ is defined as $\kr\eta_A$. Clearly, $JA$ is functorial in
$A$. This extension is universal in the sense that given any
extension
   $$0\to C\to B\bl\alpha\to A\to 0$$
with $\alpha$ in $\ff F$, there exists a commutative diagram of
extensions as follows.
\[ \xymatrix{
C \ar[r] & B \ar[r]^\alpha & A \\
J(A) \ar[u]^\xi \ar[r]^{\iota_A} & T(A) \ar[u] \ar[r]^{\eta_A} & A
\ar[u]^{\id_A} }
\]
Furthermore, $\xi$ is unique up to elementary
homotopy~\cite[4.4.1]{CT} in the sense that if $\beta,\gamma:A\to B$
are two splittings to $\alpha$ then $\xi_\beta$ corresponding to
$\beta$ is elementary homotopic to $\xi_\gamma$ corresponding to
$\gamma$. Because of this, we shall abuse notation and refer to any
such morphism $\xi$ as {\em the\/} classifying map of the extension
whenever we work with maps up to homotopy. The elementary homotopy
$H(\beta,\gamma):J(A)\to C[x]$ is explicitly constructed as follows.
Let $\wt\alpha:B[x]\to A[x],\sum b_ix^i\mapsto\alpha(b_i)x^i$, be
the natural lift of $\alpha$. Consider a $k$-linear map
   $$u:A\to B[x],\quad a\mapsto\beta(a)(1-x)+\gamma(a)x.$$
It is extended to an algebra homomorphism $\bar u:T(A)\to B[x]$. One
has a commutative diagram of algebras
\[ \xymatrix{
C[x] \ar[r] & B[x] \ar[r]^{\wt\alpha} & A[x] \\
J(A) \ar[u]^{H(\beta,\gamma)} \ar[r]^{\iota_A} & T(A) \ar[u]^{\bar
u} \ar[r]^{\eta_A} & A \ar[u]^{\iota} },
\]
where $\iota$ is the natural inclusion. It follows that
$H(\beta,\gamma)$ is an elementary homotopy between $\xi_\beta$ and
$\xi_\gamma$.

If we want to specify a particular choice of $\xi$ corresponding to
a splitting $\beta$ then we sometimes denote $\xi$ by $\xi_\beta$
indicating the splitting.

Also, if
\[ \xymatrix{
C  \ar[r] \ar[d]^f & B \ar[r]^\alpha \ar[d]^h & A \ar[d]^g \\
C' \ar[r]        & B'\ar[r]^{\alpha'}   &A' } \] is a commutative
diagram of extensions, then there is a diagram
\[ \xymatrix{ J(A) \ar[d]_{J(g)} \ar[r]^{\xi_{\beta}} & C \ar[d]^f \\
J(A') \ar[r]^{\xi_{\beta'}} & C'}\] of classifying maps, which is
commutative up to elementary homotopy (see~\cite[4.4.2]{CT}).

The elementary homotopy can be constructed as follows. Let
$\wt\alpha':B'[x]\to A'[x],\sum b'_ix^i\mapsto\alpha'(b'_i)x^i$, be
the natural lift of $\alpha'$. Consider a $k$-linear map
   $$v:A\to B'[x],\quad a\mapsto h\beta(a)(1-x)+\beta'g(a)x.$$
It is extended to a ring homomorphism $\bar v:T(A)\to B[x]$. One has
a commutative diagram of algebras
\[ \xymatrix{
C'[x] \ar[r] & B'[x] \ar[r]^{\wt\alpha'} & A'[x] \\
J(A) \ar[u]^{G(\beta,\beta')} \ar[r] & T(A) \ar[u]^{\bar v}
\ar[r]^{\eta_A} & A \ar[u]^{\iota'g} },
\]
where $\iota':A'\to A'[x]$ is the natural inclusion. It follows that
$G(\beta,\beta')$ is an elementary homotopy between $f\xi_\beta$ and
$\xi_{\beta'}J(g)$.

Let $\cc C$ be a small category and let $\Re^{\cc C}$ (respectively
$\cc U^{\cc C}$) denote the category of $\cc C$-diagrams in $\Re$
(respectively in $\cc U$). Then we can lift the functors
$F:\Re\to\cc U$ and $\wt T:\cc U\to\Re$ to $\cc C$-diagrams. We
shall denote the functors by the same letters. So we have a faithful
forgetful functor $F:\Re^{\cc C}\to\cc U^{\cc C}$ and a functor
$\wt{T}:\cc U^{\cc C}\to\Re^{\cc C}$, which is left adjoint to $F$.
The counit map $\eta_A: T(A):=\wt TF(A)\to A$, $A\in\Re^{\cc C}$, is
a levelwise fibration.

\begin{defs}{\rm
We shall say that a sequence of $\cc C$-diagrams in $\Re$
   $$0\to C\to B\bl\alpha\to A\to 0$$
is a {\it $F$-split extension\/} or just an {\it ($\ff
F$-)extension\/} if it is split exact in the abelian category $(\Mod
k)^{\cc C}$ of $\cc C$-diagrams of $k$-modules.

}\end{defs}

We have a natural extension of $\cc C$-diagrams in $\Re$
   $$0\lra{}JA\lra{\iota_A}TA\lra{\eta_A}A\lra{}0.$$
Here $JA$ is defined as $\kr\eta_A$. Clearly, $JA$ is functorial in
$A$.

\begin{lem}\label{classmnm}
Given any extension $0\to C\to B\to A\to 0$ of $\cc C$-diagrams in
$\Re$, there exists a commutative diagram of extensions as follows.
\[ \xymatrix{
C \ar[r] & B \ar[r]^\alpha & A \\
J(A) \ar[u]^\xi \ar[r]^{\iota_A} & T(A) \ar[u] \ar[r]^{\eta_A} & A
\ar[u]^{\id_A} }
\]
Furthermore, $\xi$ is unique up to a natural elementary homotopy
$H(\beta,\gamma):JA\to C[x]$, where $\beta,\gamma$ are two
splittings of $\alpha$.
\end{lem}

\begin{proof}
The proof is like that for algebras (see above).
\end{proof}

\begin{lem}\label{class}
Let
   $$ \xymatrix{C  \ar[r] \ar[d]^f & B \ar[r]^\alpha \ar[d]_h & A \ar[d]^g \\
                C' \ar[r]        & B'\ar[r]^{\alpha'} & A' }$$
be a commutative diagram of $F$-split extensions of $\cc C$-diagrams
with splittings $\beta:A\to B,\beta':A'\to B'$. Then there is a
diagram of classifying maps
   $$\xymatrix{ J(A) \ar[d]_{J(g)} \ar[r]^{\xi_\beta} & C \ar[d]^f \\
     J(A') \ar[r]^{\xi_{\beta'}} & C'}$$
which is commutative up to a natural elementary homotopy
$G(\beta,\beta'):JA\to C'[x]$.
\end{lem}

\begin{proof}
The proof is like that for algebras (see above).
\end{proof}

\begin{lem}\label{classqmn}
Let
   $$ \xymatrix{A  \ar[r] \ar[d]^f & B \ar[r]^u \ar[d]_h & C \ar[d]^g \\
                A' \ar[r]        & B'\ar[r]^{u'} & C' }$$
be a commutative diagram of $F$-split extensions of $\cc C$-diagrams
with splittings $(v,v'):(C,C')\to(B,B')$ being such that $(v,v')$ is
a splitting to $(u,u')$ in the category of arrows $\Ar(\cc U^{\cc
C})$, i.e. $hv=v'g$. Then the diagram of classifying maps
   $$\xymatrix{ J(C) \ar[d]_{J(g)} \ar[r]^{\xi_v} & A \ar[d]^f \\
     J(C') \ar[r]^{\xi_{v'}} & A'}$$
is commutative.
\end{lem}

\begin{proof}
If we regard $h$ and $g$ as $\{0\to 1\}\times\cc C$-diagrams and
$(u,u')$ as a map from $h$ to $g$, then the commutative diagram of
lemma is the classifying map corresponding to the splitting $(v,v')$
of $\{0\to 1\}\times\cc C$-diagrams.
\end{proof}

\section{The Excision Theorems}\label{excisionsec}

Throughout this section $\Re$ is assumed to be $T$-closed. Recall
that $k^\Delta$ is a contractible unital simplicial object in $\Re$
and $t:=t_0\in k^{\Delta^1}$ is a 1-simplex with
$\partial_0(t)=0,\partial_1(t)=1$. Given an algebra $B$, the
ind-algebra $\bb B^\Delta$ is defined as
   $$[m,\ell]\mapsto\Hom_{\bb S}(\sd^m\Delta^\ell,B^\Delta)=B^{\sd^m\Delta^\ell}.$$
If $B=k$ then $\bb B^\Delta$ will be denoted by $\Bbbk^\Delta$.
$B^\Delta$ can be regarded as a $k^\Delta$-module, i.e. there is a
simplicial map, induced by multiplication,
   $$B^\Delta\times k^\Delta\to B^\Delta.$$
Similarly, $\bb B^\Delta$ can be regarded as a
$\Bbbk^\Delta$-ind-module.

Given two algebras $A,B\in\Re$ and $n\geq 0$, consider the
simplicial set
   $$\Hom_{\inda}(J^nA,\bb B^\Delta(\Omega^n))\cong\Hom_{\inda}(J^nA,B\otimes_k\Bbbk^\Delta(\Omega^n)).$$
It follows from Proposition~\ref{exinficor} and
Corollary~\ref{corloop2} that it is fibrant. $\bb
B^\Delta(\Omega^n)$ is a simplicial ideal of the simplicial
ind-algebra
   $$\bb B^\Delta(I^n)=([m,\ell]\mapsto\Hom_{\bb S}(\sd^m(\Delta^1\times\bl
   n\cdots\times\Delta^1\times\Delta^\ell)\to B^\Delta)).$$
There is a commutative diagram of simplicial ind-algebras
   $$\xymatrix{P\bb B^\Delta(\Omega^n)\ar@{ >->}[r]\ar@{ >->}[d]&(\bb B^\Delta(\Omega^n))^I\ar@{ >->}[d]\ar@{->>}[r]^{d_0}&\bb B^\Delta(\Omega^n)\ar@{ >->}[d]\\
               P\bb B^\Delta(I^n)\ar@{ >->}[r]&\bb B^\Delta(I^{n+1})
               \ar@{->>}[r]^{d_0}&\bb B^\Delta(I^n)}$$
with vertical arrows inclusions and the right lower map $d_0$
applies to the last coordinate.

We claim that the natural simplicial map $d_1:P\bb
B^\Delta(\Omega^n)\to\bb B^\Delta(\Omega^n)$ has a natural
$k$-linear splitting. In fact, the splitting is induced by a natural
$k$-linear splitting $\upsilon$ for $d_1:P\bb B^\Delta(I^n)\to\bb
B^\Delta(I^n)$. Let $\mathbf{t}\in P\Bbbk^\Delta(I^n)_0$\label{elt}
stand for the composite map
   $$\sd^m(\Delta^1\times\bl{n+1}\cdots\times\Delta^1)\lra{pr}\sd^m\Delta^1\to\Delta^1\bl t\to k^\Delta,$$
where $pr$ is the projection onto the $(n+1)$th direct factor
$\Delta^1$. The element $\mathbf{t}$ can be regarded as a 1-simplex
of the unital ind-algebra $\Bbbk^\Delta(I^n)$ such that
$\partial_0(\mathbf{t})=0$ and $\partial_1(\mathbf{t})=1$. Let
$\imath:\bb B^\Delta(I^n)\to(\bb B^\Delta(I^n))^{\Delta^1}$ be the
natural inclusion. Multiplication with $\mathbf{t}$ determines a
$k$-linear map $\bb B^\Delta(I^{n+1})\lra{\mathbf{t}\cdot}P\bb
B^\Delta(I^n)$. Now the desired $k$-linear splitting is defined as
   $$\upsilon:=\mathbf{t}\cdot\imath.$$

Consider a sequence of simplicial sets 
   \begin{equation}\label{adj}
     \Hom_{\inda}(A,\bb B^\Delta)\xrightarrow{\varsigma}
     \Hom_{\inda}(JA,\bb B^\Delta(\Omega))\xrightarrow{\varsigma}\cdots\xrightarrow{\varsigma}
     \Hom_{\inda}(J^nA,\bb B^\Delta(\Omega^n))\xrightarrow{\varsigma}\cdots
   \end{equation}
\normalsize Each map $\varsigma$ is defined by means of the
classifying map $\xi_{\upsilon}$ corresponding to the $k$-linear
splitting $\upsilon$. More precisely, if we consider $\bb
B^\Delta(\Omega^n)$ as a $(\bb Z_{\geq 0}\times\Delta)$-diagram in
$\Re$, then there is a commutative diagram of extensions for $(\bb
Z_{\geq 0}\times\Delta)$-diagrams
   $$\xymatrix{J\bb B^\Delta(\Omega^n)\ar[d]_{\xi_\upsilon}\ar[r]&T\bb B^\Delta(\Omega^n)\ar[r]\ar[d]&\bb B^\Delta(\Omega^n)\ar@{=}[d]\\
               \bb B^\Delta(\Omega^{n+1})\ar[r]&P\bb B^\Delta(\Omega^n)\ar[r]^{d_1}&\bb B^\Delta(\Omega^n)}$$
For every element $f\in\Hom_{\inda}(J^nA,\bb B^\Delta(\Omega^n))$
one sets:
   $$\varsigma(f):=\xi_{\upsilon}\circ J(f)\in\Hom_{\inda}(J^{n+1}A,\bb B^\Delta(\Omega^{n+1})).$$

Now consider an $\ff F$-extension in $\Re$
   $$F\lra{i}B\lra{f}C.$$
For any $n\geq 0$ one constructs a cartesian square of simplicial
ind-algebras
   $$\xymatrix{P_f(\Omega^n)\ar[d]_{pr}\ar[r]^{pr}&P(\bb C^\Delta(\Omega^n))\ar[d]^{d_1}\\
               \bb B^\Delta(\Omega^n)\ar[r]^{f}&\bb C^\Delta(\Omega^n).}$$
We observe that the path space $P(P_f(\Omega^n))$ of $P_f(\Omega^n)$
is the fibre product of the diagram
   $$P\bb B^\Delta(\Omega^n)\xrightarrow{P(f)}P\bb C^\Delta(\Omega^n)\xleftarrow{Pd_1}P(P\bb C^\Delta(\Omega^n)).$$
Denote by ${\wt P}(P_f(\Omega^n))$ the fibre product of the diagram
   $$P\bb B^\Delta(\Omega^n)\xrightarrow{P(f)}P\bb C^\Delta(\Omega^n)\xleftarrow{d_{1}^{P\bb C^\Delta(\Omega^n)}}P(P\bb C^\Delta(\Omega^n)).$$
Given a simplicial set $X$, let
   $$\sw:X^{\Delta^1\times\Delta^1}\to X^{\Delta^1\times\Delta^1}$$
be the automorphism swapping the two coordinates of
$\Delta^1\times\Delta^1$. If $X=\bb C^\Delta(\Omega^n)$ then $\sw$
induces an automorphism
   $$\sw:P(P\bb C^\Delta(\Omega^n))\to P(P\bb C^\Delta(\Omega^n)),$$
denoted by the same letter. Notice that
   $$Pd_1=d_{1}\circ\sw.$$
Moreover, the commutative diagram
   $$\xymatrix{P\bb B^\Delta(\Omega^n)\ar@{=}[d]\ar[r]^{P(f)}&P\bb C^\Delta(\Omega^n)\ar@{=}[d]
               &P(P\bb C^\Delta(\Omega^n))\ar[l]_{Pd_1}\ar[d]^{\sw}\\
               P\bb B^\Delta(\Omega^n)\ar[r]^{P(f)}&P\bb C^\Delta(\Omega^n)&P(P\bb C^\Delta(\Omega^n))\ar[l]_{d_1}}$$
yields an isomorphism of simplicial sets
   $$P(P_f(\Omega^n))\cong\wt{P}(P_f(\Omega^n)).$$

The natural simplicial map
   $$\partial:=(d_1,Pd_1):{\wt P}(P_f(\Omega^n))\to P_f(\Omega^n)$$
has a natural $k$-linear splitting $\tau:P_f(\Omega^n)\to{\wt
P}(P_f(\Omega^n))$ defined as $\tau=(\upsilon,P\upsilon)$.

So one can define a sequence of simplicial sets
   $$\Hom_{\inda}(A,P_f)\xrightarrow{\vartheta}
     \Hom_{\inda}(JA,P_f(\Omega))\xrightarrow{\vartheta}\cdots$$
with each map $\vartheta$ defined by means of the classifying map
$\xi_{\tau}$ corresponding to the $k$-linear splitting $\tau$.

There is a natural map of simplicial ind-algebras for any $n\geq 0$
   $$\iota:\bb F^\Delta(\Omega^n)\to P_f(\Omega^n).$$

\begin{prop}\label{excis}
For any $n\geq 0$ there is a homomorphism of simplicial ind-algebras
$\alpha:J(P_f(\Omega^n))\to\bb F^\Delta(\Omega^{n+1})$ such that in
the diagram
   $$\xymatrix{J(\bb F^\Delta(\Omega^n))\ar[r]^{\xi_{\upsilon}}\ar[d]_{J(\iota)}
               &\bb F^\Delta(\Omega^{n+1})\ar[d]^\iota\\
               J(P_f(\Omega^n))\ar[r]^{\xi_\tau}\ar[ur]^\alpha
               &P_f(\Omega^{n+1})}$$
$\alpha J(\iota)=\xi_\upsilon$, $\xi_\tau
J(\iota)=\iota\xi_\upsilon$, and $\iota\alpha$ is elementary
homotopic to $\xi_\tau$.
\end{prop}

\begin{proof}
We want to construct a commutative diagram of extensions as follows.
   \begin{equation}\label{wmnc}
   \xymatrix{\bb F^\Delta(\Omega^{n+1})\ar[r]\ar[d]^{\id}
   &P(\bb F^\Delta(\Omega^n))\ar[r]^{d_1^F}\ar[d]_\chi &\bb F^\Delta(\Omega^n)\ar[d]^{\iota}\\
   \bb F^\Delta(\Omega^{n+1})\ar[r]\ar[d]^{\iota}&P(\bb B^\Delta(\Omega^n))\ar[r]^\pi\ar[d]_\theta&P_f(\Omega^n)\ar[d]^{\id}\\
   P_f(\Omega^{n+1})\ar[r]&{\wt P}P_f(\Omega^n)\ar[r]^{\partial}&P_f(\Omega^n)}
   \end{equation}
Here $\pi$ is a natural map induced by $(d_1:P(\bb
B^\Delta(\Omega^n))\to \bb B^\Delta(\Omega^n),P(f))$. A splitting
$\nu$ to $\pi$ is constructed as follows.

Let $g:C\to B,j:B\to F$ be $k$-linear splittings to $f:B\to C$ and
$i:F\to B$ respectively. So $fg=1_C$, $ji=1_F$ and $ij+gf=1_B$. Then
the simplicial map
   $$ij:\bb B^\Delta(\Omega^n)\to\bb B^\Delta(\Omega^n)$$
is $k$-linear. We define $\nu$ as the composite map
   $$\xymatrix{P_f(\Omega^n)\ar@{ >->}[d]\\
               \bb B^\Delta(\Omega^n)\times P(\bb C^\Delta(\Omega^n))\ar[rr]^{(\upsilon ij,P(g))}
               &&P(\bb B^\Delta(\Omega^n))\times P(\bb B^\Delta(\Omega^n))\ar[d]^+\\
               &&P(\bb B^\Delta(\Omega^n)).}$$

We have to define the map $\theta$. For this we construct a map of
simplicial sets
   $$\lambda:\Delta^1\times\Delta^1\to\Delta^1.$$
We regard the simplicial set $\Delta^1$ as the nerve of the category
$\{0\to 1\}$. Then $\lambda$ is obtained from the functor between
categories
   $$\{0\to 1\}\times\{0\to 1\}\to\{0\to 1\},\quad(0,1),(1,0),(1,1)\mapsto 1,(0,0)\mapsto 0.$$
The induced map $\lambda^*:\bb B^\Delta(\Omega^n)^{\Delta^1}\to\bb
B^\Delta(\Omega^n)^{\Delta^1\times\Delta^1}$ induces a map of path
spaces $\lambda^*:P\bb B^\Delta(\Omega^n)\to P(P\bb
B^\Delta(\Omega^n))$. The desired map $\theta$ is defined by the map
$(1_{P(\bb B^\Delta(\Omega^n))},f\lambda^*)$. Our commutative
diagram is constructed.

Consider the following diagrams of classifying maps
   $$\xymatrix{J(\bb F^\Delta(\Omega^n))\ar[d]^{J(\iota)}\ar[r]^{\xi_{\upsilon}}&\bb F^\Delta(\Omega^{n+1})\ar[d]^{\id}
               &&J(P_f(\Omega^n))\ar[r]^\alpha \ar[d]^{\id}&\bb F^\Delta(\Omega^{n+1})\ar[d]^\iota\\
               J(P_f(\Omega^n))\ar[r]^\alpha&\bb F^\Delta(\Omega^{n+1})
               && J(P_f(\Omega^n)) \ar[r]^{\xi_\tau} & P_f(\Omega^{n+1})}$$
Since $\chi\upsilon=\nu\iota$ then the left square is commutative by
Lemma~\ref{classqmn}, because $(d_1^F,\pi)$ yield a map of $\{0\to
1\}\times\cc C$-diagrams split by $(\upsilon,\nu)$. Also $\xi_\tau
J(\iota)=\iota\xi_\upsilon$, because $(d_1^F,\partial)$ yield a map
of $\{0\to 1\}\times\cc C$-diagrams split by $(\upsilon,\tau)$. The
right square is commutative up to elementary homotopy by
Lemma~\ref{class}.
\end{proof}

\begin{defs}{\rm
Given two $k$-algebras $A,B\in\Re$, the {\it unstable algebraic
Kasparov $K$-theory space\/} of $(A,B)$ is the space
$\mathcal{K}(\Re)(A,B)$ defined as the (fibrant) space
   $$\lp_n\Hom_{\inda}(J^nA,\bb B^\Delta(\Omega^n)).$$
Its homotopy groups will be denoted by $\mathcal{K}_n(\Re)(A,B)$,
$n\geq 0$. In what follows we shall often write $\mathcal{K}(A,B)$
to denote the same space omitting $\Re$ from notation.

}\end{defs}

\begin{rem}{\rm
The space $\mathcal{K}(\Re)(A,B)$ only depends on the endofunctor
$T:\Re\to\Re$. If $A,B$ belong to another admissible category of
algebras $\Re'$ with the same endofunctor $T$, then
$\mathcal{K}(\Re)(A,B)$ equals $\mathcal{K}(\Re')(A,B)$.

}\end{rem}

We call a functor $\cc F$ from $\Re$ to simplicial sets or spectra
{\it homotopy invariant\/} if for every $B\in\Re$ the natural map
$B\to B[x]$ induces a weak equivalence $\cc F(B)\simeq \cc F(B[x])$.

\begin{lem}\label{preexcsi}
$(1)$ For any $n\geq 0$ the simplicial functor
$B\mapsto\Hom_{\inda}(A,\bb B^\Delta(\Omega^n))$ is homotopy
invariant. In particular, the simplicial functor $\mathcal{K}(A,?)$
is homotopy invariant.

$(2)$ Given a $\ff F$-fibration $f:B\to C$, let $f[x]:B[x]\to C[x]$
be the fibration $\sum b_ix^i\mapsto\sum f(b_i)x^i$. Then
$P_{f[x]}(\Omega^n)=P_{f}(\Omega^n)[x]$ and the natural map of
simplicial sets
   \begin{equation}\label{true}
    \Hom_{\inda}(A,P_f(\Omega^n))\to\Hom_{\inda}(A,P_{f[x]}(\Omega^n))
   \end{equation}
is a homotopy equivalence for any $n\geq 0$ and $A\in\Re$.
\end{lem}

\begin{proof}
(1). By Theorem~\ref{exinfi} $\Hom_{\inda}(A,\bb
B^\Delta)=Ex^\infty(\Hom_{\ahaw}(A,B^\Delta))$. It is homotopy
invariant by~\cite[3.1]{Gar}. For any $n\geq 0$ and $A\in\Re$ there
is a commutative diagram of fibre sequences\footnotesize
   $$\xymatrix{\Hom_{\inda}(A,\bb B^\Delta(\Omega^{n+1}))\ar[r]\ar[d]
               &\Hom_{\inda}(A,P\bb B^\Delta(\Omega^n))\ar[r]\ar[d]&\Hom_{\inda}(A,\bb B^\Delta(\Omega^n))\ar[d]\\
               \Hom_{\inda}(A,\bb B[x]^\Delta(\Omega^{n+1}))\ar[r]
               &\Hom_{\inda}(A,P\bb B[x]^\Delta(\Omega^n))\ar[r] &\Hom_{\inda}(A,\bb B[x]^\Delta(\Omega^n)).}$$
\normalsize By induction, if the right arrow is a weak equivalence,
then so is the left one because the spaces in the middle are
contractible.

(2). The fact that $P_{f[x]}(\Omega^n)=P_{f}(\Omega^n)[x]$ is
straightforward. The map~\eqref{true} is the fibre product map
corresponding to the commutative diagram\footnotesize
   $$\xymatrix{\Hom_{\inda}(A,\bb B^\Delta(\Omega^{n}))\ar[r]\ar[d]
               &\Hom_{\inda}(A,\bb C^\Delta(\Omega^n))\ar[d]&\ar[l]\Hom_{\inda}(A,P\bb C^\Delta(\Omega^n))\ar[d]\\
               \Hom_{\inda}(A,\bb B[x]^\Delta(\Omega^{n}))\ar[r]
               &\Hom_{\inda}(A,\bb C[x]^\Delta(\Omega^n))&\ar[l]\Hom_{\inda}(A,P\bb C[x]^\Delta(\Omega^n)).}$$
\normalsize The left and the middle vertical arrows are weak
equivalences by the first assertion. The right vertical arrow is a
weak equivalence, because it is a map between contractible spaces.
Since the right horizontal maps are fibrations, we conclude that the
desired map is a weak equivalence.
\end{proof}

We are now in a position to prove the following result.

\begin{excisa}\label{excsion}
For any algebra $A\in\Re$ and any $\ff F$-extension in $\Re$
   $$F\lra{i}B\lra{f}C$$
the induced sequence of spaces
   $$\mathcal{K}(A,F)\lra{}\mathcal{K}(A,B)\lra{}\mathcal{K}(A,C)$$
is a homotopy fibre sequence.
\end{excisa}

\begin{proof}
We have constructed above a sequence of simplicial sets
   $$\Hom_{\inda}(A,P_f)\xrightarrow{\vartheta}
     \Hom_{\inda}(JA,P_f(\Omega))\xrightarrow{\vartheta}\cdots$$
with each map $\vartheta$ defined by means of the classifying map
$\xi_{\tau}$ corresponding to the $k$-linear splitting $\tau$. Let
$\cc X$ denote its colimit. One has a homotopy cartesian square
   $$\xymatrix{\cc X\ar[r]\ar[d]_{pr}&P\mathcal{K}(A,C)\simeq *\ar[d]^{d_1}\\
               \mathcal{K}(A,B)\ar[r]^f&\mathcal{K}(A,C).}$$
By Proposition~\ref{excis} for any $n\geq 0$ there is a diagram
   $$\xymatrix{\Hom_{\inda}(J^nA,\bb F^\Delta(\Omega^n))\ar[r]^{\varsigma}\ar[d]_{\iota}
               &\Hom_{\inda}(J^{n+1}A,\bb F^\Delta(\Omega^{n+1}))\ar[d]^\iota\\
               \Hom_{\inda}(J^nA,P_f(\Omega^n))\ar[r]^{\vartheta}\ar[ur]^a
               &\Hom_{\inda}(J^{n+1}A,P_f(\Omega^{n+1}))}$$
with $\varsigma(u)=\xi_{\upsilon}\circ J(u)$,
$\vartheta(v)=\xi_\tau\circ J(v)$, $a(v)=\alpha\circ J(v)$.
Proposition~\ref{excis} also implies that $a\iota=\varsigma$,
$\iota\varsigma=\vartheta\iota$ and that there exists a map
   $$H:\Hom_{\inda}(J^nA,P_f(\Omega^n))\to\Hom_{\inda}(J^{n+1}A,P_f(\Omega^{n+1})[x])$$
such that $\partial_x^0H=\iota a$ and $\partial_x^1H=\vartheta$.

One has a commutative diagram\footnotesize
   $$\xymatrix{\Hom_{\inda}(J^{n+1}A,P_f(\Omega^{n+1}))\ar[r]^(.35){diag}\ar[d]_i
               &\Hom_{\inda}(J^{n+1}A,P_f(\Omega^{n+1}))\times\Hom_{\inda}(J^{n+1}A,P_f(\Omega^{n+1}))\\
               \Hom_{\inda}(J^{n+1}A,P_f(\Omega^{n+1})[x]).\ar[ur]_{(\partial_x^0,\partial_x^1)}}$$
\normalsize By Lemma~\ref{preexcsi}(2) $i$ is a weak equivalence. We
see that $\Hom_{\inda}(J^{n+1}A,P_f(\Omega^{n+1})[x])$ is a path
object of $\Hom_{\inda}(J^{n+1}A,P_f(\Omega^{n+1}))$ in $\bb S$.
Since all spaces in question are fibrant, we conclude that $\iota a$
is simplicially homotopic to $\vartheta$, and hence $\pi_s(\iota
a)=\pi_s(\vartheta)$, $s\geq 0$. Therefore the induced homomorphisms
   $$\pi_s(\iota):\mathcal{K}_s(A,F)\to\pi_s(\cc X),\quad s\geq 0,$$
are isomorphisms, and hence $\iota:\mathcal{K}(A,F)\to\cc X$ is a
weak equivalence.

Since the vertical arrows in the commutative diagram
   $$\xymatrix@!0{P\mathcal{K}(A,F)\ar[rrrrr] &&&&& \mathcal{K}(A,C)\\
                  && \cc X \ar[ull]\ar[rrrrr]^(.40){pr} &&&&& \mathcal{K}(A,B)\ar[ull]\\
                  {*} \ar'[rr][rrrrr] \ar[uu] &&&&& \mathcal{K}(A,C)\ar@{=}'[u][uu]\\
                  &&\mathcal{K}(A,F)\ar[ull]\ar[uu]_(.35)\iota\ar[rrrrr]_i&&&&&\mathcal{K}(A,B)\ar[ull]\ar@{=}[uu]}$$
are weak equivalences and the upper square is homotopy cartesian,
then so is the lower one (see~\cite[13.3.13]{Hir}). Thus,
   $$\mathcal{K}(A,F)\lra{}\mathcal{K}(A,B)\lra{}\mathcal{K}(A,C)$$
is a homotopy fibre sequence. The theorem is proved.
\end{proof}

\begin{cor}\label{excwww}
For any algebras $A,B\in\Re$ the space $\Omega\mathcal{K}(A,B)$ is
naturally homotopy equivalent to $\mathcal{K}(A,\Omega B)$.
\end{cor}

\begin{proof}
Consider the extension
   $$\Omega B\lra{}EB\lra{\partial_x^1}B$$
which gives rise to a homotopy fibre sequence
   $$\mathcal{K}(A,\Omega B)\to\mathcal{K}(A,EB)\to\mathcal{K}(A,B)$$
by Excision Theorem~A. Our assertion will follow once we prove that
$\mathcal{K}(A,EB)$ is contractible.

Since $EB$ is contractible, then there is an algebraic homotopy
$h:EB\to EB[x]$ contracting $EB$. There is also a commutative
diagram\footnotesize
   $$\xymatrix{\Hom_{\inda}(J^{n}A,\bb E\bb B(\Omega^n))\ar[r]^(.35){diag}\ar[d]_i
               &\Hom_{\inda}(J^{n}A,\bb E\bb B(\Omega^n))\times\Hom_{\inda}(J^{n}A,\bb E\bb B(\Omega^n))\\
               \Hom_{\inda}(J^{n}A,\bb E\bb B[x](\Omega^{n})).\ar[ur]_{(\partial_x^0,\partial_x^1)}}$$
\normalsize By Lemma~\ref{preexcsi}(1) $i$ is a weak equivalence. We
see that $\Hom_{\inda}(J^{n}A,\bb E\bb B[x](\Omega^{n}))$ is a path
object of $\Hom_{\inda}(J^{n}A,\bb E\bb B(\Omega^{n}))$ in $\bb S$,
and hence the induced map
   $$h_*:\Hom_{\inda}(J^{n}A,\bb E\bb B(\Omega^{n}))\to\Hom_{\inda}(J^{n}A,\bb E\bb B[x](\Omega^{n}))$$
is such that $\partial_x^1h_*=\id$ is homotopic to
$\partial_x^0h_*=const$. Thus $\Hom_{\inda}(J^{n}A,\bb E\bb
B(\Omega^{n}))$ is contractible, and hence so is
$\mathcal{K}(A,EB)$.
\end{proof}

We have proved that the simplicial functor $\mathcal{K}(A,B)$ is
excisive in the second argument. It turns out that it is also
excisive in the first argument.

\begin{excisb}\label{excsionb}
For any algebra $D\in\Re$ and any $\ff F$-extension in $\Re$
   $$F\lra{i}B\lra{f}C$$
the induced sequence of spaces
   $$\mathcal{K}(C,D)\lra{}\mathcal{K}(B,D)\lra{}\mathcal{K}(F,D)$$
is a homotopy fibre sequence.
\end{excisb}

The proof of this theorem requires some machinery. We shall use
recent techniques and results from homotopical algebra (both stable
and unstable). The proof is on page~\pageref{pageexc}.

\section{The spectrum $\bb K^{unst}(A,B)$}\label{thespectrum}

Throughout this section $\Re$ is assumed to be $T$-closed.

\begin{thm}\label{loop}
Let $A,B\in\Re$; then there is a natural isomorphism of simplicial
sets
   $$\cc K(A,B)\cong\Omega\cc K(JA,B).$$
In particular, $\cc K(A,B)$ is an infinite loop space with $\cc
K(A,B)$ simplicially isomorphic to $\Omega^n\cc K(J^nA,B).$
\end{thm}

\begin{proof}
For any $n\in\bb N$ there is a commutative diagram
   $$\xymatrix{P\bb B^\Delta(\Omega^n)\ar[d]_{d_{1,\bb B^\Delta(\Omega^n)}}\ar@{ >->}[r]
               &PP\bb B^\Delta(\Omega^{n-1})\ar[d]^{d_{1,P\bb B^\Delta(\Omega^{n-1})}}\ar@{->>}[r]^{Pd_1}
               &P\bb B^\Delta(\Omega^{n-1})\ar[d]^{d_{1,\bb B^\Delta(\Omega^{n-1})}}\\
               \bb B^\Delta(\Omega^n)\ar@{ >->}[r]&P\bb B^\Delta(\Omega^{n-1})\ar@{->>}[r]_{d_1}&\bb B^\Delta(\Omega^{n-1}).}$$
The definition of the natural splitting $\upsilon$ to the lower
right arrow is naturally lifted to a natural splitting
$\nu:=P\upsilon$ for the upper right arrow in such a way that
$d_1\circ\nu=\upsilon\circ d_1$. It follows from
Lemma~\ref{classqmn} that the corresponding diagram of the
classifying maps
   $$\xymatrix{JP\bb B^\Delta(\Omega^{n-1})\ar[r]^{\xi_\nu}\ar[d]_{J(d_1)}&P\bb B^\Delta(\Omega^n)\ar[d]^{d_1}\\
               J\bb B^\Delta(\Omega^{n-1})\ar[r]^{\xi_\upsilon}&\bb B^\Delta(\Omega^n)}$$
is commutative. There is also a commutative diagram with exact rows
for every $n\geq 1$
   $$\xymatrix{\bb B^\Delta(\Omega^{n+1})\ar@{ >->}[d]_{\sw}\ar@{ >->}[r]^j
               &P\bb B^\Delta(\Omega^n)\ar@{=}[d]\ar@{->>}[r]^{d_1}&\bb B^\Delta(\Omega^{n})\ar@{=}[d]\\
               \bb B^\Delta(\Omega^{n+1})\ar@{ >->}[d]_j\ar@{ >->}[r]^{\sw\circ j}
               &P\bb B^\Delta(\Omega^n)\ar[d]^{\sw\circ Pi}\ar@{->>}[r]^{d_1}&\bb B^\Delta(\Omega^{n})\ar@{ >->}[d]^i\\
               P\bb B^\Delta(\Omega^n)\ar@{ >->}[r]&PP\bb B^\Delta(\Omega^{n-1})\ar@{->>}[r]_{Pd_1}
               &P\bb B^\Delta(\Omega^{n-1}).}$$
with $i,j$ natural inclusions and $\sw$ permuting the last two
coordinates. One has a commutative diagram of classifying maps
   $$\xymatrix{J\bb B^\Delta(\Omega^n)\ar@{=}[d]\ar[r]^{\xi_\upsilon}&\bb B^\Delta(\Omega^{n+1})\ar[d]^{\sw}\\
               J\bb B^\Delta(\Omega^n)\ar[d]_{J(i)}\ar[r]^{\sw\circ\xi_\upsilon}&\bb B^\Delta(\Omega^{n+1})\ar[d]^j\\
               JP\bb B^\Delta(\Omega^{n-1})\ar[r]^{\xi_\nu}&P\bb B^\Delta(\Omega^n)}$$

Observe that each simplicial map
   $$\Hom_{\inda}(J^nA,P\bb B^\Delta(\Omega^{n-1}))\xrightarrow{P\varsigma}\Hom_{\inda}(J^{n+1}A,P\bb B^\Delta(\Omega^n))$$
agrees with the map defined like $\varsigma$ but using $\xi_\nu$. To
see this, it is enough to consider the following commutative diagram
with exact rows and obvious splittings:
   $$\xymatrix{J(P\bb B^\Delta)\ar[d]\ar@{ >->}[r]&T(P\bb B^\Delta)\ar[d]\ar@{->>}[r]^{\eta_{P\bb B^\Delta}}&P\bb B^\Delta\ar@{=}[d]\\
               P(J\bb B^\Delta)\ar[d]\ar@{ >->}[r]&P(T\bb B^\Delta)\ar[d]\ar@{->>}[r]^{P(\eta_{\bb B^\Delta})}&P\bb B^\Delta\ar@{=}[d]\\
               P(\bb B^\Delta(\Omega))\ar@{ >->}[r]&PP\bb B^\Delta\ar@{->>}[r]_{Pd_1}&P\bb B^\Delta}$$
Therefore all squares of the diagram \small
   $$\xymatrix{\Omega\cc K(JA,B):&\cdots\ar[r]^(.25){\sw\circ\varsigma}
               &\Hom_{\inda}(J^nA,\bb B^\Delta(\Omega^n))\ar[d]\ar[r]^(.45){\sw\circ\varsigma}
               &\Hom_{\inda}(J^{n+1}A,\bb B^\Delta(\Omega^{n+1}))\ar[d]\ar[r]^(.78){\sw\circ\varsigma}&\cdots\\
               P\cc K(JA,B):&\cdots\ar[r]^(.20){P\varsigma}&\Hom_{\inda}(J^nA,P\bb B^\Delta(\Omega^{n-1}))\ar[d]_{d_1}\ar[r]^{P\varsigma}
               &\Hom_{\inda}(J^{n+1}A,P\bb B^\Delta(\Omega^n))\ar[d]^{d_1}\ar[r]^(.80){P\varsigma}&\cdots\\
               \cc K(JA,B):&\cdots\ar[r]^(.20)\varsigma&\Hom_{\inda}(J^nA,\bb B^\Delta(\Omega^{n-1}))\ar[r]^\varsigma
               &\Hom_{\inda}(J^{n+1}A,\bb B^\Delta(\Omega^n))\ar[r]^(.80)\varsigma&\cdots}$$
\normalsize are commutative. The desired isomorphism $\cc
K(A,B)\cong\Omega\cc K(JA,B)$ is encoded by the following
commutative diagram:
   $$\xymatrix{\Hom_{\inda}(A,\bb B^\Delta)\ar[d]_\varsigma\ar[r]^(.45)\varsigma&\Hom_{\inda}(JA,\bb B^\Delta(\Omega))\ar[d]_\varsigma\ar[r]^\varsigma
               &\Hom_{\inda}(J^2A,\bb B^\Delta(\Omega^2))\ar[d]_\varsigma\ar[r]^(.75)\varsigma&\cdots\\
               \Hom_{\inda}(JA,\bb B^\Delta(\Omega))\ar@{=}[d]\ar[r]^(.45){\varsigma}&\Hom_{\inda}(J^2A,\bb B^\Delta(\Omega^2))\ar[d]_{(21)}\ar[r]^{\varsigma}
               &\Hom_{\inda}(J^3A,\bb B^\Delta(\Omega^3))\ar[d]_{(321)}\ar[r]^(.75)\varsigma&\cdots\\
               \Hom_{\inda}(JA,\bb B^\Delta(\Omega))\ar[r]^(.45){\sw\circ\varsigma}
               &\Hom_{\inda}(J^2A,\bb B^\Delta(\Omega^2))\ar[r]^{\sw\circ\varsigma}
               &\Hom_{\inda}(J^3A,\bb B^\Delta(\Omega^3))\ar[r]^(.75){\sw\circ\varsigma}&\cdots}$$
The colimit of the upper sequence is $\cc K(A,B)$ and the colimit of
the lower one is $\Omega\cc K(JA,B)$. The cycle
$(n\cdots21)\in\Sigma_n$ permutes coordinates of $\bb
B^\Delta(\Omega^n)$.
\end{proof}

\begin{cor}\label{excspm}
For any algebras $A,B\in\Re$ the space $\mathcal{K}(A,B)$ is
naturally homotopy equivalent to $\mathcal{K}(JA,\Omega B)$.
\end{cor}

\begin{proof}
This follows from the preceding theorem and Corollary~\ref{excwww}.
\end{proof}

\begin{defs}{\rm
Given two $k$-algebras $A,B\in\Re$, the {\it unstable algebraic
Kasparov $KK$-theory spectrum\/} of $(A,B)$ consists of the sequence
of spaces
   $$\mathcal{K}(A,B),\mathcal{K}(JA,B),\mathcal{K}(J^2A,B),\ldots$$
together with isomorphisms
$\mathcal{K}(J^nA,B)\cong\Omega\mathcal{K}(J^{n+1}A,B)$ constructed
in Theorem~\ref{loop}. It forms an $\Omega$-spectrum which we also
denote by $\mathbb{K}^{unst}(A,B)$. Its homotopy groups will be
denoted by $\mathbb{K}_n^{unst}(A,B)$, $n\in\bb Z$. We sometimes
write $\mathbb{K}(A,B)$ instead of $\mathbb{K}^{unst}(A,B)$,
dropping ``unst" from notation.

Observe that $\mathbb{K}_n(A,B)\cong\mathcal{K}_n(A,B)$ for any
$n\geq 0$ and $\mathbb{K}_n(A,B)\cong\mathcal{K}_0(J^{-n}A,B)$ for
any $n<0$.

}\end{defs}

\begin{thm}\label{spectrumunst}
The assignment $B\mapsto\mathbb{K}(A,B)$ determines a functor
   $$\mathbb{K}(A,?):\Re\to(Spectra)$$
which is homotopy invariant and excisive in the sense that for every
$\ff F$-extension $F\to B\to C$ the sequence
   $$\mathbb{K}(A,F)\to\mathbb{K}(A,B)\to\mathbb{K}(A,C)$$
is a homotopy fibration of spectra. In particular, there is a long
exact sequence of abelian groups
   $$\cdots\to\mathbb{K}_{i+1}(A,C)\to\mathbb{K}_i(A,F)\to\mathbb{K}_i(A,B)\to\mathbb{K}_i(A,C)\to\cdots$$
for any $i\in\bb Z$.
\end{thm}

\begin{proof}
This follows from Excision Theorem~A.
\end{proof}

We also have the following

\begin{thm}\label{spectrumunstb}
The assignment $B\mapsto\mathbb{K}(B,D)$ determines a functor
   $$\mathbb{K}(?,D):\Re^{\op}\to(Spectra),$$
which is excisive in the sense that for every $\ff F$-extension
$F\to B\to C$ the sequence
   $$\mathbb{K}(C,D)\to\mathbb{K}(B,D)\to\mathbb{K}(F,D)$$
is a homotopy fibration of spectra. In particular, there is a long
exact sequence of abelian groups
   $$\cdots\to\mathbb{K}_{i+1}(F,D)\to\mathbb{K}_i(C,D)\to\mathbb{K}_i(B,D)\to\mathbb{K}_i(F,D)\to\cdots$$
for any $i\in\bb Z$.
\end{thm}

\begin{proof}
We postpone the proof till subsection~\ref{trew}.
\end{proof}

The reader may have observed that we do not involve any matrices in
the definition of $\cc K(A,B)$ as {\it any\/} sort of algebraic
$K$-theory does. This is one of {\it important\/} differences with
usual views on algebraic $K$-theory. The author is motivated by the
fact that many interesting admissible categories of algebras
deserving to be considered like that of all commutative ones are not
closed under matrices.

\section{Homotopy theory of algebras}\label{homotopy}

Let $\Re$ be a {\it small\/} admissible category of rings. In order
to prove Excision Theorem~B, we have to use results from homotopy
theory of rings. First, we introduce a model category
$U\Re_\bullet^{I,J}$ of pointed simplicial functors from $\Re$ to
$\bb S_\bullet$. This model category is a reminiscence of
Morel--Voevodsky~\cite{MV} motivic model category of pointed motivic
spaces. Second, we define a model category of $S^1$-spectra
$Sp(\Re)$ associated with $U\Re_\bullet^{I,J}$. A typical fibrant
spectrum of $Sp(\Re)$ is $\bb K^{unst}(A,-)$, $A\in\Re$. The
strategy of proving Excision Theorem~B is first to prove a kind of
Excision Theorem~B for the spectra $\bb K^{unst}(A,-)$ (see
Theorem~\ref{mainstab}) and then use standard facts from homotopical
algebra to show the original Excision Theorem~B on the level of
spaces. We mostly adhere to~\cite{Gar}.

\subsection{The category of simplicial functors $U\Re$}

We shall use the model category $U\Re$ of covariant functors from
$\Re$ to simplicial sets (and not contravariant functors as usual).
We do not worry about set theoretic issues here, because we assume
$\Re$ to be small. We shall consider both the injective and
projective model structures on $U\Re$ and refer to Dugger~\cite{D}
for further details. Both model structures are Quillen equivalent.
These are proper, simplicial, cellular model category structures
with weak equivalences and cofibrations (respectively fibrations)
being defined objectwise, and fibrations (respectively cofibrations)
being those maps having the right (respectively left) lifting
property with respect to trivial cofibrations (respectively trivial
fibrations). The fully faithful contravariant functor
   $$r:\Re\to U\Re,\ \ \ A\longmapsto\Hom_{\Re}(A,-),$$
where $rA(B)=\Hom_{\Re}(A,B)$ is to be thought of as the constant
simplicial set for any $B\in\Re$.

In the injective model structure on $U\Re$, cofibrations are the
injective maps. This model structure enjoys the following properties
(see Dugger~\cite[p.~21]{D}):

\begin{itemize}
\item[$\diamond$] every object is cofibrant;

\item[$\diamond$] being fibrant implies being objectwise
fibrant, but is stronger (there are additional diagramatic
conditions involving maps being fibrations, etc.);

\item[$\diamond$] any object which is constant in the simplicial
direction is fibrant.

\end{itemize}
If $F\in U\Re$ then $U\Re(rA\times\Delta^n,F)=F_n(A)$ (isomorphism
of sets). Hence, if we look at simplicial mapping spaces we find
   $$\Map(rA,F)=F(A)$$
(isomorphism of simplicial sets). This is a kind of ``simplicial
Yoneda Lemma''.

In the projective model structure on $U\Re$, fibrations are defined
sectionwise. The class of projective cofibrations is generated by
the set
$$
I_{U\Re}\equiv\{rA\times(\partial\Delta^{n}\subset\Delta^{n})\}^{n\geq
0}$$ indexed by $A\in\Re$. Likewise, the class of acyclic projective
cofibrations is generated by
$$
J_{U\Re}\equiv\{rA\times(\Lambda^{k}_n\subset\Delta^{n})\}^{n>0}_{0\leq
k\leq n}.
$$

The projective model structure on $U\Re$ enjoys the following
properties:

\begin{itemize}
\item[$\diamond$] every projective cofibration is an injective map (but not vice versa);

\item[$\diamond$] if $A\in\Re$ and $K$ is a simplicial set, then $rA\times K$ is a projective
cofibrant simplicial functor. In particular, $rA$ is projective
cofibrant for every algebra $A\in\Re$;

\item[$\diamond$] $rA$ is projective fibrant for every algebra $A\in\Re$.

\end{itemize}

\subsection{Bousfield localization}

Recall from~\cite{Hir} that if $\cc M$ is a model category and $S$ a
set of maps between cofibrant objects, one can produce a new model
structure on $\cc M$ in which the maps $S$ are weak equivalences.
The new model structure is called the {\it Bousfield localization\/}
or just localization of the old one. Since all model categories we
shall consider are simplicial one can use the simplicial mapping
object instead of the homotopy function complex for the localization
theory of $\cc M$.

\begin{defs}{\rm
Let $\cc M$ be a simplicial model category and let $S$ be a set of
maps between cofibrant objects.

\begin{enumerate}

\item An {\it $S$-local object\/} of $\cc M$ is a fibrant object
$X$ such that for every map $A \to B$ in $S$, the induced map of
$\Map(B,X)\to\Map(A,X)$ is a weak equivalence of simplicial sets.

\item An {\it $S$-local equivalence\/} is a map $A\to B$ such
that $\Map(B,X) \to\Map(A,X)$ is a weak equivalence for every
$S$-local object $X$.

\end{enumerate}
}\end{defs}

In words, the $S$-local objects are the ones which see every map in
$S$ as if it were a weak equivalence.  The $S$-local equivalences
are those maps which are seen as weak equivalences by every
$S$-local object.

\begin{thm}[Hirschhorn~\cite{Hir}]
Let $\cc M$ be a cellular, simplicial model category and let $S$ be
a set of maps between cofibrant objects. Then there exists a model
category $\cc M/S$ whose underlying category is that of $\cc M$ in
which

\begin{enumerate}

\item the weak equivalences are the $S$-local equivalences;

\item the cofibrations in $\cc M/S$ are the same as those in $\cc
M$;

\item the fibrations are the maps having the
right-lifting-property with respect to cofibrations which are also
$S$-local equivalences.
\end{enumerate}
Left Quillen functors from $\cc M/S$ to $\cc D$ are in one to one
correspondence with left Quillen functors $\varPhi:\cc M\to\cc D$
such that $\varPhi(f)$ is a weak equivalence for all $f\in S$. In
addition, the fibrant objects of $\cc M$ are precisely the $S$-local
objects, and this new model structure is again cellular and
simplicial.

\end{thm}

The model category $\cc M/S$ whose existence is guaranteed by the
above theorem is called the {\it $S$-localization\/} of $\cc M$. The
underlying category is the same as that of $\cc M$, but there are
more trivial cofibrations (and hence fewer fibrations).

Note that the identity maps yield a Quillen pair $\cc M
\rightleftarrows\cc M/S$, where the left Quillen functor is the map
$\id:\cc M\to\cc M/S$.

\subsection{The model category $U\Re_I$}

Let $I=\{i=i_A:r(A[t])\to r(A)\mid A\in\Re\}$, where each $i_A$ is
induced by the natural homomorphism $i:A\to A[t]$. Consider the
injective model structure on $U\Re$. We shall refer to the $I$-local
equivalences as (injective) $I$-weak equivalences. The resulting
model category $U\Re/I$ will be denoted by $U\Re_I$ and its homotopy
category is denoted by $\Ho_I(\Re)$. Notice that any homotopy
invariant functor $F:\Re\to Sets$ is an $I$-local object in $U\Re$
(hence fibrant in $U\Re_I$).

Let $F$ be a functor from $\Re$ to simplicial sets. There is a {\it
singular functor\/} $Sing_*(F)$ which is defined at each algebra $R$
as the diagonal of the bisimplicial set $F(R^\Delta)$. Thus
$Sing_*(F)$ is also a functor from $\Re$ to simplicial sets. If we
consider $R$ as a constant simplicial algebra, then the natural map
$R\to R^\Delta$ yields a natural transformation $F\to Sing_*(F)$. It
is an $I$-trivial cofibration by~\cite[3.8]{Gar}.

If we consider the projective model structure on $U\Re$, then we
shall refer to the $I$-local equivalences (respectively fibrations
in the $I$-localized model structure) as projective $I$-weak
equivalences (respectively $I$-projective fibrations). The resulting
model category $U\Re/I$ will be denoted by $U\Re^I$. It is shown
similar to~\cite[3.49]{O} that the classes of injective and
projective $I$-weak equivalences coincide. Hence the identity
functor on $U\Re$ is a Quillen equivalence between $U\Re^I$ and
$U\Re_I$.

The model category $U\Re^I$ satisfies some finiteness conditions.

\begin{defs}[\cite{H}]{\rm
An object $A$ of a model category $\cc M$ is {\it finitely
presentable\/} if the set-valued Hom-functor $\Hom_{\cc M}(A,-)$
commutes with all colimits of sequences $X_0\to X_1\to
X_2\to\cdots$. A cofibrantly generated model category with
generating sets of cofibrations $\cc I$ and trivial cofibrations
$\cc J$ is called {\it finitely generated\/} if the domains and
codomains of $\cc I$ and $\cc J$ are finitely presentable, and {\it
almost finitely generated\/} if the domains and codomains of $\cc I$
are finitely presentable and there exists a set of trivial
cofibrations $\cc J'$ with finitely presentable domains and
codomains such that a map with fibrant codomain is a fibration if
and only if it has the right lifting property with respect to $\cc
J'$. }\end{defs}

Using the simplicial mapping cylinder in $U\Re$ (it is the usual one
from simplicial sets applied objectwise), we may factor the morphism
$$
\xymatrix{r(A[t])\ar[r] & rA}
$$
into a projective cofibration composed with a simplicial homotopy
equivalence in $U\Re$
\begin{equation}
\label{homot} \xymatrix{ r(A[t])\ar[r] &
\cyl\bigl(r(A[t])\rightarrow rA\bigr)\ar[r] & rA.}
\end{equation}
Observe that the maps in (\ref{homot}) are $I$-weak equivalences.

Let $J_{U\Re^I}$ denote the set of maps
$$
r(A[t])\times\Delta^{n}\bigsqcup_{r(A[t])\times\partial\Delta^{n}}
\cyl\bigl(r(A[t])\rightarrow rA\bigr)\times\partial\Delta^{n}\to
\cyl\bigl(r(A[t])\rightarrow rA\bigr)\times\Delta^{n}
$$
indexed by $n\geq 0$ and $A\in\Re$.

Let $\Lambda$ be a set of generating trivial cofibrations for the
injective model structure on $U\Re$. Using~\cite[4.2.4]{Hir} a
simplicial functor $\cc X$ is $I$-local in the injective
(respectively projective) model structure if and only if it has the
right lifting property with respect to $\Lambda\cup J_{U\Re^I}$
(respectively $J_{U\Re}\cup J_{U\Re^I}$). It follows
from~\cite[4.2]{H} that $U\Re^I$ is almost finitely generated,
because domains and codomains of $J_{U\Re}\cup J_{U\Re^I}$ are
finitely presentable.

\subsection{The model category $U\Re_{J}$}

Let us introduce the class of excisive functors on $\Re$. They look
like flasque presheaves on a site defined by a cd-structure in the
sense of Voevodsky~\cite[section 3]{V}.

\begin{defs}{\rm
Let $\Re$ be an admissible category of algebras. A simplicial
functor $\cc X\in U\Re$ is called {\it excisive\/} with respect to
$\ff F$ if $\cc X(0)$ is contractible and for any cartesian square
in $\Re$
   $$\xymatrix{
      D\ar[r]\ar[d]&A\ar[d]\\
      B\ar[r]^f&C
     }$$
with $f$ a fibration (call such squares {\it distinguished}) the
square of simplicial sets
   $$\xymatrix{
      \cc X(D)\ar[r]\ar[d]&{\cc X(A)}\ar[d]\\
      {\cc X(B)}\ar[r]&\cc X(C)
     }$$
is a homotopy pullback square. It immediately follows from the
definition that every pointed excisive object takes $\ff
F$-extensions in $\Re$ to homotopy fibre sequences of simplicial
sets.

}\end{defs}

Consider the injective model structure on $U\Re$. Let $\alpha$
denote a distinguished square in $\Re$
   $$\xymatrix{
      D\ar[r]\ar[d]&A\ar[d]\\
      B\ar[r]&C
     }$$
and denote the pushout of the diagram
   $$\xymatrix{
      rC\ar[r]\ar[d]&rA\\
      rB}$$
by $P(\alpha)$. Notice that the diagram obtained is homotopy
pushout. There is a natural map $P(\alpha)\to rD$, and both objects
are cofibrant. In the case of the degenerate square this map has to
be understood as the map from the initial object $\emptyset$ to
$r0$.

We can localize $U\Re$ at the family of maps
   $$J=\{P(\alpha)\to rD\mid\textrm{ $\alpha$ is a distinguished square}\}.$$
The corresponding $J$-localization will be denoted by $U\Re_J$. The
weak equivalences (trivial cofibrations) of $U\Re_{J}$ will be
referred to as (injective) $J$-weak equivalences ((injective)
$J$-trivial cofibrations).

It follows that the square ``$r(\alpha)$''
   $$\xymatrix{
      rC\ar[r]\ar[d]&{rA}\ar[d]\\
      {rB}\ar[r]& rD}$$
with $\alpha$ a distinguished square is a homotopy pushout square in
$U\Re_J$. A simplicial functor $\cc X$ in $U\Re$ is $J$-local if and
only if it is fibrant and excisive~\cite[4.3]{Gar}.

We are also interested in constructing sets of generating acyclic
cofibrations for model structures. Let us apply the simplicial
mapping cylinder construction $\cyl$ to distinguished squares and
form the pushouts:
   $$\xymatrix{rC\ar[r]\ar[d]&\cyl(rC\to rA)\ar[d]\ar[r]&rA\ar[d]\\
               rB\ar[r]&\cyl(rC\to rA)\bigsqcup_{rC}rB\ar[r]&rD}$$
Note that $rC\to\cyl(rC\rightarrow rA)$ is both an injective and a
projective cofibration between (projective) cofibrant simplicial
functors. Thus $s(\alpha)\equiv\cyl(rC\to rA)\bigsqcup_{rC}rB$ is
(projective) cofibrant \cite[1.11.1]{Hov}. For the same reasons,
applying the simplicial mapping cylinder to $s(\alpha)\rightarrow
rD$ and setting $t(\alpha)\equiv\cyl\bigl(s(\alpha)\to rD\bigr)$ we
get a projective cofibration
   $$\xymatrix{\cyl(\alpha)\colon s(\alpha)\ar[r] & t(\alpha).}$$
Let $J^{\cyl(\alpha)}_{U\Re}$ consists of maps
$$
\xymatrix{
s(\alpha)\times\Delta^{n}\bigsqcup_{s(\alpha)\times\partial\Delta^{n}}
t(\alpha)\times\partial\Delta^{n}\ar[r] &
t(\alpha)\times\Delta^{n}.}
$$
It is directly verified that a simplicial functor $\cc X$ is
$J$-local if and only if it has the right lifting property with
respect to $\Lambda\cup J^{\cyl(\alpha)}_{U\Re}$, where $\Lambda$ is
a set of generating trivial cofibrations for the injective model
structure on $U\Re$.

If one localizes the projective model structure on $U\Re$ with
respect to the set of projective cofibrations
$\{\cyl(\alpha)\}_{\alpha}$, the resulting model category shall be
denoted by $U\Re^J$. The weak equivalences (trivial cofibrations) of
$U\Re^{J}$ will be referred to as projective $J$-weak equivalences
(projective $J$-trivial cofibrations). As above, $\cc X$ is fibrant
in $U\Re^J$ if and only if it has the right lifting property with
respect to $J_{U\Re}\cup J^{\cyl(\alpha)}_{U\Re}$. Since both
domains and codomains in $J_{U\Re}\cup J^{\cyl(\alpha)}_{U\Re}$ are
finitely presentable then $U\Re^J$ is almost finitely generated
by~\cite[4.2]{H}.

It can be shown similar to~\cite[3.49]{O} that the classes of
injective and projective $J$-weak equivalences coincide. Hence the
identity functor on $U\Re$ is a Quillen equivalence between $U\Re_J$
and $U\Re^J$.

\subsection{The model category $U\Re_{I,J}$}

\begin{defs}{\rm
A simplicial functor $\cc X\in U\Re$ is called {\it quasi-fibrant\/}
with respect to $\ff F$ if it is homotopy invariant and excisive.
For instance, if $\Re$ is $T$-closed and $A\in\Re$ then the
simplicial functor $\cc K(A,?)$ is quasi-fibrant by
Lemma~\ref{preexcsi} and Excision Theorem~A.

Consider the injective model structure on $U\Re$. The model category
$U\Re_{I,J}$ is, by definition, the Bousfield localization of $U\Re$
with respect to $I\cup J$. Equivalently, $U\Re_{I,J}$ is the
Bousfield localization of $U\Re$ with respect to $\{\cyl(r(A[t])\to
rA)\}\cup\{\cyl(\alpha)\}$, where $A$ runs over the objects from
$\Re$ and $\alpha$ runs over the distinguished squares. The weak
equivalences (trivial cofibrations) of $U\Re_{I,J}$ will be referred
to as (injective) $(I,J)$-weak equivalences ((injective)
$(I,J)$-trivial cofibrations). By~\cite[4.5]{Gar} a simplicial
functor $\cc X\in U\Re$ is $(I,J)$-local if and only if it is
fibrant, homotopy invariant and excisive.

}\end{defs}

\begin{defs}{\rm
Following~\cite{Gar} a homomorphism $A\to B$ in $\Re$ is said to be
an {\it $\ff F$-quasi-isomorphism\/} or just a {\it
quasi-isomorphism\/} if the map $rB\to rA$ is an $(I,J)$-weak
equivalence.

}\end{defs}

Consider now the projective model structure on $U\Re$. The model
category $U\Re^{I,J}$ is, by definition, the Bousfield localization
of $U\Re$ with respect to $\{\cyl(r(A[t])\to
rA)\}\cup\{\cyl(\alpha)\}$, where $A$ runs over the objects from
$\Re$ and $\alpha$ runs over the distinguished squares. The weak
equivalences (trivial cofibrations) of $U\Re^{I,J}$ will be referred
to as projective $(I,J)$-weak equivalences (projective
$(I,J)$-trivial cofibrations). Similar to~\cite[4.5]{Gar} a
simplicial functor $\cc X\in U\Re$ is fibrant in $U\Re^{I,J}$ if and
only if it is projective fibrant, homotopy invariant and excisive
or, equivalently, it has the right lifting property with respect to
$J_{U\Re}\cup J_{U\Re_I}\cup J^{\cyl(\alpha)}_{U\Re}$. Since both
domains and codomains in $J_{U\Re}\cup J_{U\Re_I}\cup
J^{\cyl(\alpha)}_{U\Re}$ are finitely presentable then $U\Re^{I,J}$
is almost finitely generated by~\cite[4.2]{H}.

It can be shown similar to~\cite[3.49]{O} that the classes of
injective and projective $(I,J)$-weak equivalences coincide. Hence
the identity functor on $U\Re$ is a Quillen equivalence between
$U\Re_{I,J}$ and $U\Re^{I,J}$.

It is straightforward to show that the results for the model
structures on $U\Re$ have analogs for the category $U\Re_\bullet$ of
pointed simplicial functors (see~\cite{Gar}). In order to prove
Excision Theorem~B, we have to consider a model category of spectra
for $U\Re^{I,J}_\bullet$.

\subsection{The category of spectra}\label{trew}

In this section we assume $\Re$ to be small and $T$-closed. We use
here ideas and work of Hovey~\cite{H}, Jardine~\cite{J} and
Schwede~\cite{Sch}.

\begin{defs}{\rm
The category $\Spt$ of spectra consists of sequences $\cc
E\equiv(\cc E_{n})_{n\geq 0}$ of pointed simplicial functors
equipped with structure maps $\sigma_{n}^{\cc E}:\Sigma\cc
E_{n}\rightarrow\cc E_{n+1}$ where $\Sigma=S^1\wedge-$ is the
suspension functor. A map $f:\cc E\rightarrow\cc F$ of spectra
consists of compatible maps $f_{n}:\cc E_{n}\to\cc F_{n}$ of pointed
simplicial functors in the sense that the diagrams
$$
\xymatrix{
\Sigma\cc E_{n}\ar[d]_-{\Sigma f_{n}}\ar[r]^-{\sigma_{n}^{\cc E}} & \cc E_{n+1}\ar[d]^-{f_{n+1}} \\
\Sigma\cc F_{n}\ar[r]^-{\sigma_{n}^{\cc F}} & \cc F_{n+1} }
$$
commute for all $n\geq 0$. }\end{defs}

\begin{example}{\rm
The main spectrum we shall work with is as follows. Let $A\in\Re$
and let $\cc R(A)$ be the spectrum which is defined at every
$B\in\Re$ as the sequence of spaces pointed at zero
   $$\Hom_{\inda}(A,\bb B^\Delta),\Hom_{\inda}(JA,\bb B^\Delta),\Hom_{\inda}(J^2A,\bb B^\Delta),\ldots$$
By Theorem~\ref{exinfi} each $\cc R(A)_n(B)$ is a fibrant simplicial
set and by Corollary~\ref{corloop2}
   $$\Omega^k\cc R(A)_n(B)=\Hom_{\inda}(J^nA,\bb B^\Delta(\Omega^k)).$$
Each structure map $\sigma_n:\Sigma\cc R(A)_n\to\cc R(A)_{n+1}$ is
defined at $B$ as adjoint to the map
$\varsigma:\Hom_{\inda}(J^nA,\bb
B^\Delta)\to\Hom_{\inda}(J^{n+1}A,\bb B^\Delta(\Omega))$ constructed
in~\eqref{adj}.

}\end{example}

A map $f:\cc E\to\cc F$ is a level weak equivalence (respectively
fibration) if $f_{n}\colon\cc E_{n}\rightarrow\cc F_{n}$ is a
$(I,J)$-weak equivalence (respectively projective
$(I,J)$-fibration). And $f$ is a projective cofibration if $f_{0}$
and the maps
$$
\xymatrix{\cc E_{n+1}\bigsqcup_{\Sigma\cc E_{n}}\Sigma\cc
F_{n}\ar[r] & \cc F_{n+1} }
$$
are cofibrations in $U\Re_\bullet^{I,J}$ for all $n\geq 0$.
By~\cite{H,J,Sch} we have:

\begin{prop}\label{proposi}
The level weak equivalences, projective cofibrations and level
fibrations furnish a simplicial and left proper model structure on
$\Spt$. We call this the projective model structure.
\end{prop}

The Bousfield--Friedlander category of spectra~\cite{BF} will be
denoted by $Sp$. There is a functor
   $$Sp\to\Spt$$
that takes a spectrum of pointed simplicial sets $\cc E$ to the
constant spectrum $A\in\Re\mapsto\cc E(A)=\cc E$. For any algebra
$D\in\Re$ there is also a functor
   $$U_D:\Spt\to Sp,\quad\cc X\mapsto\cc X(D).$$

Given a spectrum $\cc E\in Sp$ and a pointed simplicial functor $K$,
there is a spectrum $\cc E\wedge K$ with $(\cc E\wedge K)_n=\cc
E_n\wedge K$ and having structure maps of the form
   $$\Sigma(\cc E_n\wedge K)\cong(\Sigma\cc E_n)\wedge K\xrightarrow{\sigma_n\wedge K}\cc E_{n+1}\wedge K.$$
Given $D\in \Re$, the functor $F_D:Sp\to\Spt$, $\cc E\mapsto\cc
E\wedge rD_+$, is left adjoint to $U_D:\Spt\to Sp$. So there is an
isomorphism
\begin{equation}\label{tutu}
    \Hom_{\Spt}(\cc E\wedge rD_+,\cc X)\cong\Hom_{Sp}(\cc E,\cc X(D)).
\end{equation}

Our next objective is to define the stable model structure. We
define the fake suspension functor $\Sigma:\Spt\to\Spt$ by
$(\Sigma\cc Z)_n=\Sigma\cc Z_n$ and structure maps
   $$\Sigma(\Sigma\cc Z_n)\xrightarrow{\Sigma\sigma_n}\Sigma\cc Z_{n+1},$$
where $\sigma_n$ is a structure map of $\cc Z$. Note that the fake
suspension functor is left adjoint to the fake loops functor
$\Omega^\ell:\Spt\to\Spt$ defined by $(\Omega^\ell\cc Z)_n=\Omega\cc
Z_n$ and structure maps adjoint to
   $$\Omega\cc Z_n\xrightarrow{\Omega\wt\sigma_n}\Omega(\Omega\cc Z_{n+1}),$$
where $\wt\sigma_n$ is adjoint to the structure map $\sigma_n$ of
$\cc Z$.

\begin{defs}{\rm
A spectrum $\cc Z$ is {\it stably fibrant\/} if it is level fibrant
and all the adjoints $\widetilde\sigma_{n}^{\cc Z}\colon\cc
Z_{n}\to\Omega\cc Z_{n+1}$ of its structure maps are $(I,J)$-weak
equivalences. }\end{defs}

\begin{example}{\rm
Given $A\in\Re$, the spectrum $\mathbb{K}(A,-)$ consists of the
sequence of simplicial functors
   $$\mathcal{K}(A,-),\mathcal{K}(JA,-),\mathcal{K}(J^2A,-),\ldots$$
together with isomorphisms
$\mathcal{K}(J^nA,-)\cong\Omega\mathcal{K}(J^{n+1}A,-)$ constructed
in Theorem~\ref{loop}. Lemma~\ref{preexcsi} and Excision Theorem~A
imply $\mathbb{K}(A,-)$ is a stably fibrant spectrum. Note that
$\mathbb{K}(A,B)$ is stably fibrant in $Sp$ for every $B\in\Re$.

}\end{example}

The stably fibrant spectra determine the stable weak equivalences of
spectra. Stable fibrations are maps having the right lifting
property with respect to all maps which are projective cofibrations
and stable weak equivalences.

\begin{defs}{\rm
A map $f\colon\cc E\rightarrow\cc F$ of spectra is a {\it stable
weak equivalence\/} if for every stably fibrant $\cc Z$ taking a
cofibrant replacement $Qf\colon Q\cc E\rightarrow Q\cc F$ of $f$ in
the level projective model structure on $\Spt$ yields a weak
equivalence of pointed simplicial sets
$$
\xymatrix{\Map_{\Spt}(Qf,\cc Z)\colon \Map_{\Spt}(Q\cc F,\cc
Z)\ar[r]&\Map_{\Spt}(Q\cc E,\cc Z).}
$$
}\end{defs}

By specializing the collection of results in \cite{H,Sch} to our
setting we have:

\begin{thm}\label{stablemodel}
The classes of stable weak equivalences and projective cofibrations
define a simplicial and left proper model structure on $\Spt$.
\end{thm}

If we define the stable model category structure on ordinary spectra
$Sp$ similar to $\Spt$, then by~\cite[3.5]{H} it coincides with the
stable model structure of Bousfield--Friedlander~\cite{BF}.

Define the \emph{shift functors} $t:\Spt\xrightarrow{}\Spt$ and
$s:\Spt\xrightarrow{}\Spt$ by $(s\cc X)_{n}=\cc X_{n+1}$ and $(t\cc
X)_{n}=\cc X_{n-1}$, $(t\cc X)_{0}=pt$, with the evident structure
maps. Note that $t$ is left adjoint to $s$.

\begin{defs}{\rm
Define $\Theta:\Spt\to\Spt$ to be the functor $s\circ\Omega^\ell$,
where $s$ is the shift functor.  Then we have a natural map
$\iota_{\cc X}:\cc X\to\Theta\cc X$, and we define
\[
\Theta^{\infty}\cc X=\colim(\cc X\xrightarrow{\iota_{\cc
X}}\Theta\cc X\xrightarrow{\Theta\iota_{\cc X}}\Theta^{2}\cc
X\xrightarrow{\Theta^{2}\iota_{\cc X}}\cdots
\xrightarrow{\Theta^{n-1}\iota_{\cc X}}\Theta^{n}\cc
X\xrightarrow{\Theta^{n}\iota_{\cc X}}\cdots).
\]
Let $j_{\cc X}:\cc X\to\Theta^{\infty}\cc X$ denote the obvious
natural transformation. It is a stable equivalence
by~\cite[4.11]{H}. We call $\Theta^{\infty}\cc X$ the {\it stabilization\/} of $\cc X$.
}\end{defs}

\begin{example}{\rm
Given $A\in\Re$, there is a natural map of spectra
   $$\kappa:\cc R(A)\to\mathbb{K}(A,-).$$
One has a commutative diagram
   $$\xymatrix{\cc R(A)\ar[r]^j\ar[d]_\kappa&\Theta^\infty\cc R(A)\ar[d]^{\Theta^\infty\kappa}\\
               \mathbb{K}(A,-)\ar[r]^j&\Theta^\infty\mathbb{K}(A,-).}$$
The upper horizontal map is a stable equivalence, the lower and
right arrows are isomorphisms. Therefore the natural map of spectra
$\kappa:\cc R(A)\to\mathbb{K}(A,-)$ is a stable equivalence. In fact
for any algebra $B\in\Re$ the map
   $$\kappa_B:\cc R(A)(B)\to\mathbb{K}(A,B)$$
is a stable equivalence of ordinary spectra.

}\end{example}

By~\cite[4.6]{H} we get the following result because $\Omega(-)$
preserves sequential colimits and the model category
$U\Re^{I,J}_\bullet$ is almost finitely generated.

\begin{lem}\label{stabiliz}
The stabilization of every level fibrant spectrum is stably fibrant.
\end{lem}

\begin{lem}\label{stabild}
For any $D\in\Re$ the adjoint functors
$F_D:Sp\rightleftarrows\Spt:U_D$ form a Quillen adjunction between
the stable model category of Bousfield--Friedlander spectra $Sp$ and
the stable model category $\Spt$.
\end{lem}

\begin{proof}
Clearly, $F_D$ preserves stable cofibrations. To show that $F_D$
preserves stable trivial cofibrations, it is enough to observe that
$U_D$ preserves stable fibrant spectra (see the proof
of~\cite[3.5]{H}) and use~\eqref{tutu}.
\end{proof}

We are now in a position to prove the main result of this section.

\begin{thm}\label{mainstab}
Suppose $F\rightarrowtail B\twoheadrightarrow C$ is an $\ff
F$-extension in $\Re$. Then the commutative square of spectra
   $$\xymatrix{\bb K(C,-)\ar[r]\ar[d]&\bb K(B,-)\ar[d]\\
               pt\ar[r]&\bb K(F,-)}$$
is homotopy pushout and homotopy pullback in $\Spt$. Moreover, if
$D\in\Re$ then the square of ordinary spectra
   $$\xymatrix{\bb K(C,D)\ar[r]\ar[d]&\bb K(B,D)\ar[d]\\
               pt\ar[r]&\bb K(F,D)}$$
is homotopy pushout and homotopy pullback.
\end{thm}

\begin{proof}
Given a distinguished square $\alpha$
   $$\xymatrix{D\ar[r]\ar[d]&A\ar[d]\\
               B\ar[r]&C}$$
in $\Re$, the square $r\alpha_+$
   $$\xymatrix{rC_+\ar[r]\ar[d]&rA_+\ar[d]\\
               rB_+\ar[r]&rD_+}$$
is homotopy pushout in $U\Re^{I,J}_\bullet$.

We claim that there is a $J$-weak equivalence of pointed simplicial
functors $rA_+\to rA$ for any algebra $A\in\Re$. The object $rA$ is
a cofibre product of the diagram
   $$pt\leftarrow r0_+\to rA_+,$$
in which the right arrow is an injective cofibration. It follows
that for every pointed fibrant object $\cc X$ in $U\Re_{J,\bullet}$
the sequence of simplicial sets
   $$\Map_{U\Re_\bullet}(rA,\cc X)\to\cc X(A)\to\cc X(0)$$
is a homotopy fibre sequence with $\cc X(0)$ contractible. Hence the
left arrow is a weak equivalence of simplicial sets, and so the map
of pointed simplicial functors $rA_+\to rA$ is a $J$-weak
equivalence. Using~\cite[13.5.9]{Hir} the square $r\alpha$ with
$\alpha$ as above
   $$\xymatrix{rC\ar[r]\ar[d]&rA\ar[d]\\
               rB\ar[r]&rD}$$
is homotopy pushout in $U\Re^{I,J}_\bullet$.

Given an algebra $A\in\Re$ and $n\geq 0$, there is an $I$-weak
equivalence of simplicial functors pointed at zero
$i_{J^nA}:r(J^nA)\to Sing(r(J^nA))$. By Theorem~\ref{exinfi}
   $$\cc R(A)_n=Ex^\infty\circ Sing(r(J^nA)).$$
Since the map
   $$\xi_\upsilon:JA\to\Omega A,$$
which is functorial in $A$, is a quasi-isomorphism, then the square
   \begin{equation}\label{tyty}
    \xymatrix{r(J^nC)\ar[r]\ar[d]&r(J^nB)\ar[d]\\
               pt\ar[r]&r(J^nF)}
   \end{equation}
is weakly equivalent to the homotopy pushout square in
$U\Re_\bullet^{I,J}$
   $$\xymatrix{r(\Omega^nC)\ar[r]\ar[d]&r(\Omega^nB)\ar[d]\\
               pt\ar[r]&r(\Omega^nF)}$$
By~\cite[13.5.9]{Hir} square~\eqref{tyty} is then homotopy pushout
in $U\Re_\bullet^{I,J}$. Also, \cite[13.5.9]{Hir} implies that
   $$\xymatrix{Sing(r(J^nC))\ar[r]\ar[d]&Sing(r(J^nB))\ar[d]\\
               pt\ar[r]&Sing(r(J^nF))}$$
is homotopy pushout in $U\Re_\bullet^{I,J}$, and hence so is
   $$\xymatrix{\cc R(C)_n\ar[r]\ar[d]&\cc R(B)_n\ar[d]\\
               pt\ar[r]&\cc R(F)_n.}$$
We see that the square of spectra
   \begin{equation}\label{push}
   \xymatrix{\cc R(C)\ar[r]^u\ar[d]&\cc R(B)\ar[d]\\
               pt\ar[r]&\cc R(F)}
   \end{equation}
is level pushout. We can find a projective cofibration of spectra
$\iota:\cc R(C)\to\cc X$ and a level weak equivalence $s:\cc X\to\cc
R(B)$ such that $u=s\iota$. Consider a pushout square
   $$\xymatrix{\cc R(C)\ar[r]^\iota\ar[d]&\cc X\ar[d]\\
               pt\ar[r]&\cc Y.}$$
It is homotopy pushout in the projective model structure of spectra,
and hence it is levelwise homotopy pushout in $U\Re_\bullet^{I,J}$.
Therefore the induced map $\cc Y\to\cc R(F)$ is a level weak
equivalence, and so \eqref{push} is homotopy pushout in the
projective model structure of spectra by~\cite[13.5.9]{Hir}.

Since the vertical arrows in the commutative diagram
   $$\xymatrix@!0{{pt}\ar[rrrrr] &&&&& \mathbb{K}(F,-)\\
                  && \mathbb{K}(C,-)\ar[ull]\ar[rrrrr] &&&&&\mathbb{K}(B,-)\ar[ull]\\
                  {pt} \ar'[rr][rrrrr] \ar[uu] &&&&& \cc R(F)\ar'[u][uu]\\
                  &&\cc R(C)\ar[ull]\ar[uu]\ar[rrrrr]&&&&&\cc R(B)\ar[ull]\ar[uu]^\kappa}$$
are stable weak equivalences and the lower square is homotopy
pushout in the stable model structure of spectra, then so is the
upper square by~\cite[13.5.9]{Hir}. By~\cite[3.9; 10.3]{H} $\Spt$ is
a stable model category with respect to the stable model structure,
and therefore the square of the theorem is also homotopy pullback
by~\cite[7.1.12]{Hov}.

It follows from Lemma~\ref{stabild} that the square of simplicial
spectra
   $$\xymatrix{\bb K(C,D)\ar[r]\ar[d]&\bb K(B,D)\ar[d]\\
               pt\ar[r]&\bb K(F,D)}$$
is homotopy pullback for all $D\in\Re$. It is also homotopy pushout
in the stable model category of Bousfield--Friedlander spectra
by~\cite[7.1.12]{Hov}, because this model structure is stable.
\end{proof}

It is also useful to have the following

\begin{thm}\label{mainequiv}
Suppose $u:A\to B$ is a quasi-isomorphism in $\Re$. Then the induced
map of spectra
   $$u^*:\bb K(B,-)\to\bb K(A,-)$$
is a stable equivalence in $\Spt$. In particular, the map of spaces
   $$u^*:\cc K(B,C)\to\cc K(A,C)$$
is a weak equivalence for all $C\in\Re$.
\end{thm}

\begin{proof}
Consider the square in $U\Re_\bullet$
   $$\xymatrix{rB\ar[r]^{u^*}\ar[d]&rA\ar[d]\\
               \cc R(B)_0\ar[r]^{u^*}&\cc R(A)_0.}$$
The upper arrow is an $(I,J)$-weak equivalence, the vertical maps
are $I$-weak equivalences. Therefore the lower arrow is an
$(I,J)$-weak equivalence.

Since the endofunctor $J:\Re\to\Re$ respects quasi-isomorphisms,
then
   $$u^*:\cc R(B)\to\cc R(A)$$
is a level weak equivalence of spectra. Consider the square in
$\Spt$
   $$\xymatrix{\cc R(B)\ar[r]^{u^*}\ar[d]_\kappa&\cc R(A)\ar[d]^\kappa\\
               \mathbb{K}(B,-)\ar[r]^{u^*}&\mathbb{K}(A,-).}$$
The upper arrow is a level weak equivalence, the vertical maps are
stable weak equivalences. Therefore the lower arrow is a stable weak
equivalence.

The map $\mathbb{K}(B,-)\lra{u^*}\mathbb{K}(A,-)$ is a weak
equivalence in the projective model structure on $Sp(\Re)$, because
both spectra are stably fibrant and levelwise fibrant in
$U\Re_{\bullet}^{I,J}$. It follows that the map of spaces
   $$u^*:\cc K(B,C)\to\cc K(A,C)$$
is a weak equivalence for all $C\in\Re$.
\end{proof}

We can now prove Excision
Theorem~B.\label{pageexc}\renewcommand{\proofname}{Proof of Excision
Theorem B}

\begin{proof}
Let $\Re$ be an arbitrary admissible $T$-closed category of
$k$-algebras. We have to prove that the square of spaces
   $$\xymatrix{\cc K(C,D)\ar[r]\ar[d]&\cc K(B,D)\ar[d]\\
               pt\ar[r]&\cc K(F,D)}$$
is homotopy pullback for any extension $F\rightarrowtail
B\twoheadrightarrow C$ in $\Re$ and any algebra $D\in\Re$.

A subtle difference with what we have defined for spectra is that we
do not assume $\Re$ to be small. So to apply Theorem~\ref{mainstab}
one has to find a small admissible $T$-closed category of
$k$-algebras $\Re'$ containing $F,B,C,D$.

We can inductively construct such a category as follows. Let
$\Re'_0$ be the full subcategory of $\Re$ such that
$\Ob\Re_0'=\{F,B,C,D\}$. If the full subcategory $\Re_n'$ of $\Re$,
$n\geq 0$, is constructed we define $\Re_{n+1}'$ by adding  the
following algebras to $\Re_{n}'$:

\begin{itemize}
\item[$\vartriangleright$] all ideals and quotient algebras of algebras from
$\Re_n'$;

\item[$\vartriangleright$] all algebras which are pullbacks for diagrams
   $$A\to E\leftarrow L$$
with $A,E,L\in\Re_n'$;

\item[$\vartriangleright$] all polynomial algebras in one
variable $A[x]$ with $A\in\Re_n'$;

\item[$\vartriangleright$] all algebras $TA$ with $A\in\Re_n'$.
\end{itemize}
Then we set $\Re'=\bigcup_n\Re_n'$. Clearly $\Re'$ is a small
admissible $T$-closed category of algebras containing $F,B,C,D$. It
remains to apply Theorem~\ref{mainstab}.
\end{proof}
\renewcommand{\proofname}{Proof}

We can now also prove
Theorem~\ref{spectrumunstb}.\renewcommand{\proofname}{Proof of
Theorem~\ref{spectrumunstb}}

\begin{proof}
Let $\Re$ be an arbitrary admissible $T$-closed category of
$k$-algebras. We have to prove that the square of spectra
   $$\xymatrix{\bb K(C,D)\ar[r]\ar[d]&\bb K(B,D)\ar[d]\\
               pt\ar[r]&\bb K(F,D)}$$
is homotopy pullback for any extension $F\rightarrowtail
B\twoheadrightarrow C$ in $\Re$ and any algebra $D\in\Re$.

To apply Theorem~\ref{mainstab} one has to find a small admissible
$T$-closed category of $k$-algebras $\Re'$ containing $F,B,C,D$.
Such a category is constructed in the proof of Excision Theorem~B.
\end{proof}
\renewcommand{\proofname}{Proof}

\begin{cor}\label{suspen}
Let $\Re$ be an admissible $T$-closed category of $k$-algebras. Then
for every $A,B\in\Re$ the spectrum $\bb K(JA,B)$ has homotopy type
of $\Sigma\bb K(A,B)$.
\end{cor}

\begin{proof}
We have an extension $JA\rightarrowtail TA\twoheadrightarrow A$ in
which $TA$ is contractible by Lemma~\ref{ta}. Hence $\bb
K(TA,B)\simeq *$ by Theorem~\ref{mainequiv} (as above one can choose
a small admissible $T$-closed category of algebras such that all
considered algebras belong to it). Now our assertion follows from
Excision Theorem B.
\end{proof}

\section{Comparison Theorem A}\label{compa}

In this section we prove a couple of technical (but important!)
results giving a relation between simplicial and polynomial homotopy
for algebra homomorphisms. As an application, we prove Comparison
Theorem~A. Throughout $\Re$ is supposed to be $T$-closed.

\subsection{Categories of fibrant objects}

\begin{defs}{\rm
Let $\cc A$ be a category with finite products and a final object
$e$. Assume that $\cc A$ has two distinguished classes of maps,
called {\it weak equivalences\/} and {\it fibrations}. A map is
called a {\it trivial fibration\/} if it is both a weak equivalence
and a fibration. We define a {\it path space\/} for an object $B$ to
be an object $B^I$ together with maps
   $$B\lra{s}B^I\xrightarrow{(d_0,d_1)}B\times B,$$
where $s$ is a weak equivalence, $(d_0,d_1)$ is a fibration, and the
composite is the diagonal map.

Following Brown~\cite{B}, we call $\cc A$ a {\it category of fibrant
objects\/} or a {\em Brown category\/} if the following axioms are
satisfied.

(A) Let $f$ and $g$ be maps such that $gf$ is defined. If two of
$f$, $g$, $gf$ are weak equivalences then so is the third. Any
isomorphism is a weak equivalence.

(B) The composite of two fibrations is a fibration. Any isomorphism
is a fibration.

(C) Given a diagram
   $$A\bl u\longrightarrow C\bl v\longleftarrow B,$$
with $v$ a fibration (respectively a trivial fibration), the
pullback $A\times_CB$ exists and the map $A\times_CB\to A$ is a
fibration (respectively a trivial fibration).

(D) For any object $B$ in $\cc A$ there exists at least one path
space $B^I$ (not necessarily functorial in $B$).

(E) For any object $B$ the map $B\to e$ is a fibration.

}\end{defs}

\subsection{The Hauptlemma}

Every map $u$ in $\Re$ can be factored $u=pi$, where $p\in\ff F$ is
a fibration and $i$ is an $I$-weak equivalence. Indeed, let $A'$ be
the the fibre product of the diagram
   $$A\xrightarrow u B\xleftarrow{\partial^0_x} B[x].$$
Then the map $i:A\to A'$, $a\longmapsto(a,u(a))$, is split
 and obviously an elementary homotopy equivalence.
Hence it is an $I$-weak equivalence. We define $p:A'\to B$ as
composition of the projection $A'\to B[x]$ and $\partial^1_x$. We
call a homomorphism an {\it $I$-trivial fibration\/} if it is both a
fibration and an $I$-weak equivalence. We denote by $I^n$, $n\geq
0$, the simplicial set $\Delta^1\times\bl n\cdots\times\Delta^1$ and
by $\delta^0,\delta^1:I^n\to I^{n+1}$ the maps $1_{I^n}\times
d^0,1_{I^n}\times d^1$ whose images are
$I^n\times\{1\},I^n\times\{0\}$ respectively.

Let $\ff W_{\min}$ be the minimal class of weak equivalences
containing the homomorphisms $A\to A[t]$, $A\in\Re$, such that the
triple $(\Re,\ff F,\ff W_{\min})$ is a Brown category. We should
mention that every excisive, homotopy invariant simplicial functor
$\cc X:\Re\to SSets$ gives rise to a class of weak equivalences $\ff
W$ containing the homomorphisms $A\to A[t]$, $A\in\Re$, such that
the triple $(\Re,\ff F,\ff W)$ is a Brown category (see~\cite{Gar}).
Precisely, $\ff W$ consists of those homomorphisms $f$ for which
$\cc X(f)$ is a weak equivalence of simplicial sets.

We shall denote by $B^{\ff S^n}$, $n\geq 0$, the ind-algebra
   $$B^{\ff S^n}_0\to B^{\ff S^n}_1\to B^{\ff S^n}_2\to\cdots$$
consisting of the 0-simplices of the simplicial ind-algebra $\bb
B(\Omega^n)$. So we have for any $k\geq 0$:
   $$B^{\ff S^n}_k=\kr(B^{\sd^k I^n}\to B^{\sd^k(\partial I^n)}).$$
One also sets
   $$\wt B^{\ff S^n}_k=\kr(B^{\sd^k I^{n+1}}\to B^{\sd^k(\partial I^{n}\times I)}).$$

\begin{hauptlemma}\label{haupt}
Let $A,B\in\Re$ then for any $m,n\geq 0$ we have:
\begin{enumerate}
\item If $f:A\to B^{\sd^m\Delta^{n+1}}$ is a homomorphism, then the
homomorphism $\partial_if$ is algebraically homotopic to
$\partial_jf$ with $i,j\leq n+1$.

\item If $f:A\to B^{\sd^m I^{n+1}}$ (respectively $f:A\to\wt B^{\ff S^n}_m$) is a homomorphism, then
$d_0f,d_1f:A\to B^{\sd^m I^n}$ (respectively $d_0f,d_1f:A\to B^{\ff
S^n}_m$) are algebraically homotopic, where $d_0,d_1:B^{\sd^m
I^{n+1}}\to B^{\sd^m I^n}$ are induced by $\delta^0,\delta^1$.

\item If $f_0,f_1:A\to B^{\sd^m I^{n}}$ (respectively $f_0,f_1:A\to B^{\ff S^n}_m$)
are two algebraically homotopic homomorphisms by means of a map
$h:A\to(B^{\sd^m I^{n}})^{\sd^k\Delta^1}$ (respectively $h:A\to
(B^{\ff S^n}_m)^{\sd^k\Delta^1}$), then there are a homomorphism
$g:A'\to A$, which is obtained by pulling back an $I$-trivial
fibration along $h$, and hence $g\in\ff W_{\min}$, and a
homomorphism $H:A'\to B^{\sd^m I^{n+1}}$ (respectively $H:A'\to\wt
B^{\ff S^n}_m$) such that $d_0H=f_0g$ and $d_1H=f_1g$.
\end{enumerate}
\end{hauptlemma}

The Hauptlemma essentially says that the condition for homomorphisms
of being simplicially homotopic implies that of being polynomially
homotopic. The converse is true up to multiplication with some maps
from $\ff W_{\min}$.

\begin{proof}
(1). Given $i<j$ define a homomorphism
$\phi_{i,j}:B[t_0,\ldots,t_{n+1}]\to B^{\Delta^n}[x]$ as
   \begin{equation}\label{homot1}
    \phi_{i,j}(t_k)=
     \left\{
      \begin{array}{rcl}
       t_k,&\ k<i\\
       xt_i,&\ k=i\\
       xt_k+(1-x)t_{k-1},&\ i<k<j\\
       (1-x)t_{j-1},&\ k=j\\
       t_{k-1},&\ k>j
      \end{array}
      \right.
   \end{equation}
It takes $1-\sum_{i=0}^{n+1}t_i$ to zero, and hence one obtains a
homomorphism $\phi_{i,j}:B^{\Delta^{n+1}}\to B^{\Delta^{n}}[x]$. We
define
$\phi_{j,i}(f)(t_0,\ldots,t_n,x)=\phi_{i,j}(f)(t_0,\ldots,t_n,1-x)$
if $j>i$. It follows that for any $h\in B^{\Delta^{n+1}}$
   $$\phi_{i,j}(h)(t_0,\ldots,t_n,x)=\left\{\begin{aligned}
      \partial_ih,&\quad x=0,\\
      \partial_{j}h,&\quad x=1.
     \end{aligned}\right.$$
We see that $\partial_i\alpha$ is elementary homotopic to
$\partial_j\alpha$ for any $\alpha:A\to B^{\Delta^{n+1}}$. If there
is no likelihood of confusion we shall denote this homotopy by
$\phi_{i,j}$ omitting $\alpha$.

Now consider the algebra $B^{\sd^k\Delta^{n+1}}$, $k\geq 1$. By
definition, it is the fiber product over $B^{\Delta^n}$ of
$((n+2)!)^k$ copies of $B^{\Delta^{n+1}}$. Let $\alpha:A\to
B^{\sd^k\Delta^{n+1}}$ be a homomorphism of algebras. A polynomial
homotopy from $\partial_{j}\alpha$ to $\partial_{i}\alpha$ can be
arranged as follows. We pick up the barycenter of
$\partial_{j}\alpha$ and pull it towards the barycenter of $\alpha$.
This operation consists of finitely many polynomial homotopies. Next
we pull the vertex $i$ towards the vertex $j$. Again we have
finitely many elementary polynomial homotopies. Finally, we pull the
barycenter of $\alpha$ towards the barycenter of
$\partial_{i}\alpha$, resulting the desired polynomial homotopy.
Each step of the polynomial homotopy is determined by homotopies of
the form $\phi_{i,i+1}$ or $\phi_{i+1,i}$.

In order to give a formal description of the algorithm, it is enough
to do this for the identity homomorphism of $B^{\sd^k\Delta^{n+1}}$
without loss of generality. Recall that $\sd^k\Delta^{n+1}$ is the
nerve $|\ff{sd}^k\Delta^{n+1}|$ of a poset $\ff{sd}^k\Delta^{n+1}$.
We shall also regard the poset as a category. By the length of a
path in $\ff{sd}^k\Delta^{n+1}$ between two vertices we mean the
maximal number of (non-identity) arrows connecting them. Note that
the largest (respectively smallest) possible length equals $n+1$
(respectively 0). We can think about $B^{\sd^k\Delta^{n+1}}$ in the
following terms:
\begin{itemize}
\item[$\diamond$] the poset $\ff{sd}^k\Delta^{n+1}$;
\item[$\diamond$] to every path $\pi$ of length $k\geq 0$ in $\ff{sd}^k\Delta^{n+1}$,
a polynomial in $f_\pi\in B^{\Delta^k}$ is attached;
\item[$\diamond$] if $f_{\pi},f_{\pi'}$ are two polynomials attached to
paths $\pi,\pi'$ in $\ff{sd}^k\Delta^{n+1}$ having a non-trivial
common subpath $\pi''$, then restriction of $f_{\pi},f_{\pi'}$ to
$\pi''$ by means of face operators of $B^\Delta$ equals $f_{\pi''}$.
\end{itemize}
The desired algorithm of an algebraic homotopy from
$\partial_{j}\alpha$ to $\partial_{i}\alpha$ consists of the
following steps (we shall sometimes regard
$\partial_{j}\alpha,\partial_{i}\alpha$ as posets uniquely
associated to them):
\begin{itemize}
\item[Step 0:] The poset $A_0:=\partial_{j}\alpha$, the vertex $w_0:=i$,
and the poset $(\ff{sd}^k\Delta^{n+1})_0:=\ff{sd}^k\Delta^{n+1}$.

\item[Step 1:] Suppose we have found a subposet $A_\ell$, $\ell\geq 0$, such that the
vertices of the subposet $\partial_{j}\alpha\cap\partial_{i}\alpha$
are in $A_\ell$, $A_\ell$ is isomorphic to $A_0$ by an isomorphism
of posets leaving the vertices of
$\partial_{j}\alpha\cap\partial_{i}\alpha$ unchanged, and $A_\ell$
has a unique vertex $w_\ell$ on the edge between vertices $i$ and
$j$. If $A_\ell=(\ff{sd}^k\Delta^{n+1})_\ell=\partial_{i}\alpha$, in
which case $w_\ell=j$, we stop the process. Otherwise, we can find a
vertex $v\in
A_\ell\setminus(\partial_{j}\alpha\cap\partial_{i}\alpha)$, a
subposet $A_{\ell+1}$ having the same vertices as $A_\ell$ except a
unique vertex $v'\in
A_{\ell+1}\setminus(\partial_{j}\alpha\cap\partial_{i}\alpha)$, and
a path $v\to v'$ or a path $v'\to v$ of length 1. Having chosen
$A_{\ell+1}$, we set
$(\ff{sd}^k\Delta^{n+1})_{\ell+1}=(\ff{sd}^k\Delta^{n+1})_\ell\setminus\{v\}$.
Notice that $w_{\ell+1}\in A_{\ell+1}$ is either $v'$ or $w_\ell$.

\item[Step 2:] If we have a path $v'\to v$ from Step~1, then there
are a unique $i\in\{0,1,\ldots,n\}$ and a polynomial homotopy
$H_\ell:B^{{sd}^k\Delta^{n+1}}\to B^{{sd}^k\Delta^{n}}[x]$ such
that:

\begin{itemize}
\item  $\partial_x^0 H_\ell$ equals the homomorphism
$B^{{sd}^k\Delta^{n+1}}\to B^{|A_\ell|}$, induced by the inclusion
$A_\ell\hookrightarrow\ff{sd}^k\Delta^{n+1}$, and $\partial_x^1
H_\ell$ equals the homomorphism $B^{{sd}^k\Delta^{n+1}}\to
B^{|A_{\ell+1}|}$;

\item $H_\ell$ factors as
   $$B^{{sd}^k\Delta^{n+1}}\to B^{|E_\ell|}\xrightarrow{h_\ell} B^{{sd}^k\Delta^{n}}[x],$$
where $E_\ell$ is the subposet of $(\ff{sd}^k\Delta^{n+1})_{\ell}$
whose vertices are those of $A_\ell\cup A_{\ell+1}$ and the left
homomorphism is induced by the inclusion
$E_\ell\hookrightarrow\ff{sd}^k\Delta^{n+1}$;

\item $h_\ell$ is a fibre product of homomorphisms of the
form $i:B^{\Delta^n}\hookrightarrow B^{\Delta^n}[x]$ or
$\phi_{i,i+1}:B^{\Delta^{n+1}}\to B^{\Delta^n}[x]$.
\end{itemize}
If we have a path $v\to v'$ of Step~1, then $H_\ell$ has the same
properties as above except that $h_\ell$ is a fibre product of
homomorphisms of the form $i:B^{\Delta^n}\hookrightarrow
B^{\Delta^n}[x]$ or $\phi_{i+1,i}:B^{\Delta^{n+1}}\to
B^{\Delta^n}[x]$.

\item[Step 3:] Apply Step 1 to the pair $(A_{\ell+1},(\ff{sd}^k\Delta^{n+1})_{\ell+1})$.
\end{itemize}

The algorithm stops in finitely many steps. The desired polynomial
homotopy is given by the collection $\{H_\ell\}_{\ell}$. The
homotopy pulls the barycenter of $\partial_{j}\alpha$ towards the
barycenter of $\alpha$. Also it pulls the vertex $i$ towards $j$ by
means of $\{w_\ell\}_{\ell}$. Let us illustrate the algorithm by
considering for simplicity the algebra $B^{\sd^1\Delta^{2}}$, which
is glued out of six polynomials in $B^{\Delta^{2}}$.
  $$
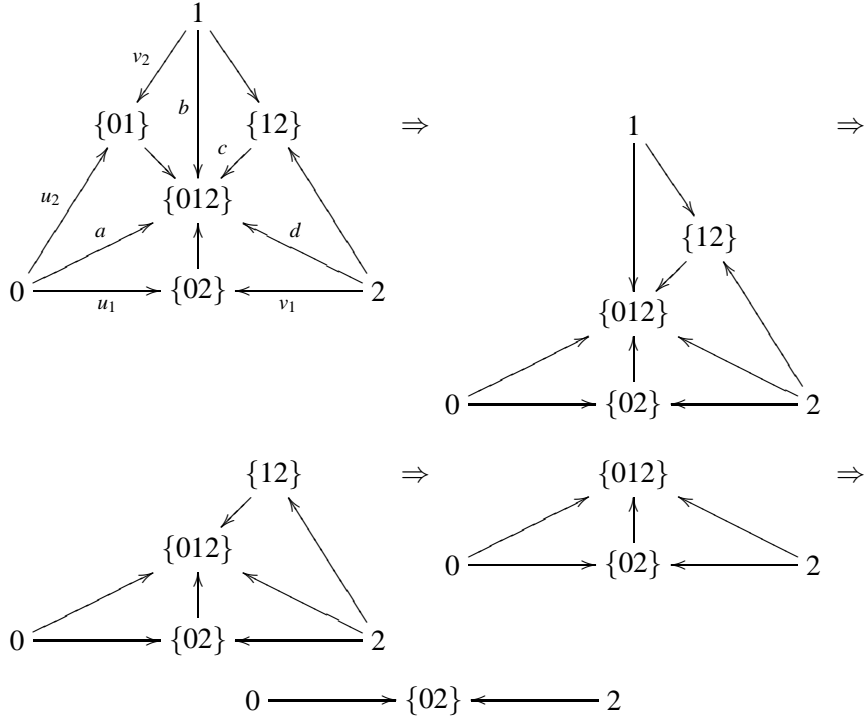
\xy
    *+{\{01\}}="\{01\}",<1cm,1.5cm>*+{1}="1",-<-1cm,1.5cm>*+{\{12\}}="\{12\}",
    -<1cm,1cm>*+{\{012\}}="\{012\}",
    -<0cm,1.2cm>*+{\{02\}}="\{02\}",-<-2.37cm,0cm>*+{2}="2",-<4.75cm,0cm>*+{0}="0",
    \ar_{v_2} "1";"\{01\}",\ar "1";"\{12\}",\ar_b "1";"\{012\}",\ar "\{01\}";"\{012\}",\ar_c "\{12\}";"\{012\}",
    \ar_{u_1} "0";"\{02\}",\ar^{v_1} "2";"\{02\}",\ar^{u_2} "0";"\{01\}",\ar "\{02\}";"\{012\}",\ar "2";"\{12\}",
    \ar^a "0";"\{012\}",\ar_d "2";"\{012\}"
   \endxy\Rightarrow
   \xy
    *+{1}="1",-<-1cm,1.5cm>*+{\{12\}}="\{12\}",
    -<1cm,1cm>*+{\{012\}}="\{012\}",
    -<0cm,1.2cm>*+{\{02\}}="\{02\}",-<-2.37cm,0cm>*+{2}="2",-<4.75cm,0cm>*+{0}="0",
    \ar "1";"\{12\}",\ar "1";"\{012\}",\ar "\{12\}";"\{012\}",
    \ar "0";"\{02\}",\ar "2";"\{02\}",\ar "\{02\}";"\{012\}",\ar "2";"\{12\}",
    \ar "0";"\{012\}",\ar "2";"\{012\}"
   \endxy\Rightarrow$$

  $$\xy
    *+{\{12\}}="\{12\}",
    -<1cm,1cm>*+{\{012\}}="\{012\}",
    -<0cm,1.2cm>*+{\{02\}}="\{02\}",-<-2.37cm,0cm>*+{2}="2",-<4.75cm,0cm>*+{0}="0",
    \ar "\{12\}";"\{012\}",
    \ar "0";"\{02\}",\ar "2";"\{02\}",\ar "\{02\}";"\{012\}",\ar "2";"\{12\}",
    \ar "0";"\{012\}",\ar "2";"\{012\}"
   \endxy\Rightarrow
   \xy
    *+{\{012\}}="\{012\}",
    -<0cm,1.2cm>*+{\{02\}}="\{02\}",-<-2.37cm,0cm>*+{2}="2",-<4.75cm,0cm>*+{0}="0",
    \ar "0";"\{02\}",\ar "2";"\{02\}",\ar "\{02\}";"\{012\}",
    \ar "0";"\{012\}",\ar "2";"\{012\}"
   \endxy\Rightarrow$$
   $$\xy
    *+{\{02\}}="\{02\}",-<-2.37cm,0cm>*+{2}="2",-<4.75cm,0cm>*+{0}="0",
    \ar "0";"\{02\}",\ar "2";"\{02\}"
   \endxy$$
The picture depicts each poset $(\ff{sd}^1\Delta^{2})_\ell$. It also
says that\footnotesize
   \begin{gather*}
     \partial_2\alpha=(0\lra{u_2}\{01\}\bl{v_2}\longleftarrow 1)\bl{H_0}\sim (0\lra{a}\{012\}\bl{b}\longleftarrow 1,w_1=1,v'=\{012\})
     \bl{H_1}\sim (0\lra{a}\{012\}\bl{c}\longleftarrow\{12\},w_2=v'=\{12\})\bl{H_2}\sim\\
     \sim (0\lra{a}\{012\}\bl{d}\longleftarrow 2,w_3=v'=2)\bl{H_3}\sim(0\lra{u_1}\{02\}\bl{v_1}\longleftarrow
     2,w_4=2,v'=\{02\})=\partial_1\alpha
   \end{gather*}
\normalsize and $\{h_0=(\phi_{2,1},\phi_{2,1}),E_0=\langle
0,\{01\},1,\{012\}\rangle\}$, $\{h_1=(i,\phi_{1,0}),E_1=\langle
0,\{012\},1,\{12\}\rangle\}$, $\{h_2=(i,\phi_{1,2}),E_2=\langle
0,\{012\},\{12\},2\rangle\}$,
$\{h_3=(\phi_{1,2},\phi_{1,2}),E_2=\langle
0,\{012\},\{02\},2\rangle\}$.

Another example is for the algebra $B^{\sd^1\Delta^3}$, which is
glued out of 24 polynomials in $B^{\Delta^3}$. Consider a
tetrahedron labeled with $0,1,2,3$
   $$
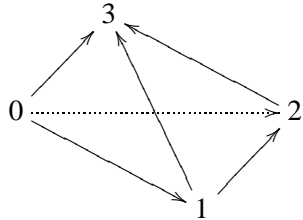
\xymatrix{&3\\
            0\ar[drr]\ar[ur]\ar@{.>}[rrr]&&&2\ar[ull]\\
            &&1\ar[ur]\ar[uul]}$$
The polynomial homotopy from $\partial_2\alpha$ to
$\partial_3\alpha$ is encoded by the following data of the
algorithm:
\begin{itemize}
\item$\diamond$ $v=\{013\}$ and $v'=\{0123\}$;
\item$\diamond$ $v=\{13\}$ and $v'=\{123\}$;
\item$\diamond$ $v=\{03\}$ and $v'=\{023\}$;
\item$\diamond$ $v=3$ and $v'=\{23\}$;
\item$\diamond$ $v=\{23\}$ and $v'=2$;
\item$\diamond$ $v=\{123\}$ and $v'=\{12\}$;
\item$\diamond$ $v=\{023\}$ and $v'=\{02\}$
\item$\diamond$ finally, $v=\{0123\}$ and
$v'=\{012\}$.
\end{itemize}

(2). Note that this assertion follows from the particular case when
$f$ is the identity map of its codomain. Nevertheless we shall prove
the original statement. The cube $I^{n+1}$ is glued out of $(n+1)!$
simplices of dimension $n+1$. Its vertices can be labeled with
$(n+1)$-tuples of numbers which equal either zero or one. The number
of vertices equals $2^{n+1}$. A homomorphism $\alpha:A\to
B^{I^{n+1}}$ is glued out of $(n+1)!$ homomorphisms $\alpha_i:A\to
B^{\Delta^{n+1}}$. The set of vertices $V_{d_0\alpha}$ of the face
$d_0\alpha$ consists of those $(n+1)$-tuples whose last coordinate
equals 1, and the set of vertices $V_{d_1\alpha}$ of the face
$d_1\alpha$ consists of those $(n+1)$-tuples whose last coordinate
equals 0. The desired algebraic homotopy from $d_0\alpha$ to
$d_1\alpha$ is constructed in the following way. We first construct
an algebraic homotopy $H_0$ from $f_0:=d_0\alpha$ to a homomorphism
$f_1:A\to B^{I^n}$. The set of vertices $V_1$ for $f_1$ equals
$(V_{d_0\alpha}\setminus\{00\ldots01\})\cup\{00\ldots0\}$. In other
words, we pull $\{00\ldots01\}$ towards $\{00\ldots0\}$. The number
of $(n+1)$-simplices having vertices from
$V_{d_0\alpha}\cup\{00\ldots0\}$ equals $n!$. Let $S$ be the set of
such simplices. If $\alpha_i:A\to B^{\Delta^{n+1}}$ is in $S$, then
the result is an algebraic homotopy $\phi_{0,1}$ defined in (1) from
$\partial_0\alpha_i$ to $\partial_1\alpha_i$. The homotopy $H_0$ at
each $\alpha_i$, $i\leq n!$, is $\phi_{0,1}$. Next one constructs an
algebraic homotopy $H_1$ from $f_1$ to $f_2:A\to B^{I^n}$. The set
of vertices $V_2$ of $f_2$ equals
$(V_1\setminus\{10\ldots01\})\cup\{10\ldots0\}$. In other word, we
pull $\{10\ldots01\}$ towards $\{10\ldots0\}$. The homotopy $H_1$ at
each simplex is either $\phi_{1,2}$ or $\id$. One repeats this
procedure $2^n$ times. The last step is to pull $(11\ldots 11)$
towards $(11\ldots 10)$ resulting a polynomial homotopy $H_{2^n-1}$
which is $\phi_{n,n+1}$ at each simplex. Clearly, if there are
boundary conditions as in (2) then the algebraic homotopy behaves on
the boundary in a consistent way.

In the case $\alpha:A\to B^{\sd^m I^{n+1}}$, $m>0$, the desired
polynomial homotopy is constructed in a similar way (we should also
use the algorithm of (1)).

Let us illustrate the algorithm by considering for simplicity the
case $\alpha:A\to B^{I^3}$. Such a map is glued out of six
homomorphisms $\alpha_i:A\to B^{\Delta^3}$, $i=1,\ldots,6$.
   $$
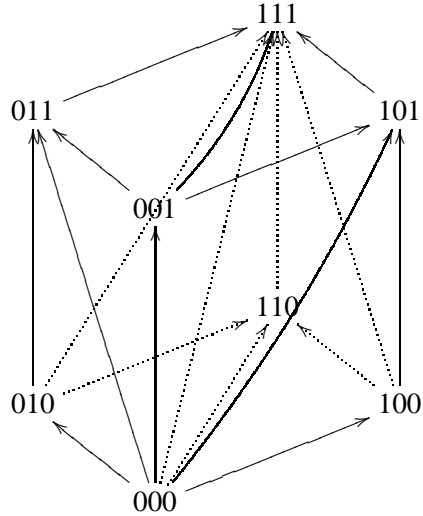
\xymatrix{&&111\\
     011\ar[urr]&&&101\ar[ul]\\
     &001\ar@/_/[uur]\ar[ul]\ar[urr]\\
     &&110\ar@{.>}[uuu]\\
     010\ar@{.>}[urr]\ar@{.>}[uuuurr]\ar[uuu]&&&100\ar[uuu]\ar@{.>}[ul]\ar@{.>}[uuuul]\\
     &000\ar@{.>}[uuuuur]\ar@{.>}[uur]\ar[uuu]\ar@/_/[uuuurr]\ar[ul]\ar[uuuul]\ar[urr]}$$
The desired algebraic homotopy from $d_0\alpha$ to $d_1\alpha$ is
arranged as follows. We first pull (001) towards (000) resulting a
polynomial homotopy $H_0$ from $d_0\alpha$, which is labeled by
$\{(001),(101),(011),(111)\}$, to the square labeled by
$\{(000),(101),(011),(111)\}$. This step is a result of the
algebraic homotopy $\phi_{0,1}$ described in (1) corresponding to
two glued tetrahedra having vertices $\{(000),(001),(011),(111)\}$
and $\{(000),(001),(101),(111)\}$ respectively. So
$H_0=(\phi_{0,1},\phi_{0,1})$. Next we pull (101) towards (100)
resulting a polynomial homotopy $H_1$ from the square labeled by
$\{(000),(101),(011),(111)\}$ to the square labeled by
$\{(000),(100),(011),(111)\}$. So $H_1=(\phi_{1,2},\id)$. The next
step is to pull (011) towards (010) resulting a polynomial homotopy
$H_2$ from the square labeled by $\{(000),(100),(011),(111)\}$ to
the square labeled by $\{(000),(100),(010),(111)\}$. So
$H_2=(\id,\phi_{1,2})$. And finally one pulls (111) towards (110)
resulting a polynomial homotopy $H_3$ from the square labeled by
$\{(000),(100),(010),(111)\}$ to the square labeled by
$\{(000),(100),(010),(110)\}$. In this case
$H_3=(\phi_{2,3},\phi_{2,3})$.

(3) We first want to prove the following statement.

\begin{hauptsub}
$B^{\sd^m I^{n+1}}$ (respectively $\wt B^{\ff S^n}_m$) is a path
space for $B^{\sd^m I^{n}}$ (respectively $B^{\ff S^n}_m$) in the
Brown category $(\Re,\ff F,\ff W_{\min})$.
\end{hauptsub}

\begin{proof}
By (2) the maps
   $$d_0,d_1:B^{\sd^m I^{n+1}}\to B^{\sd^m I^{n}}$$
are algebraically homotopic, hence equal in the category $\cc
H(\Re)$.

The map
   $$(d_0,d_1):B^{\sd^m I^{n+1}}\xrightarrow{}B^{\sd^m I^{n}}\times B^{\sd^m I^{n}}$$
is a $k$-linear split homomorphism, hence a fibration. A splitting
is defined as
   $$(b_1,b_2)\in B^{\sd^m I^{n}}\times B^{\sd^m I^{n}}\mapsto\imath(b_1)\cdot(1-\mathbf{t})+\imath(b_2)\cdot\mathbf{t}\in B^{\sd^m I^{n+1}},$$
where $\mathbf{t}\in k^{\sd^m I^{n+1}}$ is defined on
page~\pageref{elt} and $\imath:\bb B^\Delta(I^n)\to(\bb
B^\Delta(I^n))^{\Delta^1}$ is the natural inclusion. There is a
commutative diagram
   $$\xymatrix{B^{\sd^m I^{n}}\ar[rd]_{s}\ar[rr]^{diag}&&B^{\sd^m I^{n}}\times B^{\sd^m I^{n}}\\
               &B^{\sd^m I^{n+1}},\ar[ur]_{(d_0,d_1)}}$$
where $s$ is induced by projection of $I^{n+1}$ onto $I^n$ which
forgets the last coordinate. To show that $B^{\sd^m I^{n+1}}$ is a
path space, we shall check that $s$ is an $I$-weak equivalence. We
have that $d_0s=\id$. We want to check that $sd_0$ is algebraically
homotopic to $\id$.

In the proof of Proposition~\ref{excis} we have constructed a
simplicial map
   $$\lambda:I^2\to I.$$
It induces a simplicial homotopy between $sd_0$ and $\id$
   $$\lambda^*:B^{\sd^m I^{n+1}}\to B^{\sd^m I^{n+2}}.$$
By (2) these are algebraically homotopic. We conclude that $s,d_0$
are $I$-weak equivalences, and hence so is $d_1$. The statement for
$\wt B^{\ff S^n}_m$ is verified in a similar way.
\end{proof}

The algebra $B':=(B^{\sd^m I^{n}})^{\sd^k\Delta^1}$ is another path
object of $B^{\sd^m I^{n}}$, and so there is a commutative diagram
   $$\xymatrix{B^{\sd^m I^{n}}\ar[rd]_{s'}\ar[rr]^{diag}&&B^{\sd^m I^{n}}\times B^{\sd^m I^{n}}\\
               &B',\ar[ur]_{(d'_0,d'_1)}}$$
where $s$ is an $I$-weak equivalence and $(d'_0,d'_1)$ is a
fibration. Let $X$ be the fibre product for
   $$\xymatrix{B^{\sd^m I^{n+1}}\ar[rr]^(.40){(d_0,d_1)}&&B^{\sd^m I^{n}}\times B^{\sd^m I^{n}}&&\ar[ll]_(.30){(d_0',d_1')}B'.}$$
Then $(s,s')$ induces a unique map $q:B^{\sd^m I^{n}}\to X$ such
that $pr_1\circ q=s$ and $pr_2\circ q=s'$. We can factor $q$ as
   $$B^{\sd^m I^{n}}\lra{s''}B''\lra{p}X,$$
where $s''$ is an $I$-weak equivalence and $p$ is a fibration. It
follows that $u:=pr_2\circ p$ and $v:=pr_1\circ p$ are $I$-trivial
fibrations, because $vs''=s,us''=s'$. It follows that there is a
commutative diagram
   $$\xymatrix{B^{\sd^m I^{n}}\ar[rd]_{s''}\ar[rr]^{diag}&&B^{\sd^m I^{n}}\times B^{\sd^m I^{n}}\\
               &B''\ar[ur]_{(d_0'',d_1'')}}$$
with $(d_0'',d_1''):=(d_0,d_1)\circ v=(d_0',d_1')\circ u$, and so
the algebra $B''$ is a path object of $B^{\sd^m I^{n}}$.

Now let us consider a commutative diagram
   $$\xymatrix{A'\ar[r]^{h'}\ar[d]_g&B''\ar[d]_u\ar[r]^(.35)v&B^{\sd^m I^{n+1}}\ar[d]^{(d_0,d_1)}\\
               A\ar[r]_h&B'\ar[r]_(.27){(d_0',d_1')}&B^{\sd^m I^{n}}\times B^{\sd^m I^{n}}}$$
with the left square cartesian. The desired homomorphism $H:A'\to
B^{\sd^m I^{n+1}}$ is then defined as $vh'$. The homomorphism
$H:A'\to\wt B^{\ff S^n}_m$ is constructed in a similar way.
\end{proof}

The proof of the Hauptlemma also applies to showing that for any
homomorphism $h:A\to B^{\sd^m\Delta^1\times\Delta^n}$ the induced
maps $d_0h,d_1h:A\to B^{\sd^m\Delta^n}$ are algebraically homotopic.
If $m=0$ then the homotopy is constructed in $n$ steps similar to
that described above for cubes $I^n$ (each step is obtained by
applying the polynomial homotopy $\phi_{i,j}$).

We can use the homotopy to describe explicitly a polynomial
contraction of an algebra $B^{\Delta^n}$ to $B$. Precisely, consider
the maps $s:B\to B^{\Delta^n}$, $\delta:B^{\Delta^n}\to B$ induced
by the unique map $[n]\to[0]$ and the map $[0]\to[n]$ taking 0 to
$n$. Then $\delta s=1_B$ and $s\delta$ is polynomially homotopic to
the identity map of $B^{\Delta^n}$. The homotopy is constructed by
lifting the simplicial homotopy that contracts $\Delta^n$ to its
last vertex. This simplicial homotopy is given by a simplicial map
   $$\Delta^1\times\Delta^{n}\lra{h}\Delta^n$$
that takes $(v:[m]\to[1],u:[m]\to[n])$ to $\bar u:[m]\to[n]$, where
$\bar u$ is defined as the composite
   $$[m]\xrightarrow{(u,v)}[n]\times[1]\lra{w}[n]$$
and where $w(j,0)=j$ and $w(j,1)=n$.

We have a homomorphism
   $$h^*:B^{\Delta^n}\to B^{\Delta^1\times\Delta^n}$$
which is induced by $h$. Then $d_0h^*=1$ is polynomially homotopic
to $d_1h^*=s\delta$.

If a homomorphism $f:A'\to A$ is homotopic to $g:A'\to A$ by means
of a homomorphism $h:A'\to A[x]$ then $J(f)$ is homotopic to $J(g)$.
Indeed, consider a commutative diagram of algebras
   $$\xymatrix{JA'\ar[r]\ar[d]_{J(h)}&TA'\ar[d]_{T(h)}\ar[r]&A'\ar[d]^h\\
               J(A[x])\ar[r]\ar[d]_\gamma &T(A[x])\ar[d]\ar[r]&A[x]\ar@{=}[d]\\
               (JA)[x]\ar[d]_{\partial_x^{0;1}}\ar[r]&(TA)[x]\ar[d]_{\partial_x^{0;1}}\ar[r]&A[x]\ar[d]^{\partial_x^{0;1}}\\
               JA\ar[r]&TA\ar[r]&A.}$$
Then $\gamma\circ J(h)$ yields the required homotopy between $J(f)$
and $J(g)$.

Let $A,B\in\Re$ and $n\geq 0$. Part~(i) of the Hauptlemma implies
that there is a map
   $$\pi_0(\Hom_{\inda}(J^nA,\bb B(\Omega^n)))\to[J^nA,B^{\ff S^n}]$$
which is consistent with the colimit maps
   $$\varsigma:\Hom_{\inda}(J^nA,\bb B(\Omega^n))\to\Hom_{\inda}(J^{n+1}A,\bb B(\Omega^{n+1}))$$
defined by~\eqref{adj} and $\sigma:[J^nA,B^{\ff
S^n}]\to[J^{n+1}A,B^{\ff S^{n+1}}]$ which is defined like
$\varsigma$. So we get a map
   $$\Gamma:\cc K_0(A,B)\to\lp_n[J^nA,B^{\ff S^n}].$$

\begin{comparisa}\label{comar}
The map $\Gamma:\cc K_0(A,B)\to\lp_n[J^nA,B^{\ff S^n}]$ is an
isomorphism.
\end{comparisa}

\begin{proof}
It is obvious that
   $$\pi_0(\Hom_{\inda}(J^nA,\bb B(\Omega^n)))\to[J^nA,B^{\ff S^n}]$$
is surjective for each $n\geq 0$, and hence so is $\Gamma$. Suppose
$f_0,f_1:J^nA\to B^{\ff S^n}$ are polynomially homotopic by means of
$h$. By the Hauptlemma there are a homomorphism $g:A'\to J^nA$,
which is a fibre product of an $I$-trivial fibration along $h$, and
hence $g\in\ff W_{\min}$, and a homomorphism $H:A'\to\wt B^{\ff
S^n}$ such that $d_0H=f_0g$ and $d_1H=f_1g$. Similar to the proof of
Excision Theorem~B one can construct a small admissible category of
algebras $\Re'$ such that it contains all algebras
$\{A',J^nA,B^{\sd^mI^n}\}_{m,n}$ we work with and such that $g$ is a
quasi-isomorphism of $\Re'$.

By Theorem~\ref{mainequiv} the induced map of graded abelian groups
   $$g^*:\cc K_*(\Re')(J^nA,B)\to\cc K_*(\Re')(A',B)$$
is an isomorphism. We have that $g^*$ takes $f_0,f_1\in\cc
K_n(\Re')(J^nA,B)$ to the same element in $\cc K_n(\Re')(A',B)$, and
so $f_0=f_1$. We see that $\Gamma$ is also injective, hence it is an
isomorphism.
\end{proof}

\begin{cor}\label{comarcor}
The homotopy groups of $\bb K(A,B)$ are computed as follows:
   $$\bb K_m(A,B)\cong
      \left\{
      \begin{array}{rcl}
       \lp_n[J^nA,(\Omega^mB)^{\ff S^n}],\ m\geq 0\\
       \lp_n[J^{-m+n}A,B^{\ff S^n}],\ m<0
      \end{array}
      \right.$$
\end{cor}

\begin{proof}
This follows from Corollary~\ref{excwww} and the preceding theorem.
\end{proof}

\section{Comparison Theorem B}\label{compb}

In this section $\Re$ is supposed to be $T$-closed. Let $\ff W$ be a
class of weak equivalences containing homomorphisms $A\to A[t]$,
$A\in\Re$, such that the triple $(\Re,\ff F,\ff W)$ is a Brown
category.

\begin{defs}{\rm
The {\it left derived category\/} $D^-(\Re,\ff F,\ff W)$ of $\Re$
with respect to $(\ff F,\ff W)$ is the category obtained from $\Re$
by inverting the weak equivalences. }\end{defs}

By~\cite{Gar1} the family of weak equivalences in the category $\cc
H\Re$ admits a calculus of right fractions. The left derived
category $D^-(\Re,\ff F,\ff W)$ (possibly ``large") is obtained from
$\cc H\Re$ by inverting the weak equivalences. The left derived
category $D^-(\Re,\ff F,\ff W)$ is left triangulated
(see~\cite{Gar,Gar1} for details) with $\Omega$ a loop functor on
it.

There is a general method of stabilizing $\Omega$ (see
Heller~\cite{Hel}) and producing a triangulated (possibly ``large")
category $D(\Re,\ff F,\ff W)$ from the left triangulated structure
on $D^-(\Re,\ff F,\ff W)$.

An object of $D(\Re,\ff F,\ff W)$ is a pair $(A,m)$ with $A\in
D^-(\Re,\ff F,\ff W)$ and $m\in\bb Z$. If $m,n\in\bb Z$ then we
consider the directed set $I_{m,n}=\{k\in\bb Z\mid m,n\leq k\}$. The
morphisms between $(A,m)$ and $(B,n)\in D(\Re,\ff F,\ff W)$ are
defined by
   $$D(\Re,\ff F,\ff W)[(A,m),(B,n)]:=\lp_{k\in I_{m,n}}D^-(\Re,\ff F,\ff W)(\Omega^{k-m}(A),\Omega^{k-n}(B)).$$
Morphisms of $D(\Re,\ff F,\ff W)$ are composed in the obvious
fashion. We define the {\it loop\/} automorphism on $D(\Re,\ff F,\ff
W)$ by $\Omega(A,m):=(A,m-1)$. There is a natural functor
$S:D^-(\Re,\ff F,\ff W)\to D(\Re,\ff F,\ff W)$ defined by
$A\longmapsto(A,0)$.

$D(\Re,\ff F,\ff W)$ is an additive category \cite{Gar,Gar1}. We
define a triangulation $\cc Tr(\Re,\ff F,\ff W)$ of the pair
$(D(\Re,\ff F,\ff W),\Omega)$ as follows. A sequence
   $$\Omega(A,l)\to (C,n)\to(B,m)\to(A,l)$$
belongs to $\cc Tr(\Re,\ff F,\ff W)$ if there is an even integer $k$
and a left triangle of representatives
$\Omega(\Omega^{k-l}(A))\to\Omega^{k-n}(C)\to\Omega^{k-m}(B)\to\Omega^{k-l}(A)$
in $D^-(\Re,\ff F,\ff W)$. Then the functor $S$ takes left triangles
in $D^-(\Re,\ff F,\ff W)$ to triangles in $D(\Re,\ff F,\ff W)$.
By~\cite{Gar,Gar1} $\cc Tr(\Re,\ff F,\ff W)$ is a triangulation of
$D(\Re,\ff F,\ff W)$ in the classical sense of Verdier~\cite{Ver}.

Let $\cc E$ be the class of all $\ff F$-extensions of $k$-algebras
   \begin{equation}\label{extensionE}
    (E):A\to B\to C.
   \end{equation}

\begin{defs}{\rm Following Corti\~nas--Thom~\cite{CT} a {\it ($\ff
F$-)excisive homology theory\/} on $\Re$ with values in a
triangulated category $(\cc T,\Omega)$ consists of a functor
$X:\Re\to \cc T$, together with a collection $\{\partial_E:E\in\cc
E\}$ of maps $\partial_E^X=\partial_E\in{\cc T}(\Omega X(C), X(A))$.
The maps $\partial_E$ are to satisfy the following requirements. \sn
\noindent{(1)} For all $E\in \cc E$ as above,
\[
\xymatrix{\Omega
X(C)\ar[r]^{\partial_E}&X(A)\ar[r]^{X(f)}&X(B)\ar[r]^{X(g)}& X(C)}
\]
is a distinguished triangle in $\cc T$. \sn \noindent{(2)} If
\[
\xymatrix{(E): &A\ar[r]^f\ar[d]_\alpha& B\ar[r]^g\ar[d]_\beta& C\ar[d]_\gamma\\
          (E'):&A'\ar[r]^{f'}& B'\ar[r]^{g'}& C'}
\]
is a map of extensions, then the following diagram commutes
\[
\xymatrix{{\Omega} X(C)\ar[d]_{{\Omega} X(\gamma)}\ar[r]^{\partial_E}& X(A)\ar[d]^{X(\alpha)}\\
 \Omega X(C')\ar[r]_{\partial_{E'}}& X(A).}
\]
We say that the functor $X:\Re\to\cc T$ is {\it homotopy
invariant\/} if it maps homotopic homomomorphisms to equal maps, or
equivalently, if for every $A\in\ahaw$, $X$ maps the inclusion
$A\subset A[t]$ to an isomorphism.

}\end{defs}

Denote by $\ff W_{\triangle}$ the class of homomorphisms $f$ such
that $X(f)$ is an isomorphism for any excisive, homotopy invariant
homology theory $X:\Re\to\cc T$. We shall refer to the maps from
$\ff W_{\triangle}$ as {\it stable weak equivalences}. The triple
$(\Re,\ff F,\ff W_{\triangle})$ is a Brown category. In what follows
we shall write $D^-(\Re,\ff F)$ and $D(\Re,\ff F)$ to denote
$D^-(\Re,\ff F,\ff W_{\triangle})$ and $D(\Re,\ff F,\ff
W_{\triangle})$ respectively, dropping $\ff W_{\triangle}$ from
notation.

In this section we prove the following theorem.

\begin{comparisb}\label{comparisb}
For any algebras $A,B\in\Re$ there is an isomorphism of $\bb
Z$-graded abelian groups
   $$\bb K_*(A,B)\cong D(\Re,\ff F)_*(A,B)=\bigoplus_{n\in\bb Z}D(\Re,\ff F)(A,\Omega^nB),$$
functorial both in $A$ and in $B$.
\end{comparisb}

The graded isomorphism consists of a zig-zag of isomorphisms each of
which is constructed below.

\begin{cor}\label{small}
$D(\Re,\ff F)$ is a category with small Hom-sets.
\end{cor}

\begin{defs}{\rm
Let $\Re$ be a small $T$-closed admissible category of algebras. A
homomorphism $A\to B$ in $\Re$ is said to be a {\it stable $\ff
F$-quasi-isomorphism\/} or just a {\it stable quasi-isomorphism\/}
if the map $\Omega^nA\to\Omega^nB$ is a quasi-isomorphism for some
$n\geq 0$. The class of quasi-isomorphisms will be denoted by $\ff
W_{qis}$. By~\cite{Gar} the triple $(\Re,\ff F,\ff W_{qis})$ is a
Brown category.}\end{defs}

Consider the ind-algebra $(B^{\ff S^n},\bb Z_{\geq 0})$ with each
$B^{\ff S^n}_k$, $k\in\bb Z_{\geq 0}$, being $\ker(B^{\sd^kI^n}\to
B^{\partial(\sd^kI^n)})$, that is $B^{\ff S^n}$ is the underlying
ind-algebra of 0-simplices of $\bb B(\Omega^n)$. We shall denote by
$B^{\cc S^n}$ the algebra $B^{\ff S^n}_0$. Notice that $B^{\cc
S^1}=\Omega B$. There is a sequence of maps
   $$\Hom_{\ahaw}(B,B)\bl\varsigma\to\Hom_{\ahaw}(JB,B^{\ff S^1}_k)\bl\varsigma\to\Hom_{\ahaw}(J^2B,B^{\ff S^2}_k)\bl\varsigma\to\cdots$$
One sets $1^{n,k}_B:=\varsigma^n(1_B)$.

Recall that $\ff W_{\min}$ is the least collection of weak
equivalences containing $A\to A[x]$, $A\in\Re$, such that the triple
$(\Re,\ff F,\ff W_{\min})$ is a Brown category.

\begin{lem}\label{bquasnoch}
Let $\Re$ be a $T$-closed admissible category of algebras and
$B\in\Re$. Then for any $n\geq 0$ all morphisms of the sequence
   $$B^{\ff S^n}_0\to B^{\ff S^n}_1\to B^{\ff S^n}_2\to\cdots$$
belong to $\ff W_{\min}$.
\end{lem}

\begin{proof}
Recall that the simplicial ind-algebra $P\bb B^\Delta(\Omega^n)$ is
indexed over $\bb Z_{\geq 0}$ and defined as $\ker((\bb
B^\Delta(\Omega^n))^I\lra{d_0}\bb B^\Delta(\Omega^n))$. The proof of
the Hauptsublemma shows that on the level of 0-simplices $d_0$ is an
$I$-trivial fibration. Its kernel consists of 0-simplices of $P\bb
B^\Delta(\Omega^n)$ and whose underlying sequence of algebras is
denoted by
   $$PB^{\ff S^n}_0\to PB^{\ff S^n}_1\to PB^{\ff S^n}_2\to\cdots$$
For each algebra of the sequence $PB^{\ff S^n}_k$ one has $(0\to
PB^{\ff S^n}_k)\in\ff W_{\min}$, because it is the kernel of an
$I$-trivial fibration.

The assertion is obvious for $n=0$. We have a commutative diagram of
extensions for all $n\geq 1,k\geq 0$
   $$\xymatrix{B^{\ff S^n}_k\ar[d]\ar[r]&PB^{\ff S^{n-1}}_k\ar[d]\ar[r]&B^{\ff S^{n-1}}_{k}\ar[d]\\
               B^{\ff S^n}_{k+1}\ar[r]&PB^{\ff S^{n-1}}_{k+1}\ar[r]&B^{\ff S^{n-1}}_{k+1}}$$
with the right and the middle arrows belonging to $\ff W_{\min}$ by
induction, hence so is the left one.
\end{proof}

\begin{lem}\label{ochich}
Let $\Re$ be a $T$-closed admissible category of algebras and
$B\in\Re$. Then each $1^{n,k}_B$, $n,k\geq 0$, belongs to $\ff
W_{\min}$.
\end{lem}

\begin{proof}
We fix $k$. The identity map $1_B=1^{0,k}_B$ belongs to $\ff
W_{\min}$. The map $1^{1,k}_B$ is the classifying map
$\xi_\upsilon:JB\to B^{\ff S^1}_k$, which is in $\ff W_{\min}$.
Suppose $1^{n-1,k}_B$, $n>1$, belongs to $\ff W_{\min}$. Then
$1^{n,k}_B=\xi_\upsilon J(1^{n-1,k}_B)$, where
$\xi_\upsilon:J(B^{\ff S^{n-1}}_k)\to B^{\ff S^n}_k$ is in $\ff
W_{\min}$. Since $J$ respects maps from $\ff W_{\min}$, then
$1^{n,k}_B$ is in $\ff W_{\min}$.
\end{proof}

\begin{lem}\label{stabquas}
The following conditions are equivalent for a homomorphism $f:A\to
B$ in $\Re$:

\begin{enumerate}
\item $f$ is a stable quasi-isomorphism;
\item $J^n(f):J^nA\to J^nB$ is a quasi-isomorphism for some $n\geq
0$;
\item for any $k\geq 0$ there is a $n\geq 0$ such that $f^{\ff S^n}:A^{\ff S^n}_k\to B^{\ff S^n}_k$ is a quasi-isomorphism.
\end{enumerate}
\end{lem}

\begin{proof}
$(1)\Leftrightarrow(2)$. Consider a commutative diagram of
extensions
   $$\xymatrix{JA\ar[r]\ar[d]_{\rho_A}&TA\ar[r]\ar[d]&A\ar@{=}[d]\\
               \Omega A\ar[r]&EA\ar[r]&A,}$$
where $TA,EA$ are contractible. It follows that $\rho_A$ is a
quasi-isomorphism. It is plainly functorial in $A$. Since $J$
respects quasi-isomorphisms, it follows that there is a commutative
diagram for any $n\geq 1$
   $$\xymatrix{J^nA\ar[r]\ar[d]_{J^n(f)}&\Omega^n A\ar[d]^{\Omega^n(f)}\\
               J^nB\ar[r]&\Omega^nB,}$$
in which the horizontal maps are quasi-isomorphisms. We see that
$\Omega^n(f)$ is a quasi-isomorphism if and only if $J^n(f)$ is.

$(2)\Leftrightarrow(3)$. There is a commutative diagram of
extensions for all $n\geq 1,k\geq 0$
   $$\xymatrix{J(A^{\ff S^{n-1}}_k)\ar[d]\ar[r]&T(A^{\ff S^{n-1}}_k)\ar[d]\ar[r]&A^{\ff S^{n-1}}_{k}\ar@{=}[d]\\
               A^{\ff S^n}_k\ar[r]&PA^{\ff S^{n-1}}_{k}\ar[r]&A^{\ff S^{n-1}}_{k}}$$
in which the right and the middle arrows are quasi-isomorphisms,
hence so is the left one. The middle arrow is actually
quasi-isomorphic to zero. Since $J$ respects quasi-isomorphisms, we
get a chain of quasi-isomorphisms
   $$J^nA\to J^{n-1}(A^{\ff S^1}_k)\to\cdots\to J(A^{\ff S^{n-1}}_k)\to A^{\ff S^n}_k,$$
functorial in $A$. It follows that there is a commutative diagram
for any $n\geq 1$
   $$\xymatrix{J^nA\ar[r]\ar[d]_{J^n(f)}&A^{\ff S^n}_k\ar[d]^{f^{\ff S^n}}\\
               J^nB\ar[r]&B^{\ff S^n}_k,}$$
in which the horizontal maps are quasi-isomorphisms. We see that
$f^{\ff S^n}_k$ is a quasi-isomorphism if and only if $J^n(f)$ is.
\end{proof}

We call a homomorphism $f:A\to B$ in a $T$-closed category $\Re$ a
{\it $\cc K$-equivalence\/} if the induced map $\cc K(C,A)\to\cc
K(C,B)$ is a weak equivalence of spaces.

\begin{prop}\label{stabquasnoch}
Let $\Re$ be a small $T$-closed admissible category of algebras. A
homomorphism $t:A\to B$ in $\Re$ is a stable quasi-isomorphism if
and only if it is a $\cc K$-equivalence.
\end{prop}

\begin{proof}
Suppose $t:A\to B$ is a stable quasi-isomorphism. Then $\Omega^n(t)$
is a quasi-isomorphism for some $n\geq 0$, and hence a $\cc
K$-equivalence. For any algebra $C\in\Re$ the induced map
   $$\cc K(J^nC,\Omega^nA)\to\cc K(J^nC,\Omega^nB)$$
is a weak equivalence of spaces. By Corollaries~\ref{excwww}
and~\ref{excspm} the map
   $$\Omega^n\cc K(J^nC,A)\to\Omega^n\cc K(J^nC,B)$$
is a weak equivalence, hence so is the map
   $$t_*:\cc K(C,A)\to\cc K(C,B).$$
Thus $t$ is a $\cc K$-equivalence.

Suppose now $t:A\to B$ is a $\cc K$-equivalence. Then the induced
map
   $$\cc K(B,A)\to\cc K(B,B)$$
is a homotopy equivalence of spaces. There are $k,n\geq 0$, a map
$e:J^nB\to A_k^{\ff S^n}$, and a sequence of maps
   $$J^nB\xrightarrow{e}A_k^{\ff S^n}\xrightarrow{t^{\ff S^n}}B_k^{\ff S^n}$$
such that $t^{\ff S^n}e$ is simplicially homotopic to $1^{n,k}_B$.
By the Hauptlemma $t^{\ff S^n}e$ is polynomially homotopic to
$1^{n,k}_B$. By Lemma~\ref{ochich} $1^{n,k}_B$ is a
quasi-isomorphism. It follows that $e$ is a right unit in the
category $D^-(\Re,\ff F,\ff W_{qis})$. For every $m\geq 0$ one has
   \begin{equation}\label{ewqr}
    \varsigma^m(t^{\ff S^n}e)=p\circ J^m(t^{\ff S^n})\circ J^m(e)\simeq 1^{n+m,k}_B,
   \end{equation}
where $p$ is a quasi-isomorphism. By Lemma~\ref{ochich}
$1^{n+m,k}_B$ is a quasi-isomorphism. It follows that $J^m(e)$ is  a
right unit in $D^-(\Re,\ff F,\ff W_{qis})$.

We claim that $t^{\ff S^n}$ is a $\cc K$-equivalence. By assumption
$t^{\ff S^0}=t$ is a $\cc K$-equivalence. Suppose $t^{\ff S^{n-1}}$
is a $\cc K$-equivalence for $n\geq 1$. There is a commutative
diagram of extensions
   $$\xymatrix{A^{\ff S^n}_k\ar[d]_{t^{\ff S^{n}}}\ar[r]&PA^{\ff S^{n-1}}_k\ar[d]\ar[r]&A^{\ff S^{n-1}}_{k}\ar[d]^{t^{\ff S^{n-1}}}\\
               B^{\ff S^n}_k\ar[r]&PB^{\ff S^{n-1}}_{k}\ar[r]&B^{\ff S^{n-1}}_k,}$$
in which the right and the middle arrows are $\cc K$-equivalences by
induction, hence so is the left one. The middle arrow is actually
quasi-isomorphic to zero.

We see that $t^{\ff S^{n}}e$ is a $\cc K$-equivalence. The two out
of three property implies $e$ is a $\cc K$-equivalence. Therefore
the induced map
   $$e_*:\cc K(J^nA,J^nB)\to\cc K(J^nA,A_k^{\ff S^n})$$
is a homotopy equivalence of spaces. Let
$q=e_*^{-1}(1^{n,k}_A):J^{n+m}A\to(J^{n}B)^{\ff S^m}_l$; then
$e^{\ff S^m}q$ is simplicially homotopic to
$\varsigma^m(1^{n,k}_A)$. By the Hauptlemma $e^{\ff S^m}q$ is
polynomially homotopic to $\varsigma^m(1^{n,k}_A)$. It follows from
Lemma~\ref{ochich} that $\varsigma^m(1^{n,k}_A)$ is a
quasi-isomorphism. We see that $e^{\ff S^m}$ is a left unit in
$D^-(\Re,\ff F,\ff W_{qis})$. The proof of Lemma~\ref{stabquas}
shows that $J^m(e)$ is quasi-isomorphic to $e^{\ff S^m}$. Thus
$J^m(e)$ is a left unit in $D^-(\Re,\ff F,\ff W_{qis})$.

By above $J^m(e)$ is also a right unit in $D^-(\Re,\ff F,\ff
W_{qis})$, and so is an isomorphism in $D^-(\Re,\ff F,\ff W_{qis})$.
Since the canonical functor $\Re\to D^-(\Re,\ff F,\ff W_{qis})$ has
the property that a homomorphism of algebras is a quasi-isomorphism
if and only if its image in $D^-(\Re,\ff F,\ff W_{qis})$ is an
isomorphism, we see that $J^m(e)$ is a quasi-isomorphism.

By~\eqref{ewqr} $J^m(t^{\ff S^n})$ is a quasi-isomorphism, because
so are $p$, $1^{n+m,k}_B$ and $J^m(e)$. Since $J$ preserves
quasi-isomorphisms, the proof of Lemma~\ref{stabquas} shows that
there is a commutative diagram
   $$\xymatrix{J^{n+m}A\ar[r]\ar[d]_{J^{n+m}(t)}&J^m(A^{\ff S^n}_k\ar[d]^{J^m(t^{\ff S^n})})\\
               J^{n+m}B\ar[r]&J^m(B^{\ff S^n}_k),}$$
in which the horizontal maps are quasi-isomorphisms. We see that
$J^{n+m}(t)$ is a quasi-isomorphism, because so is $J^m(t^{\ff
S^n})$. So $t$ is a stable quasi-isomorphism by Lemma~\ref{stabquas}
as required.
\end{proof}

The next result is an improvement of Theorem~\ref{mainequiv}. It
will also be useful when proving Comparison Theorem~B.

\begin{thm}\label{mainequivvv}
Suppose $\Re$ is an admissible $T$-closed category of algebras and
$u:A\to B$ is a $\cc K$-equivalence in $\Re$. Then the induced map
   $$u^*:\bb K(B,D)\to\bb K(A,D)$$
is a weak equivalence of spectra for any $D\in\Re$.
\end{thm}

\begin{proof}
Similar to the proof of Excision Theorem~B one can construct a small
admissible $T$-closed full subcategory of algebras $\Re'$ such that
it contains $A,B,D$. By assumption $u$ is a $\cc K$-equivalence in
$\Re'$, hence $J^n(u)$ is a quasi-isomorphism of $\Re'$ for some
$n\geq 0$ by the preceding proposition and Lemma~\ref{stabquas}.

By Theorem~\ref{mainequiv} the induced map of spectra
   $$(J^n(u))^*:\bb K(\Re')(J^nB,D)\to\bb K(\Re')(J^nA,D)$$
is a weak equivalence. Corollary~\ref{suspen} now implies the claim.
\end{proof}

\begin{lem}\label{gammastable}
Suppose $\Re$ is an admissible $T$-closed category of algebras. Then
every stable weak equivalence in $\Re$ is a $\cc K$-equivalence.
\end{lem}

\begin{proof}
Using Theorem~\ref{spectrumunst} for every $A\in\Re$ the map
   $$\bb K(A,-):\Re\to\Ho(Sp)$$
with $\Ho(Sp)$ the homotopy category of spectra yields an excisive,
homotopy invariant homology theory. Therefore it takes stable weak
equivalence to isomorphisms in $\Ho(Sp)$.
\end{proof}

Given an ind-algebra $(B,J)\in\Re^{\ind}$ and $A\in\Re$, we set
   $$D^-(\Re,\ff F)(A,B)=\lp_{j\in J}D^-(\Re,\ff F)(A,B_j).$$
Using the fact that $J$ respects polynomial homotopy and stable weak
equivalences, we can extend the map $\varsigma:\Hom_{\inda}(A,B^{\ff
S^n})\to\Hom_{\inda}(JA,B^{\ff S^{n+1}})$ to a functor
   $$\sigma:D^-(\Re,\ff F)(A,B^{\ff S^{n}})\to D^-(\Re,\ff F)(JA,B^{\ff S^{n+1}}).$$

The functor $\sigma$ takes a map
   $$\xymatrix{&A'\ar[dl]_s\ar[dr]^f\\
               A&&B^{\ff S^{n}}}$$
in $D^-(\Re,\ff F)(A,B^{\ff S^{n}})$, where $s\in\ff W_\triangle$,
to the map
   $$\xymatrix{&JA'\ar[dl]_{J(s)}\ar[dr]^{\varsigma(f)}\\
               JA&&B^{\ff S^{n+1}}.}$$
Since $J$ respects weak equivalences and homotopy, it follows that
$\sigma$ is well-defined.

The map $\Gamma:\cc K_0(A,B)\to\lp_n[J^nA,B^{\ff S^n}]$ is an
isomorphism by Comparison Theorem~A. There is a natural map
   $$\Gamma_1:\lp_n[J^nA,B^{\ff S^n}]\to\lp_n D^-(\Re,\ff F)(J^nA,B^{\ff S^n}).$$

\begin{lem}\label{gammaone}
$\Gamma_1$ is an isomorphism, functorial in $A$ and $B$.
\end{lem}

\begin{proof}
Suppose maps $f_0,f_1:J^nA\to B^{\ff S^n}$ are such that
$\Gamma_1(f_0)=\Gamma_1(f_1)$. Using the Hauptlemma, we may choose
$n$ big enough to find a stable weak equivalence $t:A'\to J^nA$ such
that $f_0t$ is simplicially homotopic to $f_1t$. By
Lemma~\ref{gammastable} $t$ is a $\cc K$-equivalence of $\Re$. By
Theorem~\ref{mainequivvv} the induced map of graded abelian groups
   $$t^*:\cc K_*(J^nA,B)\to\cc K_*(A',B)$$
is an isomorphism. We have that $t^*$ takes $f_0,f_1\in\cc
K_n(J^nA,B)$ to the same element in $\cc K_n(A',B)$, and so
$f_0=f_1$. We see that $\Gamma_1$ is injective.

Consider a map
   $$\xymatrix{&A'\ar[dl]_s\ar[dr]^f\\
               J^nA&&B^{\ff S^{n}}}$$
with $s\in\ff W_\triangle$. By Lemma~\ref{gammastable} $s$ is a $\cc
K$-equivalence of $\Re$. By Theorem~\ref{mainequivvv} the induced
map of abelian groups
   $$s^*:\cc K_n(J^nA,B)\to\cc K_n(A',B)$$
is an isomorphism. Then there are a $m\geq 0$, a morphism
$g:J^{n+m}A\to B^{\ff S^{n+m}}$ such that $\varsigma^m(f)$ is
simplicially homotopic to $g\circ J^m(s):J^mA'\to B^{\ff S^{n+m}}$.
By the Hauptlemma these are polynomially homotopic. It follows that
$\Gamma_1(g)=fs^{-1}$, and so $\Gamma_1$ is also surjective.
\end{proof}

\begin{lem}\label{gammatwo}
The natural map
   $$\Gamma_2:\lp_n D^-(\Re,\ff F)(J^nA,B^{\cc S^n})\to\lp_n D^-(\Re,\ff F)(J^nA,B^{\ff S^n})$$
is an isomorphism, functorial in $A$ and $B$.
\end{lem}

\begin{proof}
It follows from Lemma~\ref{bquasnoch} that
   $$D^-(\Re,\ff F)(J^nA,B^{\cc S^n})\to D^-(\Re,\ff F)(J^nA,B^{\ff S^n})$$
is bijective for all $n\geq 0$. Therefore $\Gamma_2$ is an
isomorphism.
\end{proof}

Consider a commutative diagram of algebras
   $$\xymatrix{B^{\cc S^n}\ar[r]&PB^{\cc S^{n-1}}\ar[r]&B^{\cc S^{n-1}}\ar@{=}[d]\\
               J(B^{\cc S^{n-1}})\ar[d]_{\rho^{n-1}}\ar[u]^{\xi^{n-1}}\ar[r]&T(B^{\cc S^{n-1}})\ar[d]\ar[u]\ar[r]&B^{\cc S^{n-1}}\ar@{=}[d]\\
               \Omega B^{\cc S^{n-1}}\ar[r]&E(B^{\cc S^{n-1}})\ar[r]&B^{\cc S^{n-1}}\ar@{=}[u].}$$
The middle arrows are stably weak equivalent to zero and
$\rho^{n-1},\xi^{n-1}$ are stable weak equivalences, functorial in
$B$. Since $\Omega$ respects stable weak equivalences, one obtains a
functorial zig-zag of stable weak equivalences of length $2n$
   $$B^{\cc S^n}\xleftarrow{\xi^{n-1}}J(B^{\cc S^{n-1}})\xrightarrow{\rho^{n-1}}\Omega B^{\cc S^{n-1}}
     \xleftarrow{\Omega\xi^{n-2}}\cdots \xleftarrow{\Omega^{n-1}\xi^0}\Omega^{n-1}JB\xrightarrow{\Omega^{n-1}\rho^0}\Omega^nB.$$
The zig-zag yields an isomorphism $\delta^n:B^{\cc S^n}\to\Omega^nB$
in $D^-(\Re,\ff F)$.

Let us define a map
   $$\Gamma_3:\lp_n D^-(\Re,\ff F)(J^nA,B^{\cc S^n})\to\lp_n D^-(\Re,\ff F)(J^nA,\Omega^nB)$$
by taking
   $$\xymatrix{&A'\ar[dl]_s\ar[dr]^f\\
               J^nA&&B^{\cc S^{n}}}$$
to $\delta^nfs^{-1}$. We have to verify that $\Gamma_3$ is
consistent with colimit maps, where a colimit map on the right hand
side $u_n:D^-(\Re,\ff F)(J^nA,\Omega^nB)\to D^-(\Re,\ff
F)(J^{n+1}A,\Omega^{n+1}B)$ takes
   $$\xymatrix{&A'\ar[dl]_s\ar[dr]^f\\
               J^nA&&\Omega^nB}$$
to $\rho^0_{\Omega^nB}J(f)(J(s))^{-1}$. Let $v_n:D^-(\Re,\ff
F)(J^nA,B^{\cc S^n})\to D^-(\Re,\ff F)(J^{n+1}A,B^{\cc S^{n+1}})$ be
a colimit map on the left. So we have to check that
$\Gamma_3(v_n(fs^{-1}))=u_n(\Gamma_3(fs^{-1}))$.

The map $u_n(\Gamma_3(fs^{-1}))$ is a zig-zag
   \begin{gather*}
     J^{n+1}A\xleftarrow{Js}JA'\xrightarrow{Jf}
     JB^{\cc S^n}\xleftarrow{J\xi^{n-1}}J^2(B^{\cc S^{n-1}})\xrightarrow{J\rho^{n-1}}J\Omega B^{\cc S^{n-1}}
     \xleftarrow{J\Omega\xi^{n-2}}\cdots\\ \xleftarrow{J\Omega^{n-1}\xi^0}J\Omega^{n-1}JB\xrightarrow{J\Omega^{n-1}\rho^0}J\Omega^nB
     \xrightarrow{\rho^0_{\Omega^nB}}\Omega^{n+1}B.
   \end{gather*}
The map $\Gamma_3(v_n(fs^{-1}))$ is a zig-zag
   \begin{gather*}
     J^{n+1}A\xleftarrow{Js}JA'\xrightarrow{Jf}
     JB^{\cc S^n}\xrightarrow{\xi^n}B^{\cc S^{n+1}}\xleftarrow{\xi^n}JB^{\cc S^n}
     \xrightarrow{\rho^n}\Omega B^{\cc S^n}\xleftarrow{\Omega\xi^{n-1}}\cdots\\
     \xrightarrow{\Omega^{n-1}\rho^1}\Omega^nB^{\cc S^1}\xleftarrow{\Omega^n\xi^0}\Omega^nJB
     \xrightarrow{\Omega^n\rho^0}\Omega^{n+1}B.
   \end{gather*}
We can cancel two $\xi^n$-s. One has therefore to check that the
zig-zag \scriptsize
   $$JB^{\cc S^n}\xleftarrow{J\xi^{n-1}}J^2(B^{\cc S^{n-1}})\xrightarrow{J\rho^{n-1}}J\Omega B^{\cc S^{n-1}}
     \xleftarrow{J\Omega\xi^{n-2}}\cdots\xleftarrow{J\Omega^{n-1}\xi^0}J\Omega^{n-1}JB\xrightarrow{J\Omega^{n-1}\rho^0}J\Omega^nB
     \xrightarrow{\rho^0_{\Omega^nB}}\Omega^{n+1}B$$
\normalsize equals the zig-zag
   $$JB^{\cc S^n}
     \xrightarrow{\rho^n}\Omega B^{\cc S^n}\xleftarrow{\Omega\xi^{n-1}}\cdots\\
     \xrightarrow{\Omega^{n-1}\rho^1}\Omega^nB^{\cc S^1}\xleftarrow{\Omega^n\xi^0}\Omega^nJB
     \xrightarrow{\Omega^n\rho^0}\Omega^{n+1}B.$$
For this one should use the property that if $g:A\to B$ is a
homomorphism then there is a commutative diagram
   \begin{equation}\label{diagrc}
    \xymatrix{J(A)\ar[d]_{J(g)}\ar[r]^{\rho_A}&\Omega A\ar[d]^{\Omega(g)}\ar[r]&EA\ar[r]\ar[d]&A\ar[d]^g\\
               J(B)\ar[r]^{\rho_B}&\Omega B\ar[r]&EB\ar[r]&B.}
   \end{equation}
So the desired compatibility with colimit maps determines a map of
colimits.

\begin{lem}\label{gammathree}
The map $\Gamma_3$ is an isomorphism, functorial in $A$ and $B$.
\end{lem}

\begin{proof}
This follows from the fact that all $\delta^n$-s are isomorphisms in
$D^-(\Re,\ff F)$.
\end{proof}

Consider a sequence of stable weak equivalences
 $$J^nA\xrightarrow{\rho}\Omega J^{n-1}A\xrightarrow{\Omega\rho}\Omega^2J^{n-2}A\xrightarrow{\Omega^2\rho}\cdots\xrightarrow{\Omega^{n-1}\rho}\Omega^nA,$$
which is functorial in $A$. Denote its composition by $\gamma_n$.

Let us define a map
   $$\Gamma_4:\lp_n D^-(\Re,\ff F)(J^nA,\Omega^nB)\to\lp_n D^-(\Re,\ff F)(\Omega^nA,\Omega^nB)$$
by taking
   $$\xymatrix{&A'\ar[dl]_s\ar[dr]^f\\
               J^nA&&\Omega^nB}$$
to $fs^{-1}\gamma^{-1}_n$. We have to verify that $\Gamma_4$ is
consistent with colimit maps, where a colimit map on the right hand
side $w_n:D^-(\Re,\ff F)(\Omega^nA,\Omega^nB)\to D^-(\Re,\ff
F)(\Omega^{n+1}A,\Omega^{n+1}B)$ takes
   $$\xymatrix{&A'\ar[dl]_s\ar[dr]^f\\
               \Omega^nA&&\Omega^nB}$$
to $\Omega(f)(\Omega(s))^{-1}$. So we have to check that
$\Gamma_4(u_n(fs^{-1}))=w_n(\Gamma_4(fs^{-1}))$.

The map $\Gamma_4(u_n(fs^{-1}))$ equals the zig-zag from
$\Omega^{n+1}A$ to $\Omega^{n+1}B$
   $$\Omega^{n+1}B\xleftarrow{\rho}J\Omega^nB\xleftarrow{Jf}JA'\xrightarrow{Js} J^{n+1}A\xrightarrow{\rho}\Omega J^nA\xrightarrow
     {\Omega\rho}\Omega^2J^{n-1}A\xrightarrow{\Omega^2\rho}\cdots\xrightarrow{\Omega^n\rho}\Omega^{n+1}A.$$
In turn, the map $w_n(\Gamma_4(fs^{-1}))$ equals the zig-zag from
$\Omega^{n+1}A$ to $\Omega^{n+1}B$
   $$\Omega^{n+1}B\xleftarrow{\Omega f}\Omega A'\xrightarrow{\Omega s}\Omega J^nA\xrightarrow{\Omega\rho}\Omega^2J^{n-1}A\xrightarrow
     {\Omega^2\rho}\Omega^3J^{n-2}A\xrightarrow{\Omega^3\rho}\cdots\xrightarrow{\Omega^{n}\rho}\Omega^{n+1}A.$$
The desired compatibility would be checked if we showed that the
zig-zag
   \begin{equation}\label{twotwo}
    \Omega J^nA\xleftarrow{\rho}J^{n+1}A\xleftarrow{Js}JA'\xrightarrow{Jf}
    J\Omega^nB\xrightarrow{\rho}\Omega^{n+1}B
   \end{equation}
equals the zig-zag
   $$\Omega J^nA\xleftarrow{\Omega s}\Omega A'\xrightarrow{\Omega f}\Omega^{n+1}B.$$
For this we use commutative diagram~\eqref{diagrc} to show that
$\rho_{J^nA}\circ Js=\Omega s\circ\rho_{A'}$ and
$\rho_{\Omega^nB}\circ Jf=\Omega f\circ\rho_{A'}$. We see
that~\eqref{twotwo} equals $\Omega
f\circ\rho_{A'}\circ\rho_{A'}^{-1}\circ(\Omega s)^{-1}=\Omega
f\circ(\Omega s)^{-1}$ in $D^-(\Re,\ff F)$ and the desired
compatibility follows.

\begin{lem}\label{gammafour}
The map $\Gamma_4$ is an isomorphism, functorial in $A$ and $B$.
\end{lem}

\begin{proof}
This follows from the fact that all $\gamma_n$-s are isomorphisms in
$D^-(\Re,\ff F)$.
\end{proof}

\renewcommand{\proofname}{Proof of Comparison Theorem~B}
\begin{proof}
Using Comparison Theorem~A, Lemmas~\ref{gammaone}, \ref{gammatwo},
\ref{gammathree}, \ref{gammafour}, the isomorphism of abelian groups
   $$\bb K_0(A,B)\cong D(\Re,\ff F)(A,B)$$
is defined as $\Gamma_4\Gamma_3\Gamma_2^{-1}\Gamma_1$. Using
Corollary~\ref{comarcor}, we get that
   $$\bb K_{n>0}(A,B)\cong D(\Re,\ff F)(A,\Omega^{n>0}B)$$
and
   $$\bb K_{n<0}(A,B)\cong D(\Re,\ff F)(J^{-n}A,B).$$
It remains to observe that $D(\Re,\ff F)(J^{-n}A,B)\cong D(\Re,\ff
F)(A,\Omega^{n}B)$ for all negative $n$.
\end{proof}

\renewcommand{\proofname}{Proof}

\begin{cor}\label{corc}
Let $\Re$ be $T$-closed. Then the classes of stable weak
equivalences and $\cc K$-equivalences coincide.
\end{cor}

\begin{cor}\label{corcompar}
Let $\Re'$ be a full admissible $T$-closed subcategory of an
admissible $T$-closed category of algebras $\Re$. Then the natural
functor
   $$D(\Re',\ff F)\to D(\Re,\ff F)$$
is full and faithful.
\end{cor}

\begin{proof}
This follows from Comparison Theorem B.
\end{proof}

We want to introduce the class of unstable weak equivalences on
$\Re$. Recall that $\ff W_{\min}$ is the minimal class of weak
equivalences containing the homomorphisms $A\to A[t]$, $A\in\Re$,
such that the triple $(\Re,\ff F,\ff W_{\min})$ is a Brown category.
We do not know whether the canonical functor $\Re\to D^-(\Re,\ff
F,\ff W_{\min})$ has the property that $f\in\ff W_{\min}$ if and
only if its image in $D^-(\Re,\ff F,\ff W_{\min})$ is an
isomorphism. For this reason we give the following

\begin{defs}{\rm
Let $\Re$ be $T$-closed. A homomorphism of algebras $f:A\to B$,
$A,B\in\Re$, is called an {\it unstable weak equivalence\/} if its
image in $D^-(\Re,\ff F,\ff W_{\min})$ is an isomorphism. The class
of unstable weak equivalences will be denoted by $\ff W_{unst}$.

}\end{defs}

By construction, $D^-(\Re,\ff F,\ff W_{\min})=\Re[\ff
W_{\min}^{-1}]$ and $\ff W_{\min}\subseteq\ff W_{unst}$. Using
universal properties of localization, one obtains that $D^-(\Re,\ff
F,\ff W_{\min})=\Re[\ff W_{unst}^{-1}]$. We can now apply results of
the section to prove the following

\begin{thm}\label{smalltria}
Let $\Re$ be $T$-closed. Then
   \begin{equation}\label{zvezda}
    \ff W_{\triangle}=\{f\in\Mor(\Re)\mid\Omega^n(f)\in\ff W_{unst}\textrm{ for some $n\geq 0$}\}.
   \end{equation}
\end{thm}

\begin{proof}
By minimality of $\ff W_{\min}$ the functor
   $$\bb K(A,-):\Re\to Spectra$$
takes the maps from $\ff W_{\min}$ (hence the maps from $\ff
W_{unst}$) to weak equivalences for all $A\in\Re$.
Corollary~\ref{corc} implies the right hand side of~\eqref{zvezda}
is contained in $\ff W_{\triangle}$.

The proof of Proposition~\ref{stabquasnoch} can literally be
repeated for $\ff W_{unst}$ to show that for every stable weak
equivalence $f$ there is $n\geq 0$ such that $\Omega^n(f)$ is in
$\ff W_{unst}$.
\end{proof}

We can now make a table showing similarity of spaces and non-unital
algebras. It is a sort of a dictionary for both categories.
Precisely, one has:

\vskip12pt
   \begin{center}
   \begin{tabular}{|r|l|}
    \hline
    {\bf Spaces} Top & {\bf Algebras} $\Re$\\
    \hline
    Fibrations & Fibrations $\ff F$\\
    Loop spaces $\Omega X$ & Algebras $\Omega A=(x^2-x)A[x]$\\
    Homotopies $X\to Y^I$ & Polynomial homotopies $A\to
    B[x]$\\
    Unstable weak equivalences & Unstable weak equivalences $\ff W_{unst}$\\
    Unstable homotopy category & The category $D^-(\Re,\ff F,\ff
    W_{unst})$\\
    Stable weak equivalences & Stable weak equivalences $\ff W_{\triangle}$\\
    Stable homotopy category of spectra & The category $D(\Re,\ff F)$\\
    \hline
   \end{tabular}
   \end{center}
\vskip12pt

To conclude the section, we should mention that Comparison Theorem~B
implies representability of the Hom-set $D(\Re,\ff F)(A,B)$,
$A,B\in\Re$, by the spectrum $\bb K(A,B)$. By~\cite{Gar1} the
natural functor $j:\Re\to D(\Re,\ff F)$ is the universal excisive,
homotopy invariant homology theory in the sense that any other such
a theory $X:\Re\to\cc T$ uniquely factors through $j$.

\section{Morita stable and stable bivariant
$K$-theories}\label{moritast}

In this section we introduce matrices into the game. We start with
preparations.

If $A$ is an algebra and $n \leq m$ are positive integers, then
there is a natural inclusion $\iota_{n,m}: M_nA\to M_mA$ of rings,
sending $M_nA$ into the upper left corner of $M_mA$. We write
$M_\infty A=\cup_n M_nA$. Let $\Gamma A$, $A\in\ahaw$, be the
algebra of $\bb N\times\bb N$-matrices which satisfy the following
two properties.
\begin{itemize}
\item[(i)] The set $\{a_{ij}\mid i,j \in \bb N\}$ is finite.
\item[(ii)] There exists a natural number $N \in \bb N$ such that each row and each column has at most
$N$ nonzero entries.
\end{itemize}
$M_\infty A\subset\Gamma A$ is an ideal. We put
   $$\Sigma A=\Gamma A/M_\infty A.$$
We note that $\Gamma A$, $\Sigma A$ are the cone and suspension
rings of $A$ considered by Karoubi and Villamayor in \cite[p.
269]{KV}, where a different but equivalent definition is given.
By~\cite{CT} there are natural ring isomorphisms
   $$\Gamma A\cong\Gamma k\otimes A,\quad \Sigma A\cong\Sigma k\otimes A.$$
We call the short exact sequence
   $$M_\infty A\rightarrowtail\Gamma A\twoheadrightarrow\Sigma A$$
the {\it cone extension}. By~\cite{CT} $\Gamma
A\twoheadrightarrow\Sigma A\in\ff F_{\spl}$.

Throughout this section we assume that $\Re$ is a $T$-closed
admissible category of $k$-algebras with $k,M_nA,\Gamma A\in\Re$,
$n\geq 1$, for all $A\in\Re$. Then $M_\infty A,\Sigma A\in\Re$ for
any $A\in\Re$ and $M_\infty(f)\in\ff F$ for any $f\in\ff F$. Note
that $M_\infty A\cong A\otimes M_\infty(k)\in\Re$ for any $A\in\Re$.
It follows from Proposition~\ref{power=tensor} that for any finite
simplicial set $L$, there are natural isomorphisms
   $$M_\infty A\otimes k^L\cong (M_\infty A)^L\cong A\otimes (M_\infty k)^L.$$

Given an algebra $A$, one has a natural homomorphism $\iota:A\to
M_\infty(k)\otimes A\cong M_\infty(A)$ and an infinite  sequence of
maps
   $$A\lra{\iota}M_\infty(k)\otimes A\lra{\iota}M_\infty(k)\otimes M_\infty(k)\otimes A\lra{}\cdots\lra{} M_\infty^{\otimes n}(k)\otimes A\lra{}\cdots$$

\begin{defs}{\rm
(1) The {\it stable algebraic Kasparov $K$-theory space\/} of two
algebras $A,B\in\Re$ is the space
   $$\cc K^{st}(A,B)=\lp_{n}\cc K(A,M_\infty k^{\otimes n}\otimes B).$$
Its homotopy groups will be denoted by $\cc K^{st}_n(A,B)$, $n\geq
0$.

(2) The {\it Morita stable algebraic Kasparov $K$-theory space\/} of
two algebras $A,B\in\Re$ is the space
   $$\cc K^{mor}(A,B)=\lp(\cc K(A,B)\to\cc K(A,M_2k\otimes B)\to\cc K(A,M_3k\otimes B)\to\cdots).$$
Its homotopy groups will be denoted by $\cc K^{mor}_n(A,B)$, $n\geq
0$.

(3) A functor $X:\Re\to\bb S/(Spectra)$ is {\it
$M_\infty$-invariant\/} (respectively {\it Morita invariant\/}) if
$X(A)\to X(M_\infty A)$ (respectively each $X(A)\to X(M_nA),n>0$) is
a weak equivalence.

(4) An excisive, homotopy invariant homology theory $X:\Re\to\cc T$
is {\it $M_\infty$-invariant\/} (respectively {\it Morita
invariant\/}) if $X(A)\to X(M_\infty A)$ (respectively each $X(A)\to
X(M_nA)$, $n>0$) is an isomorphism.

}\end{defs}

\begin{lem}\label{morst}
The functor $\cc K^{st}(A,-)$ (respectively $\cc K^{mor}(A,-)$) is
$M_\infty$-in\-va\-riant (respectively Morita invariant) for all
$A\in\Re$.
\end{lem}

\begin{proof}
Straightforward.
\end{proof}

\begin{thm}[Excision]\label{excsionast}
For any algebra $A\in\Re$ and any $\ff F$-extension in $\Re$
   $$F\lra{i}B\lra{f}C$$
the induced sequences of spaces
   $$\mathcal{K}^{\star}(A,F)\lra{}\mathcal{K}^{\star}(A,B)\lra{}\mathcal{K}^{\star}(A,C)$$
and
   $$\mathcal{K}^{\star}(C,A)\lra{}\mathcal{K}^{\star}(B,A)\lra{}\mathcal{K}^{\star}(F,A)$$
are homotopy fibre sequences, where $\star\in\{st,mor\}$.
\end{thm}

\begin{proof}
This follows from Excision Theorems A, B and some elementary
properties of simplicial sets.
\end{proof}

\begin{defs}{\rm
(1) Given two $k$-algebras $A,B\in\Re$ and $\star\in\{st,mor\}$, the
sequence of spaces
   $$\mathcal{K}^{\star}(A,B),\mathcal{K}^{\star}(JA,B),\mathcal{K}^{\star}(J^2A,B),\ldots$$
together with isomorphisms
$\mathcal{K}^{\star}(J^nA,B)\cong\Omega\mathcal{K}^{\star}(J^{n+1}A,B)$
constructed in Theorem~\ref{loop} forms an $\Omega$-spectrum which
we also denote by $\mathbb{K}^{\star}(A,B)$. Its homotopy groups
will be denoted by $\mathbb{K}_n^{\star}(A,B)$, $n\in\bb Z$. Observe
that $\mathbb{K}_n^{\star}(A,B)\cong\mathcal{K}_n^{\star}(A,B)$ for
any $n\geq 0$ and
$\mathbb{K}_n^{\star}(A,B)\cong\mathcal{K}_0^{\star}(J^{-n}A,B)$ for
any $n<0$.

(2) The {\it stable algebraic Kasparov $K$-theory spectrum\/} of
$(A,B)$ (respectively {\it Morita stable algebraic Kasparov
$K$-theory spectrum}) is the $\Omega$-spectrum
$\mathbb{K}^{st}(A,B)$ (respectively $\mathbb{K}^{mor}(A,B)$).

}\end{defs}

\begin{thm}\label{spectrumst}
Let $\star\in\{st,mor\}$. The assignment
$B\mapsto\mathbb{K}^{\star}(A,B)$ determines a functor
   $$\mathbb{K}^{\star}(A,?):\Re\to(Spectra)$$
which is homotopy invariant and excisive in the sense that for every
$\ff F$-extension $F\to B\to C$ the sequence
   $$\mathbb{K}^{\star}(A,F)\to\mathbb{K}^{\star}(A,B)\to\mathbb{K}^{\star}(A,C)$$
is a homotopy fibration of spectra. In particular, there is a long
exact sequence of abelian groups
   $$\cdots\to\mathbb{K}_{i+1}^{\star}(A,C)\to\mathbb{K}_i^{\star}(A,F)\to\mathbb{K}_i^{\star}(A,B)
     \to\mathbb{K}_i^{\star}(A,C)\to\cdots$$
for any $i\in\bb Z$.
\end{thm}

\begin{proof}
This follows from Theorem~\ref{excsionast}.
\end{proof}

We also have the following

\begin{thm}\label{spectrumunstbst}
Let $\star\in\{st,mor\}$. The assignment
$B\mapsto\mathbb{K}^{\star}(B,D)$ determines a functor
   $$\mathbb{K}^{\star}(?,D):\Re^{\op}\to(Spectra),$$
which is excisive in the sense that for every $\ff F$-extension
$F\to B\to C$ the sequence
   $$\mathbb{K}^{\star}(C,D)\to\mathbb{K}^{\star}(B,D)\to\mathbb{K}^{\star}(F,D)$$
is a homotopy fibration of spectra. In particular, there is a long
exact sequence of abelian groups
   $$\cdots\to\mathbb{K}_{i+1}^{\star}(F,D)\to\mathbb{K}_i^{\star}(C,D)\to\mathbb{K}_i^{\star}(B,D)
     \to\mathbb{K}_i^{\star}(F,D)\to\cdots$$
for any $i\in\bb Z$.
\end{thm}

\begin{proof}
This follows from Theorem~\ref{excsionast}.
\end{proof}

\begin{thm}[Comparison]\label{comarst}
There are natural isomorphisms
   $$\cc K_0^{st}(A,B)\to\lp_{m,n}[J^nA,M_\infty(k)^{\otimes m}\otimes B^{\ff S^n}]$$
and
   $$\cc K_0^{mor}(A,B)\to\lp_{m,n}[J^nA,M_m(k)\otimes B^{\ff S^n}],$$
functorial in $A$ and $B$.
\end{thm}

\begin{proof}
This follows from Comparison Theorem~A.
\end{proof}

\begin{cor}\label{comarcorrr}
$(1)$ The homotopy groups of $\bb K^{st}(A,B)$ are computed as
follows:
   $$\bb K_i^{st}(A,B)\cong
      \left\{
      \begin{array}{rcl}
       \lp_{m,n}[J^nA,(\Omega^iM_\infty(k)^{\otimes m}\otimes B)^{\ff S^n}],\ i\geq 0\\
       \lp_{m,n}[J^{-i+n}A,M_\infty(k)^{\otimes m}\otimes B^{\ff S^n}],\ i<0
      \end{array}
      \right.$$
$(2)$ The homotopy groups of $\bb K^{mor}(A,B)$ are computed as
follows:
   $$\bb K_i^{mor}(A,B)\cong
      \left\{
      \begin{array}{rcl}
       \lp_{m,n}[J^nA,(\Omega^iM_m(B))^{\ff S^n}],\ i\geq 0\\
       \lp_{m,n}[J^{-i+n}A,M_m(B)^{\ff S^n}],\ i<0
      \end{array}
      \right.$$
\end{cor}

\begin{proof}
This follows from Corollary~\ref{excwww} and the preceding theorem.
\end{proof}

We denote by $D^-_{st}(\Re,\ff F)$ the category whose objects are
those of $\Re$ and whose maps between $A,B\in\Re$ are defined as
   $$\lp_{n}D^-(\Re,\ff F)(A,M_\infty(k)^{\otimes n}(B)).$$
Similarly, denote by $D^-_{mor}(\Re,\ff F)$ the category whose
objects are those of $\Re$ and whose maps between $A,B\in\Re$ are
defined as
   $$\lp_n D^-(\Re,\ff F)(A,M_n(B)).$$
It follows from~\cite{Gar1} that $D^-_{st}(\Re,\ff F)$ and
$D^-_{mor}(\Re,\ff F)$ are naturally left triangulated. Similar to
the definition of $D(\Re,\ff F)$ we can stabilize the loop
endofunctor $\Omega$ to get new categories $D_{mor}(\Re,\ff F)$ and
$D_{st}(\Re,\ff F)$ which are in fact triangulated.

\begin{thm}[\cite{Gar1}]\label{swsweq}
The functor $\Re\to D_{st}(\Re,\ff F)$ (respectively $\Re\to
D_{mor}(\Re,\ff F)$) is the universal $\ff F$-excisive, homotopy
invariant, $M_\infty$-invariant (respectively Morita invariant)
homology theory on $\Re$.
\end{thm}

The next result says that the Hom-sets $D_{st}(\Re,\ff F)(A,B)$
($D_{mor}(\Re,\ff F)(A,B)$), $A,B\in\Re$, can be represented as
homotopy groups of spectra $\bb K^{st}(\Re,\ff F)(A,B)$ ($\bb
K^{mor}(A,B)$).

\begin{thm}[Comparison]\label{comparisbst}
Let $\star\in\{st,mor\}$. Then for any algebras $A,B\in\Re$ there is
an isomorphism of $\bb Z$-graded abelian groups
   $$\bb K^{\star}_*(A,B)\cong D_{\star}(\Re,\ff F)_*(A,B)=\bigoplus_{n\in\bb Z}D_{\star}(\Re,\ff F)(A,\Omega^nB),$$
functorial both in $A$ and in $B$.
\end{thm}

\begin{proof}
This follows from Comparison Theorem B.
\end{proof}

\begin{thm}[Corti\~{n}as--Thom]\label{ctthm}
There is a natural isomorphism of $\bb Z$-graded abelian groups
   $$D_{st}(\Re,\ff F)_*(k,A)\cong KH_*(A),$$
where $KH_*(A)$ is the $\bb Z$-graded abelian group consisting of
the homotopy $K$-theory groups in the sense of Weibel~\cite{W1}.
\end{thm}

\begin{proof}
See~\cite{Gar1}.
\end{proof}

We end up the paper by proving the following

\begin{thm}\label{kh}
For any $A\in\Re$ there is a natural isomorphism of $\bb Z$-graded
abelian groups
   $$\bb K^{st}_*(k,A)\cong KH_*(A).$$
\end{thm}

\begin{proof}
This follows from Theorems~\ref{comparisbst} and~\ref{ctthm}.
\end{proof}

The preceding theorem is an analog of the same result of $KK$-theory
saying that there is a natural isomorphism $KK_*(\bb C,A)\cong K(A)$
for any $C^*$-algebra $A$.

\end{document}